\documentclass[sn-mathphys-num]{sn-jnl}

\usepackage{graphicx}%
\usepackage{multirow}%
\usepackage{amsmath,amssymb,amsfonts}%
\usepackage{amsthm}%
\usepackage{bm}
\usepackage{mathrsfs}%
\usepackage[title]{appendix}%
\usepackage{xcolor}%
\usepackage{textcomp}%
\usepackage{manyfoot}%
\usepackage{booktabs}%
\usepackage{algorithm}%
\usepackage{algorithmicx}%
\usepackage{algpseudocode}%
\usepackage{listings}%
\usepackage{subfigure}
\usepackage{longtable}
\usepackage{verbatim}
\numberwithin{equation}{section}
\newtheorem{theorem}{Theorem}[section]
\newtheorem{lemma}[theorem]{Lemma}

\theoremstyle{definition}

\newtheorem{example}[theorem]{Example}

\newtheorem{prop}[theorem]{Proposition}
\theoremstyle{remark}
\newtheorem{remark}[theorem]{Remark}
\newtheorem{assumption}{Assumption}
\newtheorem{corollary}[theorem]{Corollary}

\begin{document}

\title[Stochastic Asymptotical Regularization]{Stochastic asymptotical regularization for nonlinear ill-posed problems}


\author[1]{\fnm{Haie} \sur{Long}}\email{haie\_long@smbu.edu.cn}

\author*[1, 2]{\fnm{Ye} \sur{Zhang}}\email{ye.zhang@smbu.edu.cn}

\affil[1]{\orgdiv{MSU-BIT-SMBU Joint Research Center of Applied Mathematics}, \orgname{Shenzhen MSU-BIT University}, \orgaddress{\city{Shenzhen}, \postcode{518172}, \country{P.R. China}}}

\affil[2]{\orgdiv{School of Mathematics and Statistics}, \orgname{Beijing Institute of Technology}, \orgaddress{\city{Beijing}, \postcode{100081}, \country{P.R. China}}}


\abstract{Recently, the stochastic asymptotical regularization (SAR) has been developed in (\emph{Inverse Problems}, 39: 015007, 2023) for the uncertainty quantification of the stable approximate solution of linear ill-posed inverse problems. In this paper, we extend the regularization theory of SAR for nonlinear inverse problems. By combining techniques from classical regularization theory and stochastic analysis, we prove the regularizing properties of SAR with regard to mean-square convergence. The convergence rate results under the canonical sourcewise condition are also studied. Several numerical examples are used to show the accuracy and advantages of SAR: compared with the conventional deterministic regularization approaches for deterministic inverse problems, SAR can quantify the uncertainty in error estimates for ill-posed problems, improve accuracy by selecting the optimal path, escape local minima for nonlinear problems, and identify multiple solutions by clustering samples of obtained approximate solutions.}




\maketitle

\section{Introduction}\label{section-1-introduction}
This work is concerned with the stable solution and uncertainty quantification of the following deterministic model of general nonlinear inverse problems:
\begin{equation}\label{main}
F(x) = y
\end{equation}
where the mapping $F:\mathcal{D}(F)\subset \mathcal{X}\rightarrow \mathcal{Y}$ is a nonlinear operator between separable Hilbert spaces $\mathcal{X}$ and $\mathcal{Y}$. For simplicity, we use $\|\cdot\|$ to denote the norms for all $\mathcal{X}$, $\mathcal{Y}$, and $\mathcal{L}(\mathcal{X}\to\mathcal{Y})$. In general, equation (\ref{main}) is ill-posed in the sense that a solution may not exist, and even if it does exist may be nonunique and/or unstable with respect to the perturbation of the right-hand side in \eqref{main}. For a more rigorous definition and a detailed discussion of ill-posedness for nonlinear inverse problems, we refer to \cite[Definition 1]{Hofmann2018}. Hence, if only noisy data $y^\delta$, satisfying a deterministic noise model ($\delta_0>0$ is a fixed number)
\begin{equation}
\label{noiseLevel}
\left\|y^\delta-y\right\|\leq\delta, \quad \delta\in(0,\delta_0],
\end{equation}
is provided, regularization methods should be designed to obtain a meaningful approximate solution of \eqref{main}. Many effective regularization techniques have been proposed over the past few decades (see, e.g., \cite{TikhonovYagola1995,Hanke1996,Cheng2011Regularization,BK2008}). Of all the existing methods, the Landweber iteration is the most prominent iterative regularization technique, and is based on the fixed-point iteration
\begin{equation}\label{Landweber}
x^\delta_{k+1} = x^\delta_k + \Delta t F'(x_k^\delta)^*(y^\delta-F(x^\delta_k)).
\end{equation}
The continuous analog to \eqref{Landweber} as $\Delta t$ tends to zero is known as asymptotical regularization or Showalter's method (see, e.g., \cite{UT1994}). It is written as a first-order evolution equation of the form
\begin{equation}\label{CLand}
\dot{x}^\delta(t)=F'{\left( {{x^\delta }\left( t \right)} \right)^ * }\left[ {{y^\delta } - F\left( {{x^\delta }\left( t \right)} \right)} \right],
\end{equation}
where an artificial scalar time $t$ is introduced. As shown in \cite{Rieder-2005,ZhaoMatheLu2020}, by using Runge-Kutta integrators, the regularization properties of \eqref{CLand} are carried over to its numerical realization.

This continuous framework has also received much attention in the development of modern regularization theory of inverse problems. For instance, the authors in \cite{ZhangH2020,GonHofZhan2020,ZhangHof2019,Botetal18} extended (\ref{CLand}) to second-order and fractional-order gradient flows, which exhibit acceleration effect. The authors in \cite{Zhongwang2022} studied an asymptotical regularization with convex constraints for nonlinear inverse problems with special features of solutions, such as sparsity and discontinuity. Furthermore, by connecting to some classical methods (i.e. the Kalman-Bucy filter and 3DVAR) from data assimilation, the authors in \cite{LuNiuWerner2021} studied stationary and non-stationary asymptotical regularization methods for linear inverse problems (\ref{main}) in the presence of white noise. 

Moving in a different direction, for about half a century, an additional random term has been proposed for (\ref{Landweber}) in order to achieve global convergence for non-convex optimization problem $\min \left\|F(x)-y\right\|^2$. It results in a stochastic gradient descent algorithm
\begin{equation}
\label{LandweberS}
x^\delta_{k+1} = x^\delta_k + \Delta t F'(x_k^\delta)^*(y^\delta-F(x^\delta_k)) + f_k \Delta B_k, \quad x^\delta_{0} = {\bar x},
\end{equation}
where the initial point ${\bar x}\in \mathcal{X}$ is non-random, $\{ \Delta B_k \}$ are random variables following a certain distribution, and the standard deviation $f_k$ controls the strength of the noise. In the literature of numerical optimization, see e.g. \cite{Chiang1987, Kushner1987, Gelfand1991, Engquist2022}, the sequence $\{ f_k \}$ should be vanishing even for convergence applied to convex optimization. Furthermore, if $\{ f_k \}$ decreases too fast, the convergence will be to a local minimum and not the global one. Note that the choice of $\{ f_k \}$ is also crucial for our introduced new method (\ref{initial value problem}) later. In recent years, the continuous counterpart of \eqref{Landweber}, namely
\begin{equation}\label{initial value problem}
\textmd{d}{x^\delta }\left( t \right) = F'{\left( {{x^\delta }\left( t \right)} \right)^ *} \left[ {{y^\delta } - F\left( {{x^\delta }\left( t \right)} \right)} \right]\textmd{d}t + f\left( t \right)\textmd{d}{B_t}, \quad {x^\delta }\left( 0 \right) = {\bar x},
\end{equation}
where the auxiliary function $f(t)\in L^\infty({\mathbb R}_+)$ and $B_t$ is an ${\mathcal X}$-valued $Q$-Wiener process, has also gained much attention in the field of optimization and machine learning, see \cite{LiTaiE2019,Francis2023} and references therein. 

Recently, motivated by the uncertainty quantification for infinite dimensional inverse problems, the Stochastic Asymptotical Regularization (SAR) method, i.e. the linear counterpart of \eqref{initial value problem}, had been introduced in \cite{ZhangChen23,ZhangChen24} for linear inverse problems. It has been shown that compared with conventional deterministic regularization approaches for deterministic inverse problems such as \eqref{main}-\eqref{noiseLevel}, SAR exhibits four new capabilities: (i) it can quantify the uncertainty in error estimates for inverse problems (e.g. the \textit{a posteriori} (pointwise) error estimation, etc.), see \cite[Section 2.2]{ZhangChen23} for a rigorous mathematical description. (ii) It can reveal and explicate hidden information for many real-world problems that are usually obscured by incomplete mathematical modeling and the presence of large-scale noise, see \cite[Section 5.2]{ZhangChen23} for an example in biosensor tomography. (iii) It can improve accuracy by selecting the optimal path \cite[Section 4.1]{ZhangChen24}, and (iv) it can identify multi-solutions by clustering samples of obtained approximate solutions \cite[Section 4.3]{ZhangChen24}. Moreover, in comparison with the Bayesian methods, SAR requires neither noise structure nor \textit{a priori} probability distribution of the exact solution. Building on this idea, in this paper we aim to extend SAR for non-linear inverse problems \eqref{main}-\eqref{CLand}. In this paper, we will prove that the stochastic differential equation \eqref{initial value problem}, associated with an appropriate \textit{a posteriori} stopping rule, yields a regularization method, which we refer to \emph{Stochastic Asymptotical Regularization} (SAR).

The key idea for proving the regularizing properties of SAR is to carefully design auxiliary function $f(t)$ (see \eqref{ft} below) through the lens of regularization theory. To be more precise, our analysis relies on the bias-variance decomposition; i.e. the regularization error of SAR can be decomposed into two components: the bias term and the variance term. The former is deterministic and thus can be analyzed in a manner similar to the Landweber method \cite{Hanke1995}, while the latter is derived from the $\mathcal{X}$-valued $\mathcal{Q}$-Wiener process, which consists of an infinite-dimensional stochastic integral and constitutes the main technical challenge in the analysis.

Note that SAR is different from the \emph{Stochastic Gradient Descent} (SGD) method, which is another popular numerical method in today's literature of statistics, machine learning, and computational mathematics. SGD was pioneered by Robbins and Monro in statistical inference \cite{Rob1951} (see monograph \cite{Kus2003} for the asymptotic-convergence results). Algorithmically, SGD is a randomized version of the classical Landweber method \cite{Land1951}, which employs an unbiased estimator of the full gradient computed from one single randomly selected data point at each iteration. The regularizing properties of SGD for linear inverse problems with an \emph{a priori} and an \emph{a posteriori} stopping rule can be found in \cite{BJin2019} and \cite{BJin2020IP}, respectively. An extended error estimation of SGD, taking into account the discretization levels, the decay of the step-size, and the general source conditions, can be found in \cite{LuMathe2021}. SGD for nonlinear ill-posed problems was investigated in \cite{BJin2020SIAM}. The major difference between SGD and SAR is that the randomness of SGD arises from the random row index $i_k$ and changes at each iteration, while the randomness of SAR comes from the second term $f\left( t \right)\textmd{d}{B_t}$ on the right side of \eqref{initial value problem}, which is more flexible, and can be proposed differently according to different purposes in the application. Besides above mentioned four advantages of SAR, numerical experiments in this work indicate a new capability of SAR: for some nonlinear inverse problems, it can avoid the local minimum efficiently. 

The remainder of this paper is organized as follows: Section \ref{AN} presents our standing assumptions and the relevant discussion, along with a brief discussion of the necessary techniques from stochastic calculus. In Sections \ref{sec-convergence} and \ref{sec-convergencerate}, we perform the convergence analysis of SAR for the regularization property and convergence rates, respectively.
Section \ref{simulation} is devoted to the numerical implementation of SAR, with some numerical results displayed to demonstrate its performance. Finally, Section \ref{conclusion} concludes the paper, and detailed proofs for specific assertions, along with additional numerical illustrations and explanations, are provided in the appendices.

\section{Assumptions and necessary concepts}\label{AN}

\subsection{Assumptions}
For the analysis of the SAR method of the form \eqref{initial value problem}, we require a few suitable conditions, which are quite similar to the assumptions needed for the analysis of asymptotical regularization method \cite{UT1994}.
\begin{assumption} \label{assumption-1}
	Let $B_r(\bar x)$ denote the closed ball around $\bar x$ with radius $r$; then, the following conditions hold true:
	\begin{itemize}
		\item[(i)] The operator $F:\mathcal{X}\rightarrow \mathcal{Y}$ is continuous, with a continuous and uniformly bounded Fr\'{e}chet derivative, i.e.
		\begin{equation}\label{Fbounded}
		\left\|F'(x)\right\| \le 1,~~\forall x\in B_r(\bar x).
		\end{equation}
		\item[(ii)] There exists an $\eta \in (0,1)$ such that
		\begin{equation}\label{TCC}
		\left\| {F\left( {\tilde x} \right) - F\left( x \right) - F'\left( x \right)\left( {\tilde x - x} \right)} \right\| \le \eta \left\| {F\left( x \right) - F\left( {\tilde x} \right)} \right\|
		\end{equation}
		holds for all $x,\tilde x \in {B_r}\left( {\bar x} \right) \subset D\left( F \right)$.
	\end{itemize}
\end{assumption}
Assumption \ref{assumption-1}(i) imposes boundedness and continuity on the derivative $F'(x)$.
The inequality \eqref{TCC} is often known as the \emph{tangential cone condition}, and it controls the degree of nonlinearity of the operator $F$. Roughly speaking, it requires that the forward map $F$ be not too far from a linear map.
Moreover, since the solution to problem \eqref{initial value problem} may be nonunique, in this paper the reference solution $x^\dagger$ is taken to be the minimum norm solution (with
respect to the initial guess $\bar x$), which is known to be unique under Assumption \ref{assumption-1}(ii) (see, e.g., \cite{Hanke1995}).

\begin{prop}\label{prop2}
	Under Assumption \ref{assumption-1}, the following statements hold true:
	\begin{enumerate}
		\item For all $x, \tilde x \in {B_r}\left( {\bar x} \right)$,
		\begin{equation}\label{F2}
		\frac{1}{1+\eta}\left\| {F'\left( x \right)\left( {x-\tilde x} \right)} \right\| \le \left\| {F\left( x \right) - F\left( {\tilde x} \right)} \right\|\le\frac{1}{1-\eta}\left\| {F'\left( x \right)\left( {x-\tilde x} \right)} \right\|.
		\end{equation}
		\item If $x^*$ is a solution to problem \eqref{main}, any other solution $\tilde x^*$ satisfies $x^*-\tilde x^* \in \mathcal{N}(F'(x^*))$, and \textit{vice versa}.
		Here, $\mathcal{N}(\cdot)$ denotes the null space of an operator.
	\end{enumerate}
\end{prop}

As with \cite{Hanke1996,BK2008,Jin2015}, both nonlinearity and source conditions are often needed to derive convergence rates, and are presented below.

\begin{assumption}\label{assumption-2}
	Assume that there exist an element $\nu  \in \mathcal{X}$ and constants $\gamma  \in \left( {0,1/2} \right]$ and $E\geq0$ such that
	\begin{equation}\label{A2}
	\bar x - {x^\dag } = {\left( {F'{{\left( {{x^\dag }} \right)}^*}F'\left( {{x^\dag }} \right)} \right)^\gamma }\nu ,~~\left\| \nu  \right\| \le E.
	\end{equation}
	Also assume that, for all $x \in B_r(\bar x)$, there exists a linear bounded operator $R_x:\mathcal{Y} \rightarrow \mathcal{Y}$ and a constant $c_R\geq0$ such that
	\begin{equation}\label{A31}
	F'\left( x \right) = {R_x}F'\left( {{x^\dag }} \right)
	\end{equation}
	and
	\begin{equation}\label{A32}
	\left\| {{R_x} - I} \right\| \le c_R\left\| {x - {x^\dag }} \right\|.
	\end{equation}
\end{assumption}

The fractional power ${({F'(x^\dagger )}^* F'(x^\dagger ))}^\gamma$  in \eqref{A2} is defined by spectral decomposition, which usually represents a certain smoothness condition on the exact solution $x^\dagger$. The restriction $\gamma \in (0,1/2]$ is there for technical reasons.
Note that in the linear case $R_x \equiv I$; therefore, \eqref{A31} may be interpreted as a further restriction of the ``nonlinearity'' of $F$. In particular, \eqref{A31} implies that
\begin{equation*}
\mathcal{N}\left( {F'\left( {{x^\dag }} \right)} \right) \subset \mathcal{N}\left( {F'\left( x \right)} \right),~~x \in {B_r}\left( {\bar x} \right).
\end{equation*}
Note again that \eqref{A31} and \eqref{A32} imply \eqref{TCC} with $\tilde x =x^\dag$ for $r$ sufficiently small, since, for  $x \in B_r(\bar x)$, the following holds:
\begin{equation}\label{rem1}
\begin{split}
\left\| {F\left( x \right) - F\left( {{x^\dag }} \right) - F'\left( x^\dag \right)\left( {x - {x^\dag }} \right)} \right\| &\\
&\!\!\!\!\!\!\!\!\!\!\!\!\!\!\!\!\!\!\!\!\!\!\!\!\!\!\!\!\!\!\!\!\!\!\!\!\!\!\!\!\!\!\!\!\!\!\!\!\!\!\!\!\!= \left\| {\int_0^1 {\left( {F'\left( {{z_\lambda }} \right) - F'\left( x^\dag \right)} \right)} \left( {x - {x^\dag }} \right)\textmd{d}\lambda } \right\|\\
&\!\!\!\!\!\!\!\!\!\!\!\!\!\!\!\!\!\!\!\!\!\!\!\!\!\!\!\!\!\!\!\!\!\!\!\!\!\!\!\!\!\!\!\!\!\!\!\!\!\!\!\!\!\le \int_0^1 {\left\| {\left( {{R_{{z_\lambda }}} - I} \right)F'\left( {{x^\dag }} \right)\left( {x - {x^\dag }} \right)} \right\|} \textmd{d}\lambda \\
&\!\!\!\!\!\!\!\!\!\!\!\!\!\!\!\!\!\!\!\!\!\!\!\!\!\!\!\!\!\!\!\!\!\!\!\!\!\!\!\!\!\!\!\!\!\!\!\!\!\!\!\!\!\le \frac{c_R}{2}\left\| {F'\left( {{x^\dag }} \right)\left( {x - {x^\dag }} \right)} \right\|\left\| {x - {x^\dag }} \right\|
\end{split}
\end{equation}
with ${z_\lambda }: = \lambda x + \left( {1 - \lambda } \right){x^\dag }$. Moreover, if $x$ is a random variable, from the fact that $(E[x])^2\leq E[x^2]$, we obtain the stochastic version of \eqref{rem1} as follows:
\begin{equation*}
\begin{split}
\mathbb{E}[\| {F(x) - F( {{x^\dag }}) - F'( x^\dag )( {x - {x^\dag }} )} \|^2]^{\frac{1}{2}}&\\
&\!\!\!\!\!\!\!\!\!\!\!\!\!\!\!\!\!\!\!\!\!\!\!\!\!\!\!\!\!\!\!\!\!\!\!\!\!\!\!\!\!\!\!\!\!\!\!\!\!\!\!\!\!\!\!\!\le \frac{c_R}{2}\mathbb{E}[\| {x - {x^\dag }}\|^2]^{\frac{1}{2}}\mathbb{E}[\| {F'( {{x^\dag }} )( {x - {x^\dag }})} \|^2]^{\frac{1}{2}},
\end{split}
\end{equation*}
which will be frequently used later for deriving error bounds.

\subsection{Necessary concepts of stochastic calculus}
In SAR \eqref{initial value problem}, a stochastic integral with respect to a Wiener process is involved. In this subsection, we state some concepts of stochastic calculus, which will be useful for deriving corresponding error bounds in the following analysis.
Most of these concepts can be found in \cite[Chap. 2]{Gaw2011}, and are compiled here for completeness.

Recall that $\mathcal{X'}$  and $\mathcal{X}$  are separable Hilbert spaces, and let $\mathcal{Q}$ be a symmetric nonnegative definite trace-class operator on $\mathcal{X}$. Assume that all eigenvalues $\lambda_j >0$ ($j=1,2,\cdots$) and the associated eigenvectors $\varphi_j$ ($j=1,2,\cdots$) form an orthonormal basis (ONB) on $\mathcal{X}$. Then, the space $\mathcal{X}_\mathcal{Q}=\mathcal{Q}^{1/2}\mathcal{X}$ equipped with the scalar product
\begin{equation*}
\langle u,v \rangle_{\mathcal{X}_\mathcal{Q}}=\sum_{j=1}^\infty \frac{1}{\lambda_j} \langle u,\varphi_j \rangle_{\mathcal{X}}\langle v,\varphi_j \rangle_{\mathcal{X}}
\end{equation*}
is a separable Hilbert space with an ONB $\{\lambda_j^{1/2}\varphi_j\}_{j=1}^\infty$.

Let the space of Hilbert-Schmidt operators from $\mathcal{X}_\mathcal{Q}$ to $\mathcal{X'}$ be denoted by $\mathcal{L}_2(\mathcal{X}_\mathcal{Q}, \mathcal{X'})$, which is also separable since $\mathcal{X'}$ and $\mathcal{X}_\mathcal{Q}$ are separable.
If $\{e_j\}_{j=1}^\infty$ is an ONB in $\mathcal{X'}$, the Hilbert-Schmidt norm of an operator $L\in \mathcal{L}_2(\mathcal{X}_\mathcal{Q}, \mathcal{X'})$ is given by
\begin{equation*}
\begin{split}
\|L \|^2_{\mathcal{L}_2(\mathcal{X}_\mathcal{Q}, \mathcal{X'})}&=\sum\limits_{j,i = 1}^\infty \langle L(\lambda_j^{1/2}\varphi_j), e_i\rangle_\mathcal{X'}^2=\sum\limits_{j,i = 1}^\infty \langle L\mathcal{Q}^{1/2}\varphi_j, e_i\rangle_\mathcal{X'}^2\\
&=\| L\mathcal{Q}^{1/2} \|^2_{\mathcal{L}_2(\mathcal{X}, \mathcal{X'})}=\textmd{tr}\left((L\mathcal{Q}^{1/2})(L\mathcal{Q}^{1/2})^*\right).
\end{split}
\end{equation*}

Let $\{\beta_j(t)\}_{t\geq0}$, $j=1,2,\ldots$, be a sequence of independent Brownian motions defined on a filtered probability space $(\Omega,\mathcal{F},\{\mathcal{F}\}_{t\geq0},P)$; then, the $\mathcal{X}$-valued $\mathcal{Q}$-Wiener process $B_t$ is defined by
\begin{equation*}
B_t= \sum\limits_{j = 1}^\infty {\lambda_j}^{1/2} {\beta _j} \left( t \right)\varphi_j.
\end{equation*}

In addition, let $\Lambda_2(\mathcal{X}_\mathcal{Q},\mathcal{X'} )$ be a class of $\mathcal{L}_2(\mathcal{X}_\mathcal{Q},\mathcal{X'})$-valued processes satisfying the condition $\mathbb{E}\int_0^T \|\Phi(s)\|^2_{\mathcal{L}_2(\mathcal{X}_\mathcal{Q},\mathcal{X'})}\textmd{ds} < \infty$. It can be verified that $\Lambda_2(\mathcal{X}_\mathcal{Q},\mathcal{X'} )$ is a Hilbert space equipped with the norm $\|\Phi\|^2_{\Lambda_2(\mathcal{X}_\mathcal{Q},\mathcal{X'} )} = \mathbb{E}\int_0^T \|\Phi(s)\|^2_{\mathcal{L}_2(\mathcal{X}_\mathcal{Q},\mathcal{X'})}\textmd{ds}$.

For $\Phi \in  \mathcal{L}_2(\mathcal{X}_\mathcal{Q},\mathcal{X'})$, the stochastic integral
$\int^t_0 \Phi(s)\textmd{d}B_s$, $0\leq  t \leq  T$, can be defined as in the finite-dimensional case, based on elementary processes and continuous extension; see \cite[sec. 2.2]{Gaw2011} for details. The following theorem in \cite{Gaw2011}, which is the It\^o isometry in the infinite-dimensional setting, is important and forms the main tool for managing the stochastic integrals in this paper.

\begin{theorem}\label{Itoiso}\emph{(\cite[Theorem 2.3]{Gaw2011})}.
	The stochastic integral $\Phi \to \int^t_0 \Phi(s)\emph{d}B_s$ with respect to an $\mathcal{X}$-valued $\mathcal{Q}$-Wiener process $B_t$ is a continuous square-integrable martingale, and satisfies
	\begin{equation}\label{QWiener}
	\mathbb{E}\left\|\int_0^t {\Phi\left( s \right)\emph{d}{B_s}}\right\|^2_\mathcal{X'}=\mathbb{E}\int_0^t \left\|\Phi\left( s \right)\right\|^2_{\mathcal{L}_2(\mathcal{X}_\mathcal{Q},\mathcal{X'})} \emph{ds} < \infty, \quad \forall t\in[0,T].
	\end{equation}
\end{theorem}

We now present a theorem that provides conditions under which a stochastic process $\mathcal{S}(t,X(t))$ has a stochastic differential, provided that $X(t)$ also has a stochastic differential. It is also known as the It\^{o} formula for the case of a $\mathcal{Q}$-Wiener process.

\begin{theorem}\label{Itoform} \emph{(\cite[Theorem 2.9]{Gaw2011})}.
	Let $\mathcal{X'}$  and $\mathcal{X}$  be separable Hilbert spaces and $B_t$ be the $\mathcal{X}$-valued $\mathcal{Q}$-Wiener process. Assume that a stochastic process $X(t)$, $0\le t \le T$, is given by
	\begin{equation*}
	X(t)=X(0)+\int_0^t \Psi(s)\emph{ds} +\int_0^t \Phi(s)\emph{d}B_s,
	\end{equation*}
	where $X(0)$ is an $\mathcal{X'}$-valued random variable, $\Psi(s)$ is an $\mathcal{X'}$-valued Bochner-integrable process on $[0,T]$, satisfying $\int_0^T\| \Psi(s)\|_\mathcal{X'} \emph{ds} < \infty$, P-a.s., and $\Phi \in \mathcal{L}_2(\mathcal{X}_\mathcal{Q},\mathcal{X'})$.
	
	Assume that a function $\mathcal{S}: [0,T] \times \mathcal{X'} \to \mathbb{R}$ is continuous and
	its Fr\'{e}chet partial derivatives $\mathcal{S}_t$, $\mathcal{S}_x$, $\mathcal{S}_{xx}$ are continuous and bounded on bounded subsets of $[0,T] \times \mathcal{X'}$. Then, the following It\^{o} formula holds $P$-$a.s.$ for all $t\in[0,T]$:
	\begin{equation}\label{Itoformula}
	\begin{split}
	&\mathcal{S}(t,X(t))=\mathcal{S}(0,X(0)) +\int_0^t \langle \mathcal{S}_x(s,X(s)), \Phi(s)\emph{d}B_s \rangle\\
	&+\!\int_0^t \left[\mathcal{S}_t(s,X(s))\!+\!\langle \mathcal{S}_x(s,X(s)), \Psi(s)\rangle \right]\emph{ds} \!+\! \frac{1}{2} \emph{tr}[\mathcal{S}_{xx}(s,X(s))\Phi(s)(\Phi(s))^*]\emph{ds} .
	\end{split}
	\end{equation}
\end{theorem}

\section{Convergence analysis}
\label{sec-convergence}

Throughout this paper, the notation $c$, with a subscript and/or superscript, denotes a generic constant, which may differ at each occurrence but is always independent of the noise level $\delta$ and the time variable $t$. An example of these constants can be found in Table \ref{NotationTable} in Appendix D.  The notation $(\mathbb{E}[X^a])^{b}$ will be written as $\mathbb{E}[X^a]^b$, with $a,b>0$ for simplification.

\begin{prop}\label{prop1}
	Let $x^\delta(t)$ be the solution to problem \eqref{initial value problem} for $T>0$; then,
	\begin{equation}\label{a1201}
	\frac{\textmd{d}}{\textmd{dt}}\mathbb{E}[ \|{F( {{x^\delta }( t )} )-{y^\delta }}\|^2]=-2\mathbb{E}[\| F'{( {{x^\delta }( t )} )^ * }[ {F( {{x^\delta }( t )} )-{y^\delta }} ]\|^2].
	\end{equation}
\end{prop}

\begin{proof}
	Note that
	\begin{equation*}
	\textmd{d}\| F(x^\delta(t))-y^\delta\|^2 = 2\left\langle F'(x^\delta(t))\textmd{d}x^\delta(t), F(x^\delta(t))-y^\delta \right\rangle.
	\end{equation*}
	When \eqref{initial value problem} is inserted into the above equality, it easily follows that
	\begin{equation*}
	\begin{split}
	\textmd{d}\| F(x^\delta(t))-y^\delta\|^2
	=-&2\| F'{\left( {{x^\delta }\left( t \right)} \right)^ * }\left[ {F\left( {{x^\delta }\left( t \right)} \right)-{y^\delta }} \right]\|^2\textmd{dt} \\
	+ &2\left\langle F'(x^\delta(t))^*(F(x^\delta(t))-y^\delta), f\left( t \right)\textmd{d}{B_t} \right\rangle.
	\end{split}
	\end{equation*}
	Then, taking the integration of the above equality from $0$ to $t$, we obtain
	\begin{equation}\label{yaa}
	\begin{split}
	\int_0^t\textmd{d}\left\| F(x^\delta(s))-y^\delta\right\|^2 =& \left\|{F\left( {{x^\delta }\left( t \right)} \right)-{y^\delta }}\right\|^2-\left\| F(\bar x)- y^\delta \right\|^2 \\
	=&-2\int_0^t \left\| F'{\left( {{x^\delta }\left( s \right)} \right)^ * }\left[ {F\left( {{x^\delta }\left( s \right)} \right)-{y^\delta }} \right]\right\|^2\textmd{ds} \\
	&+ 2\int_0^t \left\langle F'(x^\delta(s))^*(F(x^\delta(s))-y^\delta), f\left( s \right)\textmd{d}{B_s} \right\rangle.
	\end{split}
	\end{equation}
	Recalling Theorem \ref{Itoiso}, in which $\int_0^t f\left( s \right)\textmd{d}{B_s}$ is a martingale, the following holds:
	
	\begin{equation*}
	\mathbb{E}[ \|{F( {{x^\delta }( t )} )-{y^\delta }}\|^2]
	= \left\| F(\bar x)- y^\delta \right\|^2 -2\mathbb{E} \int_0^t \left\| F'{\left( {{x^\delta }\left( s \right)} \right)^ * }\left[ {F\left( {{x^\delta }\left( s \right)} \right)-{y^\delta }} \right]\right\|^2\textmd{ds}.
	\end{equation*}
	Consequently,
	\begin{equation*}
	\begin{split}
	\frac{\textmd{d}}{\textmd{dt}}\mathbb{E}[ \|{F\left( {{x^\delta }\left( t \right)} \right)-{y^\delta }}\|^2]&=-2\frac{\textmd{d}}{\textmd{dt}}\mathbb{E}\int_0^t \left\| F'{\left( {{x^\delta }\left( s \right)} \right)^ * }\left[ {F\left( {{x^\delta }\left( s \right)} \right)-{y^\delta }} \right]\right\|^2\textmd{ds}\\ \nonumber
	&=-2\mathbb{E}[\| F'{\left( {{x^\delta }\left( t \right)} \right)^ * }\left[ {F\left( {{x^\delta }\left( t \right)} \right)-{y^\delta }} \right]\|^2] \le 0,
	\end{split}
	\end{equation*}
	which completes the proof.
\end{proof}
Property \eqref{a1201} shows that the discrepancy $\mathbb{E}[\|{F\left( {{x^\delta }\left( t \right)} \right)-{y^\delta }}\|^2]$ as a function of $t$ is monotonically non-increasing.
This serves as an indication to choose the auxiliary function $f(t)$, namely
\begin{equation}\label{ft}
f(t):= \delta \mathbb{E}[\|{F\left( {{x^\delta }\left( t \right)} \right)-{y^\delta }}\|^2]^\frac{1}{2} g(t),
\end{equation}
where $g(t)$ can be any monotonically decreasing function such that
\begin{equation}\label{gt}
\lim_{t\to\infty} g(t) =0, \quad g(0)=\frac{\sqrt{\epsilon_0 \eta}}{\delta_0}, \quad \epsilon_0\in\left(0,2\left(\frac{1}{\eta}-1\right) \right).
\end{equation}

It is clear that $f(t)\in L^\infty (\mathbb{R}_+)$, which exactly coincides with the restriction on $f(t)$ given in the initial value problem \eqref{initial value problem}. Moreover, by the definition of $g(t)$ we deduce the following inequality:
\begin{equation}\label{hg}
\left|f(t)\right|^2 \le \delta^2  \mathbb{E}[\|{F\left( {{x^\delta }\left( t \right)} \right)-{y^\delta }}\|^2] [g(0)]^2 \le \epsilon_0 \eta  \mathbb{E}[\|{F\left( {{x^\delta }\left( t \right)} \right)-{y^\delta }}\|^2],
\end{equation}
which will play an important role in further analysis.

The following proposition shows the well-posedness of initial value problem \eqref{initial value problem}.

\begin{prop}\label{prop0}
	For any monotonically decreasing function $g(t)$ satisfying \eqref{gt}, stochastic differential equation \eqref{initial value problem} has a unique solution $x^\delta (t) \in \mathcal{X}$, given by
	\begin{equation}\label{cg01}
	{x^\delta }\left( t \right) = \bar x + \int_0^t {F'{{\left( {{x^\delta }\left( s \right)} \right)}^ * }\left[ {{y^\delta } - F\left( {{x^\delta }\left( s \right)} \right)} \right]\textmd{ds}}  + \int_0^t {f\left( s \right)\textmd{d}{B_s}}.
	\end{equation}
\end{prop}

\begin{proof}
	Equation \eqref{cg01} follows immediately by integrating stochastic differential equation \eqref{initial value problem} from $0$ to $t$. We only need to deal with the stochastic term of \eqref{cg01} here.
	
	Recall that $B_t$ is an $\mathcal{X}$-valued $\mathcal{Q}$-Wiener process and $f(t)\in L^\infty (\mathbb{R}_+)$. Then according to It\^{o}'s isometry, introduced in Theorem \ref{Itoiso}, we have
	\begin{equation*}
	\mathbb{E}\left\|\int_0^t {f\left( s \right)\textmd{d}{B_s}}\right\|^2=\mathbb{E}\int_0^t \left|f\left( s \right)\right|^2 \textmd{ds} < \infty
	\end{equation*}
	for all $0\le t \le T$. The proposition thus holds true.
\end{proof}

We are now able to prove the following proposition:
\begin{prop}\label{prop3}
	Let $x^*$ and $x^\delta(t)$ be the solutions to problems \eqref{main} and \eqref{initial value problem}, respectively. If Assumption \ref{assumption-1} is satisfied, then the following inequality holds:
	\begin{equation}\label{monoton}
	\begin{split}
	&\frac{\textmd{d}}{{\textmd{dt}}}\mathbb{E}[ {{{\| {{x^\delta }\left( t \right) - {x^*}} \|}^2}} ]\\
	&\le\left((2+\epsilon_0)\eta-2\right)\left(\mathbb{E} [{{{\| {F\left( {{x^\delta }\left( t \right)} \right) - {y^\delta }} \|^2}}}]^\frac{1}{2} - \frac{{2 + 2\eta }}{2-(2+\epsilon_0)\eta}\delta \right)\mathbb{E}[ {\| {F\left( {{x^\delta }\left( t \right)} \right) - {y^\delta }} \|^2}]^\frac{1}{2}.
	\end{split}
	\end{equation}
	Furthermore, in the case of $\delta=0$, the following also holds:
	\begin{equation}\label{a1213}
	\lim_{t\to\infty} \mathbb{E} [{{{\| {F\left( {{x}\left( t \right)} \right) - {y }} \|}}}^2]=0.
	\end{equation}
\end{prop}

\begin{proof}
	It is clear that equation \eqref{cg01} is equal to
	\begin{equation}\label{cg02}
	{x^\delta }\left( t \right) - {x^*} = \bar x - {x^*} + \int_0^t {F'{{\left( {{x^\delta }\left( s \right)} \right)}^ * }\left[ {{y^\delta } - F\left( {{x^\delta }\left( s \right)} \right)} \right]\textmd{ds}}  + \int_0^t {f\left( s \right)\textmd{d}{B_s}}.
	\end{equation}
	From It\^{o}'s formula, as shown in Theorem \ref{Itoform}, we have
	\begin{equation}\label{a11}
	\begin{split}
	{\left\| {{x^\delta }\left( t \right)\! - \!{x^*}} \right\|^2} \!=\! {\left\| {\bar x \!-\! {x^*}} \right\|^2} \!&+\! 2\!\int_0^t {\!\!\!\left\langle {{x^\delta }\left( s \right) \!- \! {x^*},F'{{\left( {{x^\delta }\left( s \right)} \right)}^ * }\!\left[ {{y^\delta }\! - \!F\left( {{x^\delta }\left( s \right)} \right)} \right]}\! \right\rangle } \textmd{ds}\\
	& \!\!\!\!\!\! \!\!\!\! \!\!\!\! \!\!\!\!\!\!+ 2 \int_0^t {\left\langle {{x^\delta }\left( s \right) - {x^*},f\left( s \right)\textmd{d}{B_s}} \right\rangle }  + \int_0^t {\textmd{tr}\left[ {f\left( s \right){{\left( {f\left( s \right)} \right)}^*}} \right]\textmd{ds}}.
	\end{split}
	\end{equation}
	Using the fact that $\int_0^t f\left( s \right)\textmd{d}{B_s}$ is a martingale (see Theorem \ref{Itoiso}), we obtain
	\begin{equation}\label{cg04}
	\begin{split}
	\mathbb{E}[ {{{\| {{x^\delta }\left( t \right) - {x^*}} \|}^2}} ]& = {{{\left\| {\bar x - {x^*}} \right\|}^2}} + {\int_0^t {\left|{f\left( s \right)} \right|^2\textmd{ds}} }\\
	&~~+ 2\mathbb{E} {\int_0^t {\left\langle {{x^\delta }\left( s \right) - {x^*},F'{{\left( {{x^\delta }\left( s \right)} \right)}^ * }\left[ {{y^\delta } - F\left( {{x^\delta }\left( s \right)} \right)} \right]} \right\rangle } \textmd{ds}}.
	\end{split}
	\end{equation}
    Differentiating \eqref{cg04} with respect to $t$, we get
	\begin{equation*}
	\frac{\textmd{d}}{{\textmd{dt}}}\mathbb{E}[ {{{\| {{x^\delta }\left( t \right) - {x^*}} \|}^2}} ]=2\mathbb{E}\left[ {\left\langle {F'\left( {{x^\delta }\left( t \right)} \right)\left( {{x^\delta }\left( t \right) - {x^*}} \right), {{y^\delta } - F\left( {{x^\delta }\left( t \right)} \right)} } \right\rangle } \right] + {| {f\left( t \right)} |^2}.
	\end{equation*}
	If
	\[{\Delta} := \left\langle {F'\left( {{x^\delta }\left( t \right)} \right)\left( {{x^\delta }\left( t \right) - {x^*}} \right), {{y^\delta } - F\left( {{x^\delta }\left( t \right)} \right)} } \right\rangle,\]
	we then have
	\begin{equation*}
	\begin{split}
	\left\| {{\Delta}} \right\| \le& \left\|\left\langle {F\left( {{x^*}} \right) - F\left( {{x^\delta }\left( t \right)} \right) - F'\left( {{x^\delta }\left( t \right)} \right)\left( {{x^*} - {x^\delta }\left( t \right)} \right),{y^\delta } - F\left( {{x^\delta }\left( t \right)} \right)} \right\rangle\right\| \\
	&+ \left\|\left\langle {F\left( {{x^\delta }\left( t \right)} \right) \!-\! {y^\delta },{y^\delta } - F\left( {{x^\delta }\left( t \right)} \right)} \right\rangle\right\|  + \left\|\left\langle {{y^\delta }\! - y,{y^\delta } - F\left( {{x^\delta }\left( t \right)} \right)} \right\rangle\right\| \\
	\le &\left( {\eta  - 1} \right){\left\| {F\left( {{x^\delta }\left( t \right)} \right) - {y^\delta }} \right\|^2} + \left( {1 + \eta } \right)\delta \left\| {F\left( {{x^\delta }\left( t \right)} \right) - {y^\delta }} \right\|.
	\end{split}
	\end{equation*}
	Thus, it follows from \eqref{hg} that
	\begin{equation*}
	\begin{split}
	&\frac{\textmd{d}}{{\textmd{dt}}}\mathbb{E}[ {{{\| {{x^\delta }\left( t \right) - {x^*}} \|}^2}} ] \le 2\mathbb{E}[\left\|\Delta\right\|]+{|{f\left( t \right)} |^2} \\
	&\le2\left( {\eta \! - \!1} \right)\mathbb{E}[ {{{\| {F\left( {{x^\delta }\left( t \right)} \right) - {y^\delta }} \|}^2}} ] + 2 \delta \left( {1 \!+ \!\eta } \right)\mathbb{E}[ {\| {F\left( {{x^\delta }\left( t \right)} \right) - {y^\delta }} \|}] \!+\!{| {f\left( t \right)} |^2}\\
	&\le\left((2+\epsilon_0)\eta\!-\!2\right)\mathbb{E}[ {{{\| {F\left( {{x^\delta }\left( t \right)} \right) - {y^\delta }} \|}^2}} ]+2 \delta \left( {1 + \eta } \right)\mathbb{E}[ {{{\| {F\left( {{x^\delta }\left( t \right)} \right) - {y^\delta }} \|}^2}} ]^\frac{1}{2}\\
	&=\left((2+\epsilon_0)\eta\!-\!2\right)\left(\mathbb{E}[ {{{\| {F\left( {{x^\delta }\left( t \right)} \right) - {y^\delta }} \|}^2}} ]^\frac{1}{2} - \frac{{2 + 2\eta }}{2-(2+\epsilon_0)\eta}\delta \right)\mathbb{E}[ {{{\| {F\left( {{x^\delta }\left( t \right)} \right) - {y^\delta }} \|}^2}} ]^\frac{1}{2},
	\end{split}
	\end{equation*}
	which leads to the first assertion.
	
	In addition, in the case of $\delta=0$, by integrating the above inequality we obtain
	\begin{equation*}
	\int_0^\infty\frac{\textmd{d}}{{\textmd{dt}}}\mathbb{E}[ {{{\| {{x }\left( t \right) - {x^*}} \|}^2}} ] \le -\left(2-(2+\epsilon_0)\eta\right)\int_0^\infty \mathbb{E} [{{{\| {F\left( {{x }\left( t \right)} \right) - {y}} \|}}}^2] \textmd{dt},
	\end{equation*}
	which implies the second assertion, since
	\begin{equation}\label{a1214}
	\int_0^\infty \mathbb{E} [{{{\| {F\left( {{x }\left( t \right)} \right) - {y}} \|}}}^2] \textmd{dt} \le \frac{1}{2-(2+\epsilon_0)\eta} {{{\| {{\bar x} - {x^*}} \|}^2}}.
	\end{equation}
\end{proof}

Property \eqref{monoton} shows that the error $\mathbb{E}[ {{{\| {{x^\delta }\left( t \right) - {x^*}} \|}^2}} ]$ as a function of $t$ is strongly monotonically decreasing as long as $\mathbb{E}[ {\| {F\left( {{x^\delta }\left( t \right)} \right) - {y^\delta }} \|}^2]\geq \tau^2 \delta^2$ holds with $\tau>({2 + 2\eta })/({2-(2+\epsilon_0)\eta})$.
This leads us to consider the following modified version of Morozov's discrepancy principle to determine the stopping time; i.e. $t=t^*$ is chosen such that
\begin{equation}\label{stopping}
t^* = \inf\left\{ t>0: H(t) <0 \right\},
\end{equation}
where
\begin{equation}\label{hstopping}
H(t):=\mathbb{E} [{{{\| {F\left( {{x^\delta }\left( t \right)} \right) - {y^\delta }} \|}}}^2] - \tau^2\delta^2, \quad \tau >({2 + 2\eta })/({2-(2+\epsilon_0)\eta}).
\end{equation}

\begin{example}
	If $\eta=3/4$ and $\epsilon_0$ is chosen to be a very small number, $\tau$ can be chosen such that $\tau>7$. If we choose $\epsilon_0\in\left( \max\{\frac{2-\sqrt{2}}{\eta}-(2+\sqrt{2}),0\}, 2\left(\frac{1}{\eta}-1\right) \right)$, we have $\tau>\sqrt{2}$, which will be used in the proof of the convergence-rate results (see, e.g., \eqref{cc} in Proposition \ref{PropositionBiasErr}).
\end{example}

The following proposition indicates that stopping rule \eqref{stopping} above is well defined.

\begin{prop}\label{prop4}
	Let $x^\delta(t)$ be the solution to \eqref{initial value problem} and $x^*$ a solution to \eqref{main}. Assume that ${\|{F\left( {{\bar x }} \right) - {y^\delta }} \|}>\tau \delta$. Then, under Assumption \ref{assumption-1}, there exists a unique $t^*$, defined by \eqref{stopping}.
\end{prop}

\begin{proof}
	Note that the function $H(t)$ is continuous, with $H(0)={{{\| {F\left( \bar x\right) - {y^\delta }} \|}}}^2-\tau^2 \delta^2>0$, and, according to Proposition \ref{prop1}, $H(t)$ is monotonically non-increasing as long as $t<t^*$.
	
	We now prove the assertion by contradiction. Assuming an outcome contrary to the conclusion, there exists $t_0 <\infty$ such that, for $t \in [t_0 , t_0 +\varepsilon]$ and with some $\varepsilon>0$,
	\begin{equation}\label{pfcontradiction}
	\mathbb{E} [{{{\| {F( {{x^\delta }( t)}) - {y^\delta }} \|}}}^2] = \tau^2_0 \delta^2, \quad \tau_0\geq \tau.
	\end{equation}
	Consequently, $(\textmd{d}/\textmd{dt})\mathbb{E} [{{{\| {F( {{x^\delta }( t )} ) - {y^\delta }} \|}}}^2] = 0$ for all $t \in (t_0 , t_0 +\varepsilon)$. Together with \eqref{a1201}, we have
	\begin{equation*}
	\mathbb{E}[\| F'{\left( {{x^\delta }\left( t \right)} \right)^ * }\left[ {F\left( {{x^\delta }\left( t \right)} \right)-{y^\delta }} \right]\|^2]=0,~~\forall t\in[t_0,t_0+\varepsilon],
	\end{equation*}
	which implies
	\begin{equation*}
	F'{\left( {{x^\delta }\left( t \right)} \right)^ * }\left[ {F\left( {{x^\delta }\left( t \right)} \right)-{y^\delta }} \right]=0,~ a.s. ~\forall t\in[t_0,t_0+\varepsilon].
	\end{equation*}
	Combining this with \eqref{cg04}, we have
	\begin{equation*}
	\mathbb{E}[ {{{\| {{x^\delta }\left( t \right) - {x^*}} \|}^2}}] = {{{\left\| {\bar x - {x^*}} \right\|}^2}} + {\int_0^t {\left|{f\left( s \right)} \right|^2\textmd{ds}} },~~ \forall t\in[t_0,t_0+\varepsilon].
	\end{equation*}
	Taking the derivative of the above equality with respect to $t$ yields
	\begin{equation*}
	\frac{\textmd{d}}{\textmd{dt}}\mathbb{E}[ {{{\| {{x^\delta }\left( t \right) - {x^*}} \|}^2}} ] = \left|f(t)\right|^2,~~\forall t\in[t_0,t_0+\varepsilon].
	\end{equation*}
	Consequently, together with \eqref{monoton}, we have
	\begin{equation*}
	\left( (2+\epsilon_0)\eta  - 2 \right)\mathbb{E}[ {{{\| {F\left( {{x^\delta }\left( t \right)} \right) - {y^\delta }} \|}^2}} ] + 2 \delta \left( {1 + \eta } \right)\mathbb{E}[ {\| {F\left( {{x^\delta }\left( t \right)} \right) - {y^\delta }} \|^2}]^\frac{1}{2} > 0,
	\end{equation*}
	for all $t\in[t_0,t_0+\varepsilon]$; that is,
	\begin{equation*}
	\mathbb{E}[ {{{\| {F\left( {{x^\delta }\left( t \right)} \right) - {y^\delta }} \|}^2}} ] < \left( \frac{2+2\eta}{2-(2+\epsilon_0)\eta}\right)^2\delta^2<\tau^2\delta^2,~~\forall t\in[t_0,t_0+\varepsilon].
	\end{equation*}
	This contradicts our assumption $\tau_0\geq \tau$ in \eqref{pfcontradiction}, and so the proof is completed.
\end{proof}

\begin{remark}
	From the inequalities
	\[\mathbb{E}[\|x^\delta(t)-\bar x\|]^2 \le \mathbb{E}[\|x^\delta(t)-\bar x\|^2] \le \mathbb{E}[\|x^\delta(t)-x^*\|^2] +\|x^*-\bar x\|^2,\]
	it follows that $\mathbb{E}[\|x^\delta(t)-\bar x\|]^2 \le 2\|x^*-\bar x\|^2$ for $t>0$ since $(\textmd{d}/\textmd{dt})\mathbb{E}[\|x^\delta(t)-x^*\|^2]\le 0$ for all $t\le t^*$, and hence $\mathbb{E}[\|x^\delta(t)-\bar x\|] \le 2\|x^*-\bar x\|$. This means that the solution $x^\delta(t)$ remains in $B_r(\bar x)$ with $r=2\|\bar x - x^*\|$ in the sense of expectation.
	Similarly, $\mathbb{E}[\|x(t)-\bar x\|] \le 2\|x^*-\bar x\|$ also holds, since $(\textmd{d}/\textmd{dt})\mathbb{E}[\|x(t)-x^*\|^2]\le 0$ for all $t< \infty$.
\end{remark}

We can now state the convergence of SAR for the exact data $y$, provided Assumption \ref{assumption-1} holds.
\begin{theorem}[convergence for exact data]\label{convergencenoisefree}
	Let $x(t)$ be the solution to \eqref{initial value problem}, with exact data $y$ and $x^*$ a solution to \eqref{main}. If Assumption \ref{assumption-1} is fulfilled, the following holds:
	\begin{equation*}
	\mathop {\lim }\limits_{t  \to \infty} \mathbb{E}[ {{{\| {{x}\left( t \right) - {x^*}}\|}^2}} ] = 0.
	\end{equation*}
	In addition, if $\mathcal{N}(F'(x^\dag))\subset \mathcal{N}(F'(x(t)))$ for all $t>0$,
	\begin{equation*}
	\mathop {\lim }\limits_{t  \to \infty} \mathbb{E}[ {{{\| {{x}\left( t \right) - {x^\dag}}\|}^2}} ] = 0.
	\end{equation*}
\end{theorem}

\begin{proof}
	Let $\hat x$ be any solution to problem \eqref{main}. We introduce the notation
	\begin{equation*}
	e(t):=\hat x- x(t)
	\end{equation*}
	and will show that $\mathbb{E}[\|e(t)\|^2]=0$ as $t\to \infty$.
	
	Let $l$ be any arbitrary real number, with $l>t$; then, we can easily obtain
	\begin{equation*}
	\mathbb{E}[\left\|e(t)-e(l)\right\|^2]=\mathbb{E}[\left\|e(t)\right\|^2]-\mathbb{E}[\left\|e(l)\right\|^2]+2\mathbb{E}[\left\langle e(l)-e(t),e(t)\right\rangle].
	\end{equation*}
	From \eqref{monoton}, with $\delta=0$, it follows that $\mathbb{E}[\left\|e(t)\right\|^2]$ and hence $\mathbb{E}[\left\|e(l)\right\|^2]$ converge to some $\epsilon \geq 0$ as $t\to \infty$. Consequently,
	\[\lim_{t\to\infty}\mathbb{E}[\left\|e(t)\right\|^2-\left\|e(l)\right\|^2]=\epsilon-\epsilon=0.\]
	Next, we will show that $\mathbb{E}[\left\langle e(l)-e(t),e(t)\right\rangle]\to 0$ as $t\to \infty$.
	
	Since
	\begin{equation*}
	e(l)-e(t)=x(t)-x(l)=\int_t^l F'{\left( {{x }\left( s \right)} \right)^ * }\left[ y-{F\left( {{x }\left( s \right)} \right)} \right]\textmd{ds}+\int_t^l f(s){\textmd{d}B_s},
	\end{equation*}
	it immediately follows that
	\begin{equation*}
	\left\langle e(l)-e(t),e(l)\right\rangle=\langle\int_t^l F'{\left( {{x }\left( s \right)} \right)^ * }\left[ y-{F\left( {{x }\left( s \right)} \right)} \right]\textmd{ds},\hat x-x(l)\rangle +\langle \int_t^l f(s){\textmd{d}B_s},\hat x-x(l)\rangle.
	\end{equation*}
	Furthermore,
	\begin{equation*}
	\begin{split}
	& |\langle\int_t^l F'{\left( {{x }\left( s \right)} \right)^ * }\left[ y-{F\left( {{x }\left( s \right)} \right)} \right]\textmd{ds},\hat x-x(l)\rangle |  \\
	&\le \int_t^l \left\|y-{F\left( {{x }\left( s \right)} \right)}\right\|\left\|F'{\left( {{x }\left( s \right)} \right)}\left(\hat x-x(l)\right)\right\|\textmd{ds}\\
	& \le \int_t^l \left\|y-{F\left( {{x }\left( s \right)} \right)}\right\|\left(\left\|F'{\left( {{x }\left( s \right)} \right)}\left(\hat x-x(s)\right)\right\|+\left\|F'{\left( {{x }\left( s \right)} \right)}\left( x(s)-x(l)\right)\right\|\right)\textmd{ds}.
	\end{split}
	\end{equation*}
	Again, using the fact that $\int_t^l f(s){\textmd{d}B_s}$ is a martingale, we obtain
	\begin{equation*}
	\begin{split}
	\left|\mathbb{E}[\left\langle e(l)-e(t),e(l)\right\rangle]\right|
	\le& \int_t^l \mathbb{E} \left[\left\|y-{F\left( {{x }\left( s \right)} \right)}\right\|\left\|F'{\left( {{x }\left( s \right)} \right)}\left(\hat x-x(s)\right)\right\|\right]\textmd{ds}\\
	&+\int_t^l \mathbb{E} \left[\left\|y-{F\left( {{x }\left( s \right)} \right)}\right\|\left\|F'{\left( {{x }\left( s \right)} \right)}\left(x(s)-x(l)\right)\right\|\right]\textmd{ds} \\
	\le& \int_t^l \mathbb{E} [\left\|y-{F\left( {{x }\left( s \right)} \right)}\right\|^2]^{\frac{1}{2}}\mathbb{E}[\|F'{\left( {{x }\left( s \right)} \right)}\left(\hat x-x(s)\right)\|^2]^{\frac{1}{2}}\textmd{ds}\\
	&+ \int_t^l \mathbb{E} [\left\|y-{F\left( {{x }\left( s \right)} \right)}\right\|^2]^{\frac{1}{2}}\mathbb{E}[\|F'{\left( {{x }\left( s \right)} \right)}\left(x(s)-x(l)\right)\|^2]^{\frac{1}{2}}\textmd{ds}.
	\end{split}
	\end{equation*}
	From \eqref{F2}, it follows that
	\begin{equation*}
	\mathbb{E}[\|F'{\left( {{x }\left( s \right)} \right)}\left(\hat x-x(s)\right)\|^2]\le (1+\eta)^2\mathbb{E}[\|y-F(x(s))\|^2]
	\end{equation*}
	and
	\begin{equation*}
	\mathbb{E}[\|F'{\left( {{x }\left( s \right)} \right)}\left(x(s)-x(l)\right)\|^2]\le 4(1+\eta)^2\mathbb{E}[\|y-F(x(s))\|^2].
	\end{equation*}
	Thus,
	\begin{equation}\label{tem1}
	\left|\mathbb{E}[\left\langle e(l)-e(t),e(l)\right\rangle]\right|\le 3(1+\eta) \int_t^l \mathbb{E}[\|y-F(x(s))\|^2]\textmd{ds}.
	\end{equation}
	Combining estimate \eqref{tem1} with \eqref{a1213}, we now have
	\begin{equation*}
	\lim_{t\to\infty}\left|\mathbb{E}[\left\langle e(l)-e(t),e(l)\right\rangle]\right|=0,
	\end{equation*}
	which implies
	\begin{equation*}
	\lim_{t\to\infty}\mathbb{E}[\left\|e(t)-e(l)\right\|^2]=0.
	\end{equation*}
	This means that $\lim_{t\to \infty} e(t)$ and hence $\lim_{t\to \infty} x(t)$ exist. We use $x^*$ to denote the limit of the latter; furthermore, $x^*$ is a solution, since from \eqref{a1213} the mean-square residual $\mathbb{E} [{{{\| {F\left( {{x }\left( t \right)} \right) - {y }} \|}}}^2]$ converges to zero as $t \to \infty$.
	
	Moreover, from \eqref{F2}, problem \eqref{main} has a unique solution of minimal distance to the initial guess $\bar x$, which satisfies
	\begin{equation*}
	x^\dag-\bar x \in \mathcal{N}(F'(x^\dag))^\bot.
	\end{equation*}
	If $\mathcal{N}(F'(x^\dag))\subset \mathcal{N}(F'(x(t)))$ for all $t>0$, it is clear that $x(t)-\bar x \in \mathcal{N}(F'(x^\dag))^\bot$. Hence,
	\begin{equation*}
	x^\dag-x^*= x^\dag-\bar x+ \bar x -x^* \in \mathcal{N}(F'(x^\dag))^\bot,
	\end{equation*}
	which, together with Proposition \ref{prop2}, implies $x^*=x^\dag$.
\end{proof}

We are now able to show the regularization property of $x^\delta(t^*)$ under an \textit{a posteriori} choice of $t^*$ in \eqref{stopping}.

\begin{theorem}[convergence for noisy data]\label{thmconvergence}
	Let Assumption \ref{assumption-1} be satisfied and ${\|{F\left( {{\bar x }} \right) - {y^\delta }} \|}>\tau \delta$. Then, if the terminating time $t^*=t^*(\delta, y^\delta)$ is determined by \eqref{stopping}, there exists a solution $x^*$ to problem \eqref{main} such that
	\[\mathop {\lim }\limits_{\delta  \to 0} \mathbb{E}[ {{{\| {{x^\delta }\left( t^* \right) - {x^*}}\|}^2}} ] = 0.\]
	In addition, if $\mathcal{N}(F'(x^\dag))\subset \mathcal{N}(F'(x(t)))$ for all $t>0$,
	\[\mathop {\lim }\limits_{\delta  \to 0} \mathbb{E}[ {{{\| {{x^\delta }\left( t^* \right) - {x^\dag}}\|}^2}} ] = 0.\]
\end{theorem}

\begin{proof}
	We distinguish between two cases when $t^*<+\infty$ and when $t^*=+\infty$ as $\delta\to0$. The first case can be done in a manner similar to that of \cite[Theorem 2.4]{Hanke1995}. Hence, we show only the second case. First, let $\{\delta_T\}_{T>0}$ be a sequence converging to zero as $T \to 0$, and  $\{ y^{\delta_T}\}$ a corresponding sequence of noisy data. For each pair $(\delta_T,y^{\delta_T})$, let $t^\delta_{T}:=t^*(\delta_T)$ denote the stopping index determined by discrepancy principle \eqref{stopping}. Furthermore, we may assume that $t^\delta_{T}$ increases strictly monotonically with respect to $T$. Then, for any $\overline{T} <T$, we conclude from Proposition \eqref{monoton} that
	\begin{equation*}
	\begin{split}
	\mathbb{E}[ {{{\| {{x^{\delta_T} }\left( t^\delta_T \right) - {x^*}}\|}^2}} ] &\le  \mathbb{E}[ {{{\| {{x^{\delta_{{T}}} }\left( t^\delta_{\overline{T}} \right) - {x^*}}\|}^2}} ] \\
	&\le 2\mathbb{E}[ {{{\| {{x^{\delta_{{T}}} }\left( t^\delta_{\overline{T}} \right) - {x}\left( t^\delta_{\overline{T}} \right)}\|}^2}} ]+2\mathbb{E}[ {{{\| {{x}\left( t^\delta_{\overline{T}} \right) - {x^*}}\|}^2}} ].
	\end{split}
	\end{equation*}
	Using Theorem \ref{convergencenoisefree}, we can fix a large $\overline{T}$ so that the term $\mathbb{E}[ {{{\| {{x}( t^\delta_{\overline{T}}) - {x^*}}\|}^2}}]$ is sufficiently small. As $t^\delta_{\overline{T}}$ is fixed, we can safely say that $ {x^{\delta_{{T}}} }( t^\delta_{\overline{T}} )$ depends continuously on $y^\delta$ because of the representation of $x^\delta$ in \eqref{cg01} and Assumption \ref{assumption-1}(i). We can then conclude that $\mathbb{E}[ {{{\| {{x^{\delta_{{T}}} }( t^\delta_{\overline{T}} ) - {x}( t^\delta_{\overline{T}})}\|}^2}} ]$ must go to zero as $T \to 0$. This completes the proof of the first assertion.
	The case $\mathcal{N}(F'(x^\dag))\subset \mathcal{N}(F'(x(t)))$ follows in a manner similar to that of Theorem \ref{convergencenoisefree}.
\end{proof}

\section{Convergence rates}\label{sec-convergencerate}
We split the mean-square error of total regularization error $\mathbb{E}[{\| {{x^\delta }(t) - {x^\dag }} \|^2}]$ by using the classic bias-variance decomposition, namely
\begin{equation}\label{bias-variance decomposition}
\mathbb{E}{[\| {{x^\delta }\left( t \right) - {x^\dag }} \|^2]}= {\| {\mathbb{E}{x^\delta }\left( t \right) - {x^\dag }}\|^2} + \mathbb{E}{[\| {{x^\delta }\left( t \right) - \mathbb{E}{x^\delta }\left( t \right)}\|^2]}.
\end{equation}
The two summands on the right-hand side of \eqref{bias-variance decomposition} are called the \emph{bias error} and the \emph{variance error} (extreme right). It should be noted that the bias error is deterministic and can be estimated with standard techniques in regularization theory, while the variance error arising from the $\mathcal{X}$-valued $\mathcal{Q}$-Wiener process consists of an infinite-dimensional stochastic integral, and hence has to be treated with some care.

\subsection{Estimation of the bias error}
To bound the bias error $\|\mathbb{E}{x^\delta }( t ) - {x^\dag }\|$, we first provide an explicit formula for the error between ${x^\delta }\left( t \right)$ and ${x^\dag }$. This is the basis for obtaining order-optimal error bounds. To do this, for any real $t>0$ and any linear operator $K:\mathcal{X} \rightarrow \mathcal{Y}$ such that $\|K\|\leq1$, we define
\begin{equation*}
{e^{ - {K^ * }Kt}} = I + \sum\limits_{n = 1}^\infty  {\frac{{{{\left( { - 1} \right)}^n}{t^n}}}{{n!}}{{\left( {{K^*}K} \right)}^n}}.
\end{equation*}
The sum is absolutely convergent, since $K, K^*$ are bounded linear operators.

\begin{prop}\label{po1}
	Let $x^\delta(t)$ be the solution to problem \eqref{initial value problem} and $K:=F'(x^\dag)$. In addition, let $x^\dag$ be the unique solution of minimal distance to $\bar x$ and deterministic and stationary with respect to the time variable $t$. Then,
	\begin{equation}\label{total error}
	\begin{split}
	{x^\delta }\left( t \right) - {x^\dag } = &~{e^{ - {K^ * }Kt}}\left( {\bar x - {x^\dag }} \right) + \int_0^t{{e^{ - {K^ * }K\left( {t - s} \right)}}} {K^*}\left( {{y^\delta } - y} \right)\textmd{ds}\\
	&+ \int_0^t {{e^{ - {K^ * }K\left( {t - s} \right)}}} \theta \left( s \right)\textmd{ds} + \int_0^t {{e^{ - {K^ * }K\left( {t - s} \right)}}} f\left( s \right)\textmd{d}{B_s},
	\end{split}
	\end{equation}
	where
	\begin{equation*}
	\theta \left( s \right): = {K^*}K\left( {{x^\delta }\left( s \right) - {x^\dag }} \right) - {F'(x^\delta(s))^*}{\left[ {F\left( {{x^\delta }\left( s \right)} \right) - {y^\delta }} \right]} - {K^*}\left( {{y^\delta } - y} \right).
	\end{equation*}
\end{prop}

\begin{proof}
	First, by using It\^{o}'s isometry \eqref{QWiener}, we conclude that
	\begin{eqnarray*}
		\mathbb{E}\left\|\int_0^t {{e^{ - {K^ * }K\left( {t - s} \right)}}}{f\left( s \right)\textmd{d}{B_s}}\right\|^2=\mathbb{E}\int_0^t \left\|{{e^{ - {K^ * }K\left( {t - s} \right)}}}f\left( s \right)\right\|^2_{\mathcal{L}_2(\mathcal{X}_\mathcal{Q},\mathcal{X})}\textmd{ds}< \infty
	\end{eqnarray*}
	for all $0\le t \le T$, which yields the well-definedness of \eqref{total error}.
	
	Using integration by parts, we have
	\begin{equation}\label{p1}
	\int_0^t {{e^{ - {K^*}K\left( {t - s} \right)}}} \textmd{d}{x^\delta }\left( s \right) = {x^\delta }\left( t \right) - {e^{ - {K^*}Kt}}\bar x - \int_0^t {{e^{ - {K^*}K\left( {t - s} \right)}}{K^*}K} {x^\delta }\left( s \right)\textmd{ds}.
	\end{equation}
	Note that
	\begin{equation}\label{p2}
	\int_0^t {{e^{ - {K^*}K\left( {t - s} \right)}}{K^*}K} {x^\dag }\textmd{ds} = {x^\dag } - {e^{ - {K^*}Kt}}{x^\dag }.
	\end{equation}
	Subtracting \eqref{p1} and \eqref{p2}, we obtain, together with the equation \eqref{initial value problem} for $x^\delta(t)$,
	\begin{equation*}
	\begin{split}
	&{x^\delta }\left( t \right) - {x^\dag } = {e^{ - {K^*}Kt}}\left( {\bar x - {x^\dag }} \right) + \int_0^t {{e^{ - {K^*}K\left( {t - s} \right)}}{K^*}K} \left( {{x^\delta }\left( s \right) - {x^\dag }} \right)\textmd{ds} \\
	&+ \int_0^t {{e^{ - {K^*}K\left( {t - s} \right)}}F'} {\left( {{x^\delta }\left( s \right)} \right)^*}\left[ {{y^\delta } - F\left( {{x^\delta }\left( s \right)} \right)} \right]\textmd{ds} + \int_0^t {{e^{ - {K^*}K\left( {t - s} \right)}}f\left( s \right)} \textmd{d}{B_s},
	\end{split}
	\end{equation*}
	which leads to \eqref{total error}.
\end{proof}

The following corollary gives representations of both the bias term $\mathbb{E}{x^\delta }( t ) - {x^\dag }$ and the variance term ${x^\delta }\left( t \right) - \mathbb{E}{x^\delta }\left( t \right)$.
\begin{corollary}\label{coro1}
	Under the assumptions of Proposition \ref{po1}, the following holds:
	\begin{equation}\label{bias term}
	\begin{split}
	\mathbb{E}{x^\delta }\left( t \right) - {x^\dag } ={e^{ - {K^ * }Kt}}\left( {\bar x - {x^\dag }} \right) &+ \int_0^t {{e^{ - {K^ * }K\left( {t - s} \right)}}} {K^*}\left( {{y^\delta } - y} \right)\textmd{ds} \\
	&+\mathbb{E} \int_0^t {{e^{ - {K^ * }K\left( {t - s} \right)}}} \theta \left( s \right)\textmd{ds},
	\end{split}
	\end{equation}
	and
	\begin{equation}\label{variance term}
	{x^\delta }\left( t \right) - \mathbb{E}{x^\delta }\left( t \right)=\int_0^t{{e^{ - {K^ * }K\left( {t - s} \right)}}} f\left( s \right)\textmd{d}{B_s}.
	\end{equation}
\end{corollary}

\begin{proof}
	By taking the full expectation of \eqref{total error}, we derive \eqref{bias term} by noting the martingale property of $\int_0^t {{e^{ - {K^*}K\left( {t - s} \right)}}f\left( s \right)} \textmd{d}{B_s}$. Furthermore, \eqref{variance term} follows from the relation
	\begin{equation*}
	{x^\delta }\left( t \right) - \mathbb{E}{x^\delta }\left( t \right) =  \left( {x^\delta }\left( t \right) - {x^\dag } \right) - \left( \mathbb{E}{x^\delta }\left( t \right) - {x^\dag } \right).
	\end{equation*}
\end{proof}

The following two lemmas will be useful auxiliary tools for our subsequent analysis. The proofs can be found in Appendix A.
\begin{lemma}\label{lemma9}
	If there is a bounded Fr\'{e}chet derivative $K$ satisfying $\|K\|\le 1$, the following holds:
	\begin{equation}\label{use1}
	\sup_{0\le \lambda \le 1}\left|\lambda^\gamma e^{-\lambda t} \right|  \le \frac{c_\gamma}{(1+t)^\gamma},~~\lambda \in (0, \|K^*K\|],
	\end{equation}
	where $\lambda\geq0$ and $c_\gamma = \max\left\{\gamma^\gamma,1\right\}$.
\end{lemma}

\begin{lemma}\label{lem1}
	Suppose that $k+j >1$. Then, the following holds:
	\begin{equation}\label{A1}
	\int_0^t \frac{ds}{(1+t-s)^k (1+s)^j} < \frac{2^{k+j}-2}{k+j-1} \cdot \frac{1}{(1+t)^{k+j-1}}.
	\end{equation}
\end{lemma}

We now bound the bias error $\|\mathbb{E}{x^\delta }( t ) - {x^\dag }\|$ in a weighted norm. First, we let
\begin{equation*}
\begin{split}
& K=F'(x^\dag),~~V=K^*K,~~V^{\frac{1}{2}}=K,\\
& {I_1} = V^\sigma {e^{ - Vt}}\left( {\bar x - {x^\dag }} \right), \quad
{I_2} = \int_0^t {{e^{ - V\left( {t- s} \right)}}} V^{\sigma+\frac{1}{2}}\left( {{y^\delta } - y} \right)\textmd{ds},\\
& {I_3} =\mathbb{E}\int_0^t {{e^{ - V\left( {t- s} \right)}}} V^{\sigma+\frac{1}{2}}\psi \left( s \right)\textmd{ds},
\end{split}
\end{equation*}
where
\begin{equation}\label{psi}
\psi \left( s \right)\!: = \!\left( {I \!- \!R_{{x^\delta }\left( s \right)}^*} \right)\left[ {F\left( {{x^\delta }\left( s \right)} \right) \!- \!{y^\delta }} \right] - \left[ {F\left( {{x^\delta }\left( s \right)} \right) \!-\! y \!-\! V^\frac{1}{2} \left( {{x^\delta }\left( s \right) \!-\! {x^\dag }} \right)} \right].
\end{equation}
In addition, we let
\begin{equation}
\label{pq}
{p}\left( t \right): = \mathbb{E}[\| {{x^\delta }\left( t \right) - {x^\dag }}\|^2]^\frac{1}{2}, \quad
{q}\left( t \right): = \mathbb{E}[\| {V^\frac{1}{2}\left( {{x^\delta }\left( t \right) - {x^\dag }} \right)} \|^2]^\frac{1}{2}.
\end{equation}

\begin{prop}
	\label{PropositionBiasErr}
	Let Assumptions \ref{assumption-1} and \ref{assumption-2} be satisfied. Choose $\epsilon_0\in( \max\{\frac{2-\sqrt{2}}{\eta}-(2+\sqrt{2}),0\}, 2(\frac{1}{\eta}-1) )$.  Then, for any $\sigma \geq0$ , the following holds:
	\begin{equation*}
	\left\| {\mathbb{E}(V^\sigma({x^\delta }\left( t \right) C {x^\dag }))}\right\| \le  \frac{c^{a}_\gamma E}{{{{\left( {1 + t} \right)}^{\gamma+\sigma} }}} + {c_1}q\left( t \right)(1+t)^{\frac{1}{2}-\sigma}  + {c_2}\int_0^t {\frac{{p\left( s \right)q\left( s \right)}}{(1+t-s)^{\sigma+\frac{1}{2}}}} \textmd{ds},
	\end{equation*}
	where $c^{a}_\gamma=\max\{(\gamma+\sigma)^{\gamma+\sigma}, 1\}$, $c^b_\gamma=\max\{(\gamma+1/2)^{\gamma+1/2}, 1\}$, $c_1=\frac{\sqrt{2}c^b_\gamma}{(1-\eta)\sqrt{\tau^2-2}}$ and $c_2=c^b_\gamma c_R \left( {{\frac{\sqrt{2}\tau }{\sqrt{\left( {\tau^2  - 2} \right)} \left( {1 - \eta } \right)} }+\frac{1}{2}} \right)$.
\end{prop}

The proof is technical, and is provided in Appendix B.

\subsection{Estimation of the variance error}
We now study the upper bound of the variance error $\mathbb{E}[\|V^\sigma({x^\delta }\left( t \right) - \mathbb{E}{x^\delta }\left( t \right))\|^2]$. The cases $\sigma=0$ and $\sigma=\frac{1}{2}$ will be employed in deriving the desired error bound of the SAR method. In addition, from this point we assume that $B_s$ is a $\mathcal{Q}$-Wiener process for $\mathcal{Q}:=I_\mathcal{X}$, where $I_\mathcal{X}$ is the identity operator on $\mathcal{X}$. All assertions below remain valid for a general $\mathcal{Q}$ with appropriate weighted norm; nevertheless, the usage of this assumption can simplify the statement in the proofs.

\begin{prop}\label{po2}
	Let the assumptions of Proposition \ref{PropositionBiasErr} be satisfied. Suppose that $g(t)\leq g(0)/\sqrt{1+t}$. Then, for any $\sigma \geq 0$, the following holds:
	\begin{equation}\label{v2}
	\mathbb{E}[\|V^\sigma({x^\delta }\left( t \right) - \mathbb{E}{x^\delta }\left( t \right))\|^2]^{\frac{1}{2}}\le \sqrt{\frac{4^{\gamma+\sigma} -1}{\gamma+\sigma}} c_3c^b_\gamma \frac{E}{(1+t)^{\gamma+\sigma}},
	\end{equation}
	where $c_3=\sqrt{\frac{2\tau^2\epsilon_0 \eta}{(1-\eta)^2(\tau^2-2)}}$.
\end{prop}

\begin{proof}
	According to \eqref{variance term}, we have
	\begin{equation*}
	V^\sigma({x^\delta }\left( t \right) - \mathbb{E}{x^\delta }\left( t \right))=\int_0^tV^\sigma{{e^{ - {K^ * }K\left( {t - s} \right)}}} f\left( s \right)\textmd{d}{B_s}.
	\end{equation*}
	From Theorem \ref{Itoiso}, we conclude from the requirement of $f$ that
	\begin{equation*}
	\begin{split}
	\mathbb{E}[\|V^\sigma({x^\delta }\left( t \right) - \mathbb{E}{x^\delta }\left( t \right))\|^2]&=\mathbb{E}\left\|\int_0^t{V^\sigma{e^{ - V\left( {t - s} \right)}}} f\left( s \right)\textmd{d}{B_s}\right\|^2\\
	&=\mathbb{E}\int_0^t\left\|{V^\sigma{e^{ - V\left( {t - s} \right)}}} f\left( s \right)\right\|^2_{\mathcal{L}_2(\mathcal{X}_\mathcal{Q},\mathcal{X})} \textmd{ds} < \infty.
	\end{split}
	\end{equation*}
	Hence, the square of the Hilbert-Schmidt norm
	\[\left\|{V^\sigma{e^{ - V\left( {t - s} \right)}}} f\left( s\right)\right\|^2_{\mathcal{L}_2(\mathcal{X
		}_\mathcal{Q},\mathcal{X})}= \textmd{tr}\left[{V^\sigma{e^{ - V\left( {t - s} \right)}}} f\left( s \right)\left({V^\sigma{e^{ - V\left( {t - s} \right)}}} f\left( s \right)\right)^*\right]\]
	is bounded.
	Recall that $V^\sigma$ with $\sigma \geq 0$ are bounded linear operators and $f(t)\in L^\infty (\mathbb{R}_+)$; then, we additionally obtain
	\begin{equation}
	\label{IneqVar1}
	\mathbb{E}[\|V^\sigma({x^\delta }\left( t \right) - \mathbb{E}{x^\delta }\left( t \right))\|^2]
	\le \mathbb{E}\int_0^t\left\|{V^\sigma{e^{ - V\left( {t - s} \right)}}}\right\|^2\left| f\left( s \right)\right|^2\textmd{ds}.
	\end{equation}
	
	From inequalities \eqref{hg} and \eqref{IneqDataSolu}, we obtain, together with the assumption of $g(t)$,
	\begin{equation*}
	\left| f\left( s \right)\right|^2 \le \frac{\epsilon_0 \eta}{1+s} {\mathbb{E}[\|{F\left( {{x^\delta }\left( s \right)} \right)-{y^\delta }}\|^2]} \le \frac{c^2_3}{1+s} \mathbb{E}[\|{V^{\frac{1}{2}}\left( {{{x^\delta }\left( s \right)} - {x^\dag }} \right)} \|^2].
	\end{equation*}
	By using the monotonic decreasing of $\mathbb{E}[\|{x^\delta }\left( s \right) - {x^\dag }\|^2]$ before the terminating time, we derive
	\begin{equation}
	\label{IneqVar2}
	\left\|{V^\sigma{e^{ - V\left( {t - s} \right)}}}\right\|^2\left| f\left( s \right)\right|^2 \le \frac{c^2_3}{1+s} {\left\|{V^\sigma{e^{ - V\left( {t - s} \right)}}}\left(\bar x -x^\dag\right)\right\|^2}.
	\end{equation}
	To bound the numerator of the right side in the above inequality, we combine \eqref{A2} with Lemma \ref{lemma9}, obtaining
	\begin{equation}
	\label{IneqVar3}
	\begin{split}
	& {\left\|{V^\sigma{e^{ - V\left( {t - s} \right)}}}\left(\bar x -x^\dag\right)\right\|^2} \le {\left\|{{e^{ - V\left( {t - s} \right)}}}V^{\gamma+\sigma}\nu\right\|^2} \\
	& \qquad \le  E^2 {\left|\sup_{0\le \lambda \le 1}\lambda^{\gamma+\sigma} e^{-\lambda (t-s)}\right|^2} \le  \frac{ (c^b_\gamma)^2 E^2}{(1+t-s)^{2\left({\gamma+\sigma}\right)}}.
	\end{split}
	\end{equation}
	Consequently, applying Lemma \ref{lem1}, we derive, together with inequalities \eqref{IneqVar1}-\eqref{IneqVar3},
	\begin{equation*}
	\begin{split}
	& \mathbb{E}[\|V^\sigma({x^\delta }\left( t \right) - \mathbb{E}{x^\delta }\left( t \right))\|^2] \le {c_3^2 (c^b_\gamma)^2  E^2} \int_0^t \frac{1}{(1+t-s)^{2\left({\gamma+\sigma}\right)}(1+s)}\textmd{ds} \\
	& \qquad \le c_3^2 (c^b_\gamma)^2 \frac{4^{\gamma+\sigma} -1}{\gamma+\sigma} \frac{E^2}{(1+t)^{2\left({\gamma+\sigma}\right)}}.
	\end{split}
	\end{equation*}
	The proof is thus concluded by taking the square root of the above inequality.
\end{proof}

\subsection{Convergence rates}
In this subsection, we summarize both bias and variance bounds to derive our main results. We first estimate the functions $p(t)$ and $q(t)$, defined in \eqref{PropositionBiasErr}.

\begin{prop}
	\label{ProfIneqpq}
	Let the assumptions of Proposition \ref{po2} be satisfied. For a sufficiently small number $E$, there exists a constant $c^*:=c(\gamma, \tau, \eta)$ such that the following two estimates hold:
	\begin{equation}\label{pest}
	p(t)\le \frac{c^* E}{(1+t)^\gamma}\equiv z_1(t),
	\end{equation}
	\begin{equation}\label{qest}
	q(t)\le \frac{c^* E}{(1+t)^{\gamma+\frac{1}{2}}}\equiv z_2(t).
	\end{equation}
\end{prop}
The detailed proof is provided in Appendix C.

\begin{lemma}
	For any $\varsigma \in \mathcal{X}$, the interpolation inequality in a stochastic sense holds, i.e. for any $0 \le a \le b$,
	\begin{equation}\label{Einterpolation}
	\|\mathbb{E}[V^a \varsigma ] \| \le \|\mathbb{E}[V^b \varsigma ] \|^\frac{a}{b} \|\mathbb{E}[\varsigma ] \|^{1-\frac{a}{b}}.
	\end{equation}
\end{lemma}

\begin{proof}
	From the interpolation inequality $\|V^a \varsigma  \| \le \|V^b \varsigma  \|^\frac{a}{b} \|\varsigma  \|^{1-\frac{a}{b}}$, we have
	\begin{equation*}
	\|\mathbb{E}[V^a \varsigma ] \| \le \mathbb{E}[\|V^a \varsigma\| ] \le \mathbb{E}[ \|V^b \varsigma  \|^\frac{a}{b} \|\varsigma  \|^{1-\frac{a}{b}}].
	\end{equation*}
	Then, from the H\"{o}lder inequality, the following holds:
	\begin{equation*}
	\begin{split}
	\mathbb{E}[ \|V^b \varsigma  \|^\frac{a}{b} \|\varsigma  \|^{1-\frac{a}{b}}] &\le
	\left(\mathbb{E}[\left(\|V^b \varsigma \|^\frac{a}{b}\right)^\frac{b}{a}]\right)^\frac{a}{b} \left( \mathbb{E}[(\|\varsigma\|^{\frac{b-a}{b}})^\frac{b}{b-a}]\right)^\frac{b-a}{b}\\
	&= \mathbb{E}[\|V^b \varsigma \|]^\frac{a}{b}  \mathbb{E}[\|\varsigma \|]^{1-\frac{a}{b}}.
	\end{split}
	\end{equation*}
	This ends the proof.
\end{proof}

\begin{prop}
	\label{PropRvar}
	Let the assumptions of Proposition \ref{po2} be satisfied. For a sufficiently small number $E$, there exists a constant $c_e^0$ such that
	\begin{equation}\label{Rvar}
	\mathbb{E}[\|x^\delta(t^*)-\mathbb{E}x^\delta(t^*)\|^2] \le c_e^0 E^{\frac{1}{2\gamma+1}} \delta^{\frac{2\gamma}{2\gamma+1}}.
	\end{equation}
\end{prop}

\begin{proof}
	From inequality \eqref{F2}, we derive
	\begin{equation*}
	\begin{split}
	\mathbb{E}[\|F(x^\delta(t))-y^\delta\|^2] &\le 2\mathbb{E}[\|F(x^\delta(t))-F(x^\dag)\|^2] + 2\|y-y^\delta\|^2 \\
	& \le\frac{2}{(1-\eta)^2}\mathbb{E}[\|V^\frac{1}{2}(x^\delta(t)-x^\dag)\|^2] + 2\delta^2.
	\end{split}
	\end{equation*}
	According to the definition of $g(t)$ and It\^o's isometry with identity $\mathcal{Q}$, the following holds:
	\begin{equation*}
	\begin{split}
	& \mathbb{E}[\|\int_0^{t^*} {{e^{ - V\left( {t^* - s} \right)}}} f\left( s \right)\textmd{d}{B_s}\|^2]=\mathbb{E}\int_0^{t^*} \left\|{{e^{ - V\left( {t^* - s} \right)}}} f\left( s \right)\right\|^2\textmd{ds}\\
	& \qquad \le [g(0)]^2 \delta^2 \int_0^{t^*} \frac{\left\|{{e^{ - V\left( {t^* - s} \right)}}}\right\|^2\mathbb{E}[\|F(x^\delta(s))-y^\delta\|^2]}{1+s}\textmd{ds}.
	\end{split}
	\end{equation*}
	By incorporating the two inequalities above, we obtain
	\begin{equation}\label{bb}
	\begin{split}
	\mathbb{E}[\|x^\delta(t^*)-\mathbb{E}x^\delta(t^*)\|^2]&\\
	&\!\!\!\!\!\!\!\!\!\!\!\!\!\!\!\!\!\!\!\!\!\!\!\!\!\!\!\!\!\!\!\!\!\!\!\!\le \frac{2[g(0)]^2 \delta^2}{(1-\eta)^2}\int_0^{t^*} \frac{\left\|{{e^{ - V\left( {t^* - s} \right)}}}\right\|^2\mathbb{E}[\|V^\frac{1}{2}(x^\delta(s)-x^\dag)\|^2] }{1+s}\textmd{ds}\\
	&\!\!\!\!\!\!\!\!\!\!\!\!\!\!\!\!\!\!\!\!\!\!\!\!\!\!\!\!\!\!\!\!+ \frac{4[g(0)]^2 \delta^4}{(1-\eta)^2}\int_0^{t^*} \frac{\left\|{{e^{ - V\left( {t^* - s} \right)}}}\right\|^2 }{1+s}\textmd{ds}.
	\end{split}
	\end{equation}
 
	Considering the first summand of the above inequality \eqref{bb}, it follows from estimate \eqref{pest} and Lemma \ref{lemma9} (with $\gamma=1/2$) that
	\begin{equation*}
	\begin{split}
	& \int_0^{t^*} \frac{\left\|{{e^{ - V\left( {t^* - s} \right)}}}\right\|^2\mathbb{E}[\|V^\frac{1}{2}(x^\delta(s)-x^\dag)\|^2] }{1+s}\textmd{ds} \\
	& \le \int_0^{t^*} \frac{\left\|V^\frac{1}{2}{{e^{ - V\left( {t^* - s} \right)}}}\right\|^2\mathbb{E}[\|x^\delta(s)-x^\dag\|^2] }{1+s}\textmd{ds} \le (c^* E)^2 \int_0^{t^*} \frac{\left\|V^\frac{1}{2}{{e^{ - V\left( {t^* - s} \right)}}}\right\|^2}{(1+s)^{2\gamma+1}}\textmd{ds}\\
	& \le (c^* E)^2 \int_0^{t^*} \frac{\left( \sup_{0\le \lambda \le1} \lambda^{\frac{1}{2}} e^{-\lambda(t^*-s)} \right)^2}{(1+s)^{2\gamma+1}}\textmd{ds}\\
	& \le (c^* E)^2 \int_0^{t^*} \frac{1}{(1+t^*-s)(1+s)^{2\gamma+1}}\textmd{ds}\le \frac{2^{2\gamma+2}-2}{2\gamma+1} \frac{(c^*)^2 E^2}{(1+t^*)^{2\gamma +1}}.
	\end{split}
	\end{equation*}
	Meanwhile, the second summand of \eqref{bb} has the estimate
	\begin{equation*}
	\begin{split}
	\int_0^{t^*} \frac{\left\|{{e^{ - V\left( {t^* - s} \right)}}}\right\|^2 }{1+s}\textmd{ds} \le \int_0^{t^*} \frac{1}{1+s}\textmd{ds} = \ln(1+t^*) \le 1+t^*.
	\end{split}
	\end{equation*}
	Combining the last two inequalities, we derive, together with the fact that $\sqrt{a+b}\le \sqrt{a}+\sqrt{b}$,
	\begin{equation}
	\label{PfIneqEx}
	\begin{split}
	\mathbb{E}[\|x^\delta(t^*)\!-\!\mathbb{E}x^\delta(t^*)\|^2]^\frac{1}{2} & \!\le\! \frac{c^* E g(0)}{1-\eta} \sqrt{\frac{2^{2\gamma+3}\!-\!4}{2\gamma\!+\!1}} \frac{ \delta\sqrt{1\!+\!t^*}}{(1+t^*)^{\gamma +1}} \!+\!  \frac{2g(0)}{1-\eta}  \delta^2 \sqrt{1\!+\!t^*} \\ & \le \left\{ \frac{c^* E g(0)}{1-\eta} \sqrt{\frac{2^{2\gamma+3}-4}{2\gamma+1}} +  \frac{2g(0)\delta_0}{1-\eta} \right\}  \delta \sqrt{1+t^*}.
	\end{split}
	\end{equation}
 
	In addition, it follows from inequalities \eqref{IneqDelta} and \eqref{qest} that
	\begin{equation*}
	\begin{split}
	\delta \le   \sqrt{\frac{2}{(1-\eta)^2(\tau^2-2)}} \mathbb{E}[\|{V^{\frac{1}{2}}\left( {{{x^\delta }\left( t \right)} - {x^\dag }} \right)} \|^2]^\frac{1}{2} \le \sqrt{\frac{2}{(1-\eta)^2(\tau^2-2)}} \frac{c^* E}{(1+t^*)^{\gamma+\frac{1}{2}}},
	\end{split}
	\end{equation*}
	which yields
	\begin{equation}
	\label{rate1}
	\sqrt{1+t^*}\delta \le \left( \frac{2}{(1-\eta)^2(\tau^2-2)} \right)^{\frac{1}{4\gamma+2}} (c^*E)^{\frac{1}{2\gamma+1}} \delta^{\frac{2\gamma}{2\gamma+1}}.
	\end{equation}
	We then combine \eqref{PfIneqEx} and \eqref{rate1} to obtain estimate \eqref{Rvar}, with
	\begin{equation}
	\label{ce0}
	c_e^0 = \left\{ \frac{c^* E g(0)}{1-\eta} \sqrt{\frac{2^{2\gamma+3}-4}{2\gamma+1}} +  \frac{2g(0)\delta_0}{1-\eta} \right\} \left( \frac{2(c^*)^2}{(1-\eta)^2(\tau^2-2)} \right)^{\frac{1}{4\gamma+2}}.
	\end{equation}
\end{proof}

We are now able to show the order-optimal error bounds for the mean-square total regularization error $\mathbb{E}[\|{x^\delta }\left( t^* \right) - {x^\dag }\|^2]$ with $t^*$ chosen from \eqref{stopping}, which is the main result of this section.

\begin{theorem}\label{thmrate}
	Let the assumptions of Proposition \ref{po2} be satisfied. For a sufficiently small number $E$, there exists a constant $c_e$ such that
	\begin{equation}\label{convrate}
	\mathbb{E}[\|{x^\delta }\left( t^* \right) - {x^\dag }\|^2]^\frac{1}{2} \le c_{e} E^{\frac{1}{2\gamma+1}} \delta^{\frac{2\gamma}{2\gamma+1}}.
	\end{equation}
\end{theorem}

\begin{proof}
	From \eqref{total error}, it follows that
	\begin{equation}\label{terror}
	\begin{split}
	{x^\delta }\left( t^* \right) \!- \!{x^\dag } \!=\! V^{\gamma} \nu^* \!+\! \int_0^{t^*}{{e^{ - V\left( {t^* - s} \right)}}} V^\frac{1}{2}\left( {{y^\delta } \!-\! y} \right)\textmd{ds}\!+\! \int_0^{t^*} {{e^{ -\! V\left( {t^* - s} \right)}}} f\left( s \right)\textmd{d}{B_s},
	\end{split}
	\end{equation}
	where
	\begin{equation*}
	\nu^*:= e^{-Vt^*}\nu+ \int_0^{t^*} {{e^{ - V\left( {t^* - s} \right)}}}V^{\frac{1}{2}-\gamma} \psi \left( s \right)\textmd{ds},
	\end{equation*}
	and $\psi (s)$ is given by \eqref{psi}.
	Then, according to Lemma \ref{lemma9}, the following holds:
	\begin{equation*}
	\begin{split}
	\left\|\nu^*\right\| &\le \left\|\nu\right\|+ \int_0^{t^*} \left\|{{e^{ - V\left( {t^* - s} \right)}}}V^{\frac{1}{2}-\gamma} \psi \left( s \right)\right\|\textmd{ds}\le E+ \int_0^{t^*} \frac{\left\|\psi(s)\right\|}{(1+t^*-s)^{\frac{1}{2}-\gamma}}\textmd{ds}.
	\end{split}
	\end{equation*}
	In addition, via inequality \eqref{cc}, Proposition \ref{ProfIneqpq}, and Lemma \ref{lem1}, we deduce that
	\begin{equation*}
	\begin{split}
	&\mathbb{E}\int_0^{t^*} \frac{\left\|\psi(s)\right\|}{(1+t^*-s)^{\frac{1}{2}-\gamma}}\textmd{ds} \le c_0 \int_0^{t^*} \frac{p(s)q(s)}{(1+t^*-s)^{\frac{1}{2}-\gamma}}\textmd{ds}\\
	&\le c_0 (c^*)^2 E^2\int_0^{t^*} \frac{\textmd{ds}}{(1+t^*-s)^{\frac{1}{2}-\gamma}(1+s)^{2\gamma+\frac{1}{2}}}\le  c_0 (c^*)^2 E^2 \frac{2^{\gamma+1}-2}{\gamma}.
	\end{split}
	\end{equation*}
	Consequently, we have
	\begin{equation}
	\label{c1e}
	\mathbb{E}[\|\nu^*\|] \le E+ c_0 (c^*)^2 E^2 \frac{2^{\gamma+1}-2}{\gamma} =: c^1_{e} E.
	\end{equation}
	
	Next, we bound the three terms of the right side of \eqref{terror} separately.
	Let
	\begin{equation*}
	J:= \int_0^{t^*}{{e^{ - V\left( {t^* - s} \right)}}} V^\frac{1}{2}\left( {{y^\delta } - y} \right)\textmd{ds}.
	\end{equation*}
	Using Lemma \ref{lemma9} again ($\gamma=1/2$ implies $c_\gamma=1$), we obtain
	\begin{equation*}
	\begin{split}
	\left\|J\right\| \le \delta \int_0^{t^*} \left\|\sup_{0\le \lambda \le 1} \lambda^\frac{1}{2}{{e^{ - \lambda \left( {t^* - s} \right)}}} \right\|\textmd{ds}\le \delta \int_0^{t^*} \frac{1}{\sqrt{1+t^*-s}}\textmd{ds}< \sqrt{t^*+1} \delta.
	\end{split}
	\end{equation*}
	In addition, it follows from \eqref{rate1} that
	\begin{equation}
	\label{Jnorm}
	\left\|J\right\| \le  \left( \frac{2(c^*)^2}{(1-\eta)^2(\tau^2-2)} \right)^{\frac{1}{4\gamma+2}} E^{\frac{1}{2\gamma+1}} \delta^{\frac{2\gamma}{2\gamma+1}}.
	\end{equation}
	
	Taking $\mathbb{E}[V^{\frac{1}{2}}(\cdot)]$ on both sides of \eqref{terror} and rearranging terms, we obtain
	\begin{equation*}
	\begin{split}
	\mathbb{E}[V^{\gamma+\frac{1}{2}}\nu^*] =\mathbb{E}[V^{\frac{1}{2}}({x^\delta }\left( t^* \right) - {x^\dag })]- \int_0^{t^*}{{e^{ - V\left( {t^* - s} \right)}}} V\left( {{y^\delta } - y} \right)\textmd{ds}.
	\end{split}
	\end{equation*}
	Consequently, from assumption \eqref{F2} and stopping rule \eqref{stopping}, we obtain
	\begin{equation}
	\label{c2e}
	\begin{split}
	\|\mathbb{E}[V^{\gamma+\frac{1}{2}}\nu^*]\|
	&\le\mathbb{E}[\|V^{\frac{1}{2}}({x^\delta }\left( t^* \right) - {x^\dag })\|]+\int_0^{t^*} \|{{e^{ - V\left( {t^* - s} \right)}}} V\left( {{y^\delta } - y} \right)\|\textmd{ds}\\
	&\le(1+\eta)\mathbb{E}[\|F({x^\delta }\left( t^* \right)) - F({x^\dag })\|]+\sup_{0\le\lambda \le 1}(1-e^{-\lambda t^*})\cdot \delta\\
	&\le(1+\eta)\left(\mathbb{E}[\|F({x^\delta }\left( t^* \right)) - y^\delta\|^2]^\frac{1}{2}+\delta\right)+\delta\\
	&\le(1+\eta)(\tau\delta + \delta) + \delta =: c_e^2 \delta.
	\end{split}
	\end{equation}
	From the interpolation inequality in a stochastic sense \eqref{Einterpolation}
	with $a=\gamma$, $b=\gamma + \frac{1}{2}$  and $\varsigma=\nu^*$, we additionally obtain
	\begin{equation}\label{rate2}
	\begin{split}
	\|\mathbb{E}[V^{\gamma}\nu^*]\| &\le \mathbb{E}[\|V^{\gamma+\frac{1}{2}} \nu^* \|]^\frac{2\gamma}{2\gamma+1}  \mathbb{E}[\| \nu^* \|]^{\frac{1}{2\gamma+1}}\\
	&\le (c_e^1)^{\frac{1}{2\gamma+1}}(c_e^2)^\frac{2\gamma}{2\gamma+1}E^{\frac{1}{2\gamma+1}} \delta^\frac{2\gamma}{2\gamma+1} =: c_e^3 E^{\frac{1}{2\gamma+1}} \delta^\frac{2\gamma}{2\gamma+1}.
	\end{split}
	\end{equation}
	Combining \eqref{Jnorm} and \eqref{rate2} then leads to
	\begin{equation}
	\label{ce4}
	\|\mathbb{E}{x^\delta }\left( t^* \right) - {x^\dag }\| = \|\mathbb{E}[ V^{\gamma} \nu^* ]+ J\| \le \|\mathbb{E}[ V^{\gamma} \nu^* ]\|+ \left\|J\right\| \le c_e^4 E^{\frac{1}{2\gamma+1}} \delta^{\frac{2\gamma}{2\gamma+1}},
	\end{equation}
	where $c_e^4 = \left( \frac{2(c^*)^2}{(1-\eta)^2(\tau^2-2)} \right)^{\frac{1}{4\gamma+2}} +c_e^3$.
	
	Finally, applying the above inequality and Proposition \ref{PropRvar}, we conclude that
	\begin{equation}
	\label{c_e}
	\begin{split}
	&\mathbb{E}[\|{x^\delta }\left( t^* \right) - {x^\dag }\|^2] = \|\mathbb{E}{x^\delta }\left( t^* \right) - {x^\dag }\|^2 + \mathbb{E}[\|{x^\delta }\left( t^* \right) - \mathbb{E}{x^\delta }\left( t^* \right)\|^2]\\
	&\le \left[c_e^4 E^{\frac{1}{2\gamma+1}} \delta^{\frac{2\gamma}{2\gamma+1}}\right]^2+\left[ c_e^0 E^{\frac{1}{2\gamma+1}} \delta^{\frac{2\gamma}{2\gamma+1}}\right]^2
	= \left( (c_e^0)^2 + (c_e^4)^2 \right) E^{\frac{2}{2\gamma+1}} \delta^{\frac{4\gamma}{2\gamma+1}},
	\end{split}
	\end{equation}
	which yields the estimate with $c_e=\sqrt{(c_e^0)^2 + (c_e^4)^2}$.
\end{proof}

\section{Numerical illustrations}
\label{simulation}

The main algorithms and some simulation results are shared by the first author at \url{https://github.com/Haie555/SAR-method}. All computations were performed using MATLAB R2021a on an HP workstation with an Intel Core i7-8700 CPU running at 3.20 GHz and 16 GB of RAM.

\subsection{Model problem 1: a parameter-identification problem in PDEs}
\label{Ex1}

This group of simulations present some numerical experiments on a benchmark parameter-identification problem, with one-dimensional (1D) and two-dimensional (2D) elliptic equations to study the performance of SAR.
More specifically, we consider the following elliptic equation:
\begin{equation}\label{elliptic equation}
\begin{split}
- \Delta {u} + c{u} &= w~~\textmd{in}~~\Omega,\\
\frac {{\partial{u}}} {\partial{n}} &= 0~~\textmd{on}~~\partial\Omega,
\end{split}
\end{equation}
where $\Omega\subset \mathbb{R}^{d}$ ($d=1, 2$) is an open bounded domain with a Lipschitz boundary and $w\in L^2{(\Omega)}$ is the source term. The considered inverse problem, associated with PDE \eqref{elliptic equation} can be described as
\begin{equation}
\label{model}
{F}\left( c \right) = {u}\left( c \right)
\end{equation}
by defining the parameter-to-solution maps as $F:c\longmapsto u$, where $u(c)$ is the solution of \eqref{elliptic equation}.
Let
\[D\left( F \right): = \left\{ {c \in {L^2}\left( \Omega  \right):{{\left\| {c - \hat c} \right\|}_{{L^2}\left( \Omega  \right)}} \le {\zeta _0} ~\textmd{for some} ~\hat c \ge 0,~a.e.} \right\}\] be the admissible set of $F$. Under the above settings, $F$ is well defined for some positive constant ${\zeta_0}$, and is also Fr\'{e}chet-differentiable. Furthermore, for any $q,\omega \in {L^2}\left( \Omega  \right)$, the Fr\'{e}chet derivative of $F$ and its adjoint can be calculated as
\begin{equation*}
F'\left( c \right)q =  - A{\left( c \right)^{ - 1}}\left( {qF\left( c \right)} \right),\quad
F'{\left( c \right)^*}\omega =  - u\left( c \right)A{\left( c \right)^{ - 1}}\omega,
\end{equation*}
where $A\left( c \right):{H^2} \cap D\left( F \right) \to {L^2}$ is defined by $A\left( c \right)u =  - \Delta u + cu$; see \cite{Hanke1995} for more detailed theoretical aspects of this inverse problem.

To use SAR in practice, we need to discretize the stochastic flow \eqref{initial value problem} with respect to the artificial time variable $t$. Numerous algorithms have been proposed (see, e.g., \cite{JentzenKloeden2009, LordPowell2014}) for accurate numerical approximation of stochastic differential equations. However, for simplicity, we focus only on the Euler method here, which coincides with the stochastic gradient descent algorithm mentioned, as seen in formula \eqref{LandweberS}, i.e. 
\begin{eqnarray}
\label{EM}
x^\delta_{k+1} = x^\delta_{k} + \Delta t F'{\left( x^\delta_{k} \right)^ *} ( y^\delta- F(x_k^\delta) ) + f_k \Delta B_k, 
\end{eqnarray}
where $f_k=f(t_k)$, $\Delta t$ is the size of the uniform time step, and $\Delta B_k = \sqrt{\Delta t} \xi_k$ with $\{\xi_k\}$ being a sequence of independent, standard, normally distributed random variables, i.e. its $i$-th component $[\xi_k]_i \sim N(0,1)$. Using the definitions of $f(t)$ and $g(t)$ in \eqref{ft} and in the conditions of theorems, we define $f_k= \delta r_k s_k $ with $r_k= \min (\|F(x^\delta_k)-y^\delta\|, \|F(x_0)-y^\delta\|)$ and
\begin{equation*}
s_k= {\frac{\theta}{\sqrt{1+t_k}}},
\end{equation*}
where $\theta$ is employed to describe the level of randomization, which ranges between $0$ and $\frac{\sqrt{\epsilon_0 \eta}}{\delta_0}$. The higher the value of $\theta$, the more randomization is assumed. For $\theta=0$, there will be no randomization, and this happens to be the conventional asymptotical regularization method as introduced in \cite{UT1994}.

The regularization properties of scheme \eqref{EM} can be performed as described in the theoretical part of this paper. It is beyond the scope of this paper, since as shown in \cite{Rieder-2005}, any appropriate numerical scheme of stochastic flow \eqref{initial value problem} can serve as an efficient regularization algorithm for some special nonlinear inverse problems.

\subsubsection{Settings for simulations}

We consider the following two examples:

\begin{itemize}
  \item[$\centerdot$] \textbf{1D example}: $\Omega=[-1,1]$, ${w}\left( c \right) =1$ and ${c^\dag }:=  2 + \cos(\pi c)$.
  \item[$\centerdot$] \textbf{2D example}: $\Omega  = {[ - 1,1]^2}$, ${w}\left( c_1, c_2 \right) =1$, and $
{c^\dag }\left( {c_1, c_2} \right):= 1 + \sin \left( \pi{ c_1} \right)+\cos \left( { c_2} \right)$. 
\end{itemize}

Unless otherwise stated, these examples are discretized with a dimensionality $n=256$ for the 1D case and $n^2=128^2=16384$ for 2D. The noisy data $u^\delta$ is generated from the exact data $u$ as follows:
\begin{equation}
u_j^\delta=u_j+\delta\max_i (|u_i|)\xi_j, ~~j=1,\cdots,n,
\end{equation}
where $\delta$ is the noise level, and the random variables $\xi_j$ follow the standard Gaussian distribution. To ensure a fair evaluation of our method, the initial guess 
$x_0$ is randomly selected in each trial of the numerical experiments.

It is crucial to note that the level of randomization, the parameter $\theta$, and the sample size significantly impact the behavior of SAR. A detailed numerical illustration of this effect can be found in Appendix E, along with Tables \ref{tab1} and \ref{tab2} below. With that in mind, let us investigate the advantages of SAR.

\begin{table}[h]
	\caption{Comparisons for different level of randomization $\theta$ in 1-D case, with $\delta=2\%$, and $\tau=1.3$.}\label{tab1}
	\begin{tabular*}{\textwidth}{@{\extracolsep\fill}lcccccc}
		\toprule%
		$\theta$ & $N$\footnotemark[1] &  MinNorm\footnotemark[2] & MeanNorm\footnotemark[3] & $\mathbb{E}(n_*)$\footnotemark[4] & {Min}$(n_*)$\footnotemark[5] & {Med}$(n_*)$\footnotemark[6] \\
		\midrule
		\multirow{4}{*}{0}   
		&100 & 18.2317 & 18.2317 & 134.00 & 134 & 134  \\ 
		& 500 & 18.6904  & 18.6904 & 94.00 & 94 & 94  \\
        & 1000 & 17.6631 & 17.6631 & 81.00 & 81 & 81  \\
		& 3000 & 17.2430 & 17.2430 & 87.00 & 87 & 87 \\
		\midrule
        \multirow{4}{*}{0.1}   
		&100 & 1.3990 & 2.3719 & 26.46 & 12 & 26  \\ 
		&500 & 1.1380 & 2.3624 & 27.12 & 12 & 26  \\ 
		& 1000 & 1.1028  & 2.5217 & 27.24 & 12 & 26  \\
		& 3000 & 1.0931 & 2.3967 & 26.92 & 12 & 26  \\
		\midrule
		\multirow{4}{*}{0.5}
		&100 & 1.3522 & 2.4792  & 16.11 & 12 & 15  \\  
		&500 & 1.2564  & 2.3608 & 15.49 & 12 & 15  \\ 
		& 1000 & 1.1861  & 2.3687 & 15.85 & 12 & 15 \\
		& 3000 & 1.1377 & 2.4049 & 15.68 & 12 & 15  \\
		\midrule
		\multirow{4}{*}{1.0}  
		&100 & 1.4628 & 2.3746 & 26.38 & 14 & 26  \\ 
		&500 & 1.2576 & 2.3639 & 26.27 & 13 & 26  \\ 
		& 1000 & 1.1087  & 2.4969 & 25.93 & 14 & 26  \\
		& 3000 & 1.0631 & 2.4808 & 26.29 & 13 & 26  \\
  	\midrule
		\multirow{4}{*}{3.0} 
		&100 & 1.5505 & 2.7528 & 26.90 & 14 & 25  \\ 
		&500 & 1.5065 & 2.6846 & 27.18 & 14 & 26  \\ 
		& 1000 & 1.4454  & 2.6017 & 27.47 & 14 & 26  \\
		& 3000 & 1.3902 & 2.7037 & 27.78 & 14 & 27  \\
		\botrule
	\end{tabular*}
    \footnotetext[1]{The number of sample sizes used in the experiment.}
	\footnotetext[2]{$\texttt{MinNorm}=\mathop {\min }\limits_N \|x_{n_*}^{\texttt{SAR}}-x^\dag\|$.}
	\footnotetext[3]{$\texttt{MeanNorm}= \mathbb{E} \|x_{n_*}^{\texttt{SAR}}-x^\dag\|$.}
	\footnotetext[4]{The expectation of the number of iterations among all $N$ samples when the stopping rule is met.}
	\footnotetext[5]{The minimum number of iterations among all $N$ samples when the stopping rule is met.}
	\footnotetext[6]{The median number of iterations across all $N$ samples when the stopping rule is met.}
 \end{table}
\begin{table}[h]
	\caption{Probability comparisons P($\lambda$) = $\mathbb{P}(\| x_{n_*}^{\texttt{SAR}}-x^\dag\|\leq \lambda \cdot \| x_{n_*}^{\texttt{Land}}-x^\dag\|)$ for different $\lambda$ in 1-D case, and $\tau=1.3$.}\label{tab2}
	\begin{tabular*}{\textwidth}{@{\extracolsep\fill}lccccc}
		\toprule%
		$\theta$ & $N$ &P(0.8)&P(0.5)&P(0.3)&P(0.1)\\
		\midrule
		\multirow{4}{*}{0.1}   
       & 100  & 100$\%$ &100$\%$ & 100$\%$ &  6$\%$  \\
       &500 & 100$\%$ & 100$\%$ & 98.40$\%$ & 19.20$\%$  \\
       &1000 & 100$\%$ & 100$\%$ & 99.20$\%$ & 16.50$\%$  \\
       & 3000 & 99.97$\% $ & 99.93$\%$ & 99.33$\%$ & 12.30$\%$ \\
       \midrule
		\multirow{4}{*}{0.5}
       &100 & 100$\%$ & 100$\%$ & 100$\%$ & 17$\%$  \\
       &500 & 99.80$\%$ & 99.80$\%$ &99.40$\%$ & 19.80$\%$  \\
       & 1000 & 100$\% $ & 99.90$\%$ & 99$\%$ & 25.30$\%$   \\
       & 3000 & 100$\% $ & 100$\%$ & 98.97$\%$ & 20.83$\%$   \\
		\midrule
		\multirow{4}{*}{1.0} 
       & 100  & 100$\%$ & 100$\%$ & 99$\%$ & 29$\%$ \\
       &500 & 100$\%$ & 99.80$\%$ & 98.60$\%$ & 15.40$\%$  \\
       &1000 & 100$\%$ & 100$\%$ & 99.60$\%$ & 18$\%$  \\
       & 3000 & 100$\% $ & 99.90$\%$ & 99.27$\%$ & 15.50$\%$  \\
       \midrule
		\multirow{4}{*}{3.0} 
       & 100  & 100$\%$ & 100$\%$ & 98$\%$ & 4$\%$ \\
       &500 & 100$\%$ & 100$\%$ & 99$\%$ & 2$\%$  \\
       &1000 & 100$\%$ & 99.80$\%$ & 99.20$\%$ & 6$\%$  \\
       & 3000 & 99.97$\% $ & 99.93$\%$ & 99.13$\%$ & 6.93$\%$   \\
		\botrule
	\end{tabular*}
\end{table}

\subsubsection{Advantage I: Escaping from local minima by selecting the optimal path}

In this subsection, we present numerical results to demonstrate the first advantage of SAR: its ability to escape from local minima, which are often encountered by its non-stochastic counterpart, the Landweber iteration. It is worth noting that a similar iterative scheme with a specifically designed random term, which also decreases to zero during the iteration, was recently proposed in the field of numerical optimization by \cite{Engquist2022}. The authors demonstrated the quantitative convergence of their algorithm in finding the global minimizer of a highly nonconvex function, such as the Rastrigin function. However, this designed random term performs poorly in our parameter-identification problem due to the unknown complex structure of the forward operator $F$ in \eqref{model}.

Now, let us examine our numerical results. Figure \ref{1D-bestpath} shows that both the relative error (Figure \ref{d1p100}) and the residual error (Figure \ref{d2p100}) decrease significantly when an appropriate level of randomization is introduced in the SAR method. In contrast, the Landweber iteration ($\theta=0$) exhibits slower convergence. Additionally, while the residual errors generated by SAR are nearly the same as those produced by its non-stochastic counterpart when the stopping criterion is met, SAR requires significantly fewer iterations. Furthermore, the flat trend observed in the relative and residual errors, along with the reconstructed result of the Landweber method in Figure \ref{d1p100Landweber}, indicates that the Landweber iteration becomes trapped in a local minimum. It is important to note that these results were selected from the best path sample out of $N=500$ for both the Landweber method and SAR.

\begin{figure}[t]
	\centering
	{
		\subfigure[]{
			\label{d1p100}
			\includegraphics[scale=0.06]{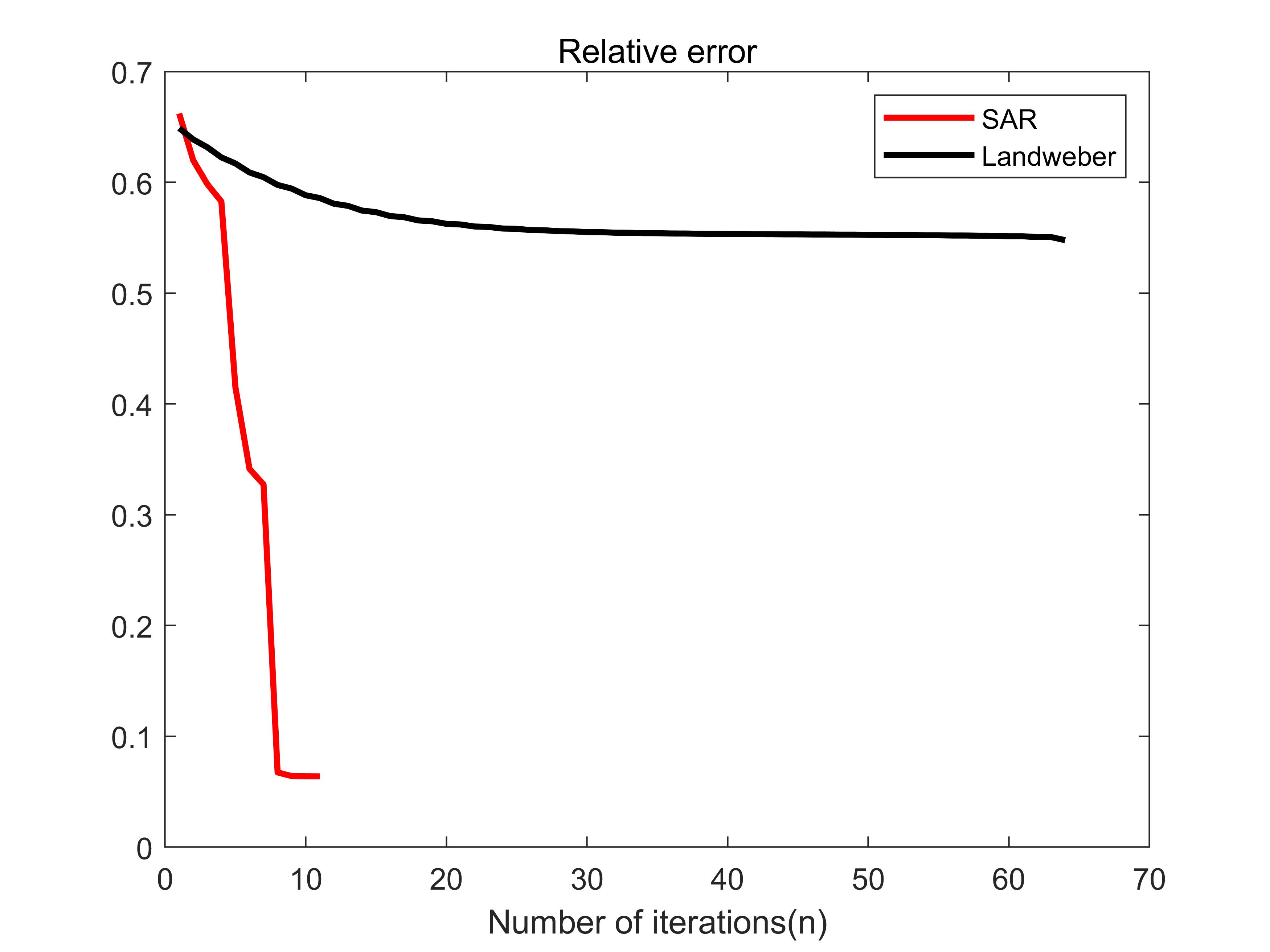} }
		\subfigure[]{
			\label{d2p100}
			\includegraphics[scale=0.06]{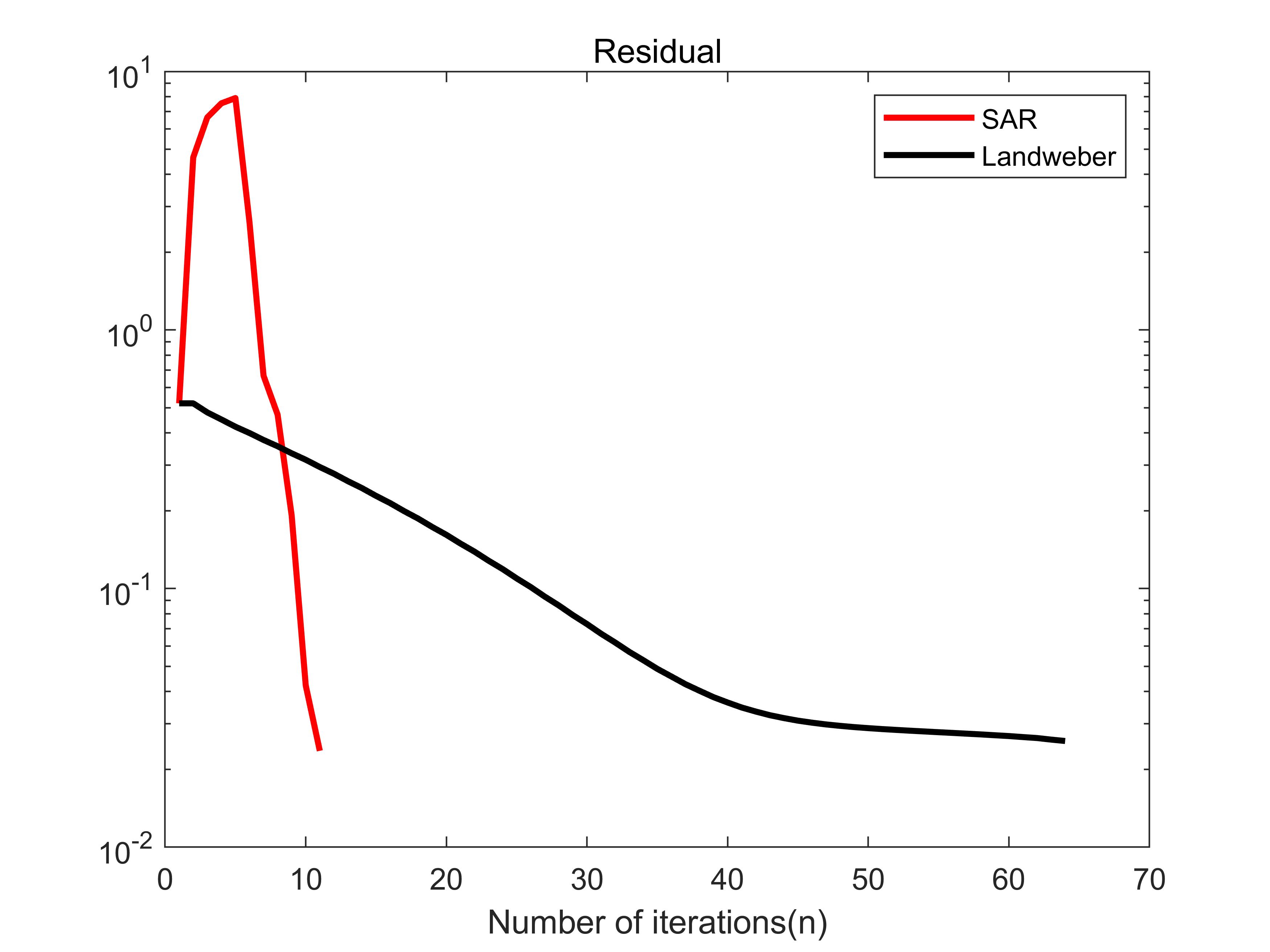} }
		\subfigure[]{
			\label{d1p100Landweber}
			\includegraphics[scale=0.06]{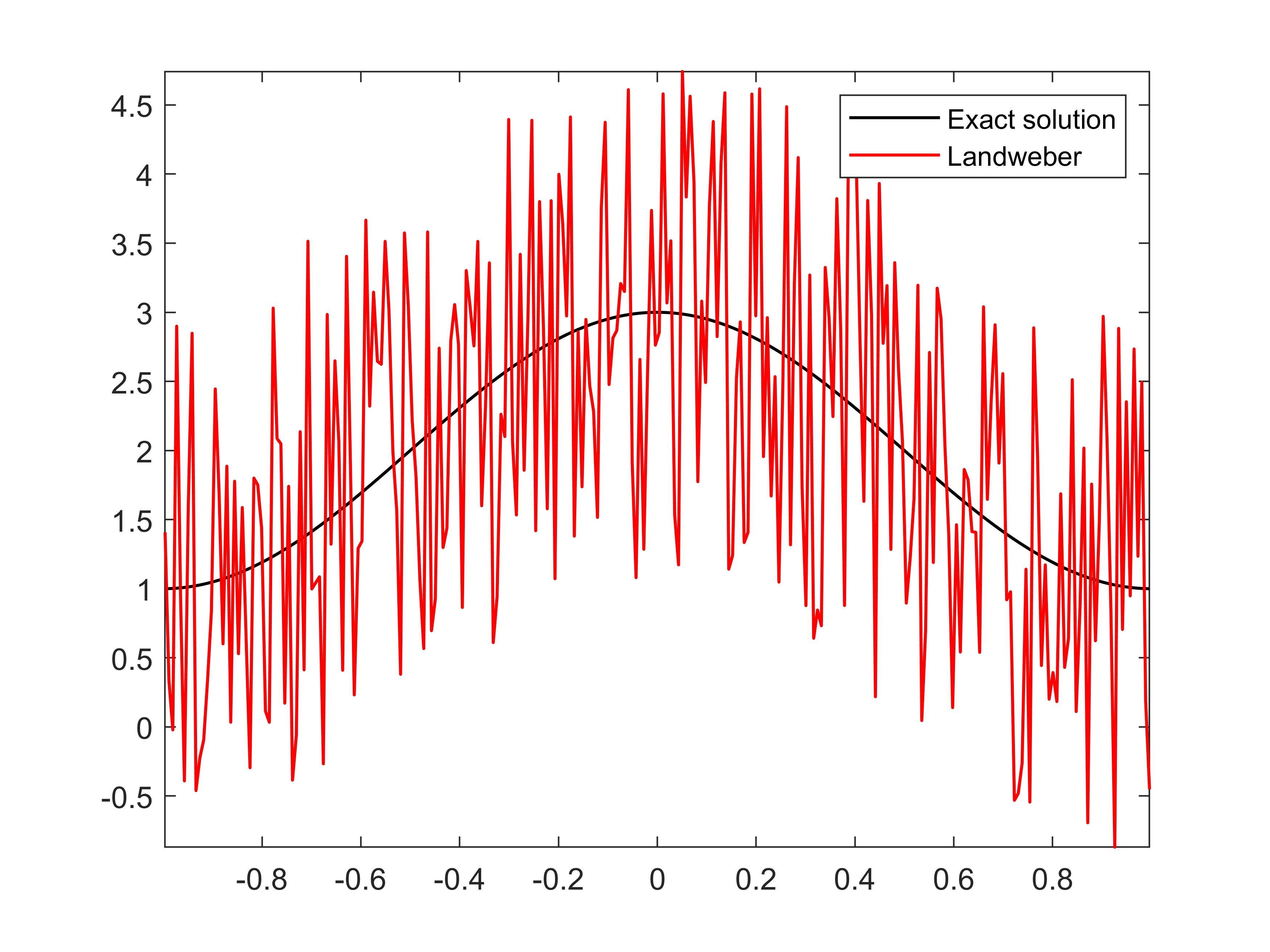} }
		\subfigure[]{
			\label{d2p100SAR}
			\includegraphics[scale=0.06]{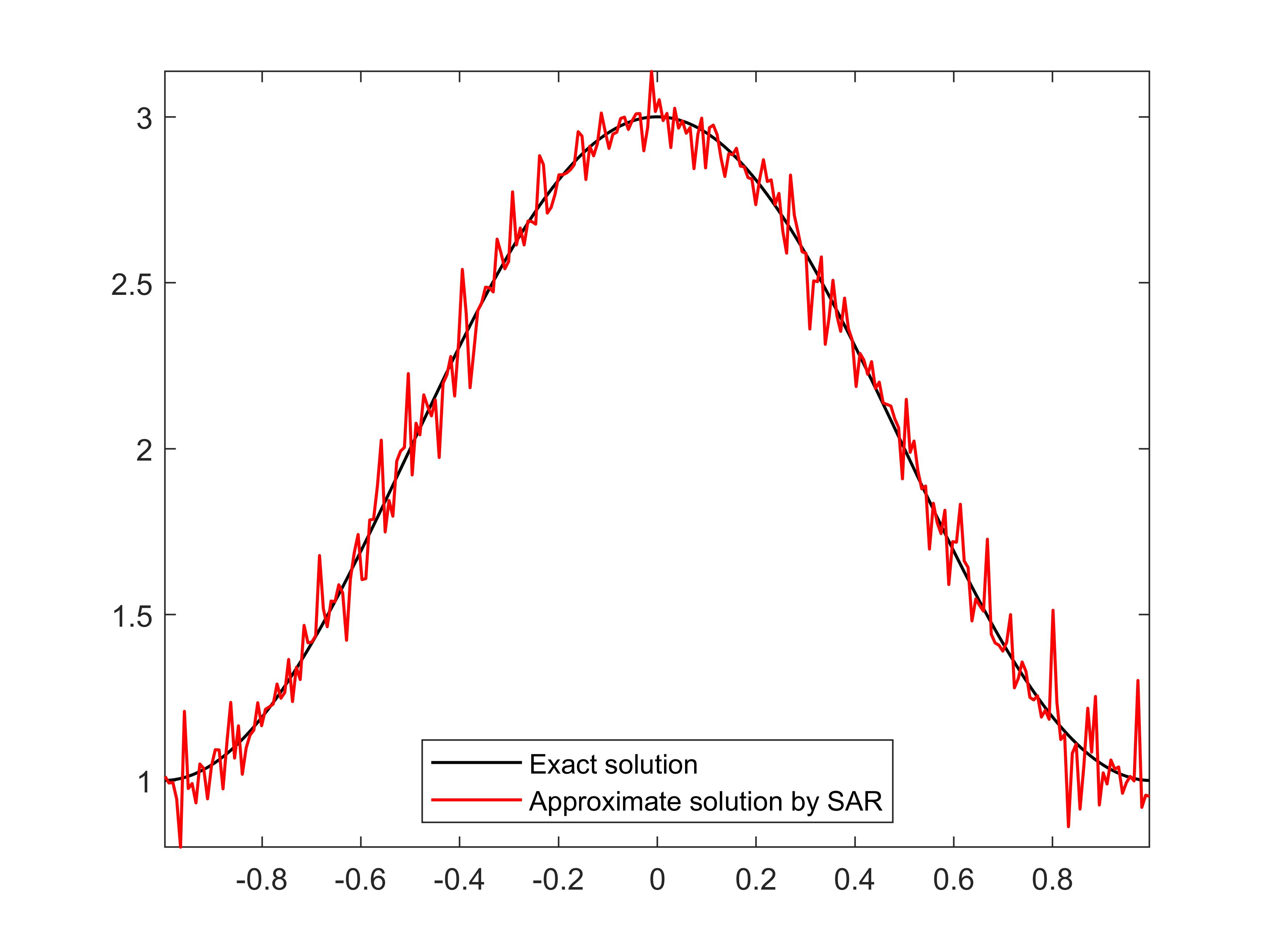} }
 }
	\caption{Results for the 1D case, with $\delta = 2\%$, $N=500$, $\theta=0.5$. (a) Relative errors; (b) Semi-log plot of residual errors; (c) Approximate solution by Landweber; (d) Approximate solution by SAR. 
	}
	\label{1D-bestpath}
\end{figure}


Additionally, Table \ref{tab1} shows that with an appropriate level of randomization and a suitable sample size, the SAR method significantly improves the inversion results compared to the case of $\theta = 0$ (Landweber method). All relevant quantities, such as MinNorm, MeanNorm, $\mathbb{E}(n_*)$,  {Min}$(n_*)$ and {Med}$(n_*)$, are noticeably reduced.

Furthermore, Table \ref{tab2} allows us to assess the performance of the SAR method across different levels of randomization $\theta$ and various sample sizes $N$. The table compares the probabilities $P(\lambda)$ for
different values of $\lambda$ ($\lambda= 0.8, 0.5, 0.3, 0.1$). The results indicate that, in general, the SAR method (with varying $\theta$) outperforms the Landweber method 70\% of the accuracy. With an appropriate choice of randomization (e.g. $\theta=0.5$), the accuracy of SAR is up to 10 times better than that of the Landweber method in 70\% of the cases.



Figure \ref{addfig1} summarizes the estimated probability of convergence with respect to the number of iterations, comparing SAR across different noise levels in the measurements. The figure shows that the iterate $x_n$ approaches the true solution $x^\dagger$ (the global minimum of the noise-free least squares problem for \eqref{model}, i.e., $x^\dagger=\arg\min \|F(c) - u\|^2$)  with high probability, even under varying noise levels. Quantitatively, using SAR, the probability $P(\|x_n - x^\dag\| \leq \delta)$ is approximately 99.8\% for $\delta\in [1\%, 5\%]$, while for the classical Landweber method, $P(\|x_n - x^\dag\|\leq \delta)$ is approximately 0, regardless of the number of iterations. These statistics are based on 500 independent runs, with the starting point uniformly sampled as $0.01 + 4 \cdot \text{rand}$.  

\begin{figure}[!htb]
	\centering
	{
		\subfigure{\label{2Derrd1s002N}
			\includegraphics[scale=0.08]{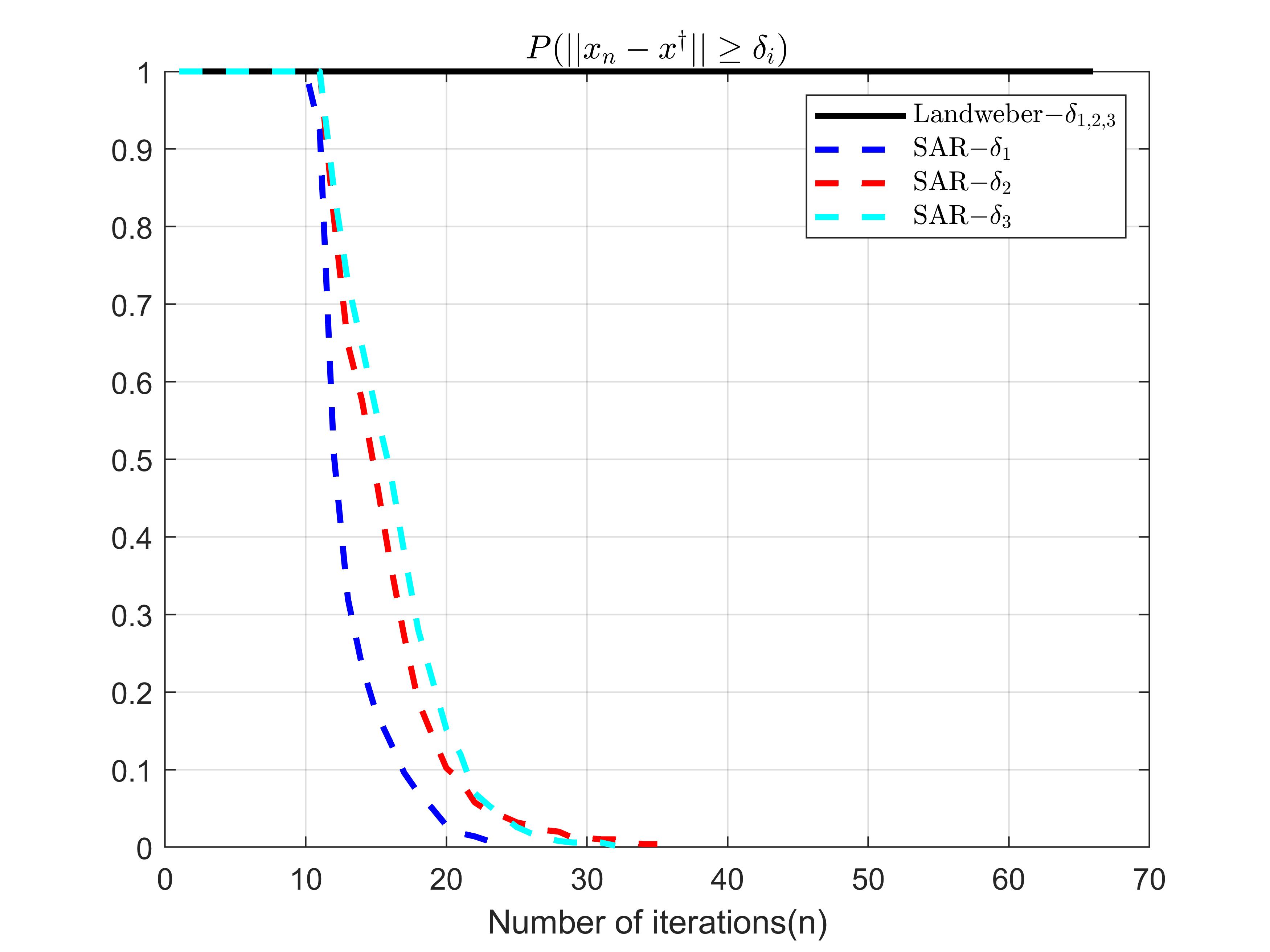} }
		\caption{Estimation of $P(\|x_n - x^\dag\|\geq \delta_i)$ using $N=500$ independent runs for 1D case with $\theta = 0.5 $ and different noise levels $\delta_1= 1 \%, \delta_2=3\%, \delta_3=5\%$. The numerical results are derived from 500 independent simulations, with the starting point uniformly sampled using the formula $0.01 + 4 \cdot \text{rand}$.}
		\label{addfig1}
	}
\end{figure}

\subsubsection{Advantage II: Uncertainty quantification of the solution}

Similar to \cite{ZhangChen23} for linear inverse problems, SAR \eqref{initial value problem} can also provide the uncertainty quantification of the quantity of interest, namely (i) it can quantify the uncertainty in error estimates for inverse problems, and (ii) it can be used to reveal and explicate the hidden information about real-world problems, usually obscured by the incomplete mathematical modeling and the ascendence of complex-structured noise. Due to limited space, we only present the first advantage of SAR since for clearly demonstrating the second advantage of SAR, some physical backgrounds must be described. In Figure \ref{1DCI75}, which show the 1D example, besides the deterministic approximate solution (i.e. the expectation of SAR), the $75\%$ and $99\%$ confidence intervals are provided with different noisy level, respectively. The figures show that with the corresponding high probability, the constructed confidence intervals contain the exact solution function. We mention that the theoretical explanations of the obtained results are the same as in \cite[Section 2]{ZhangChen23}, and we omit it. 


\begin{figure}[h]
	\centering
	{
		\subfigure[]{
			\includegraphics[scale=0.06]{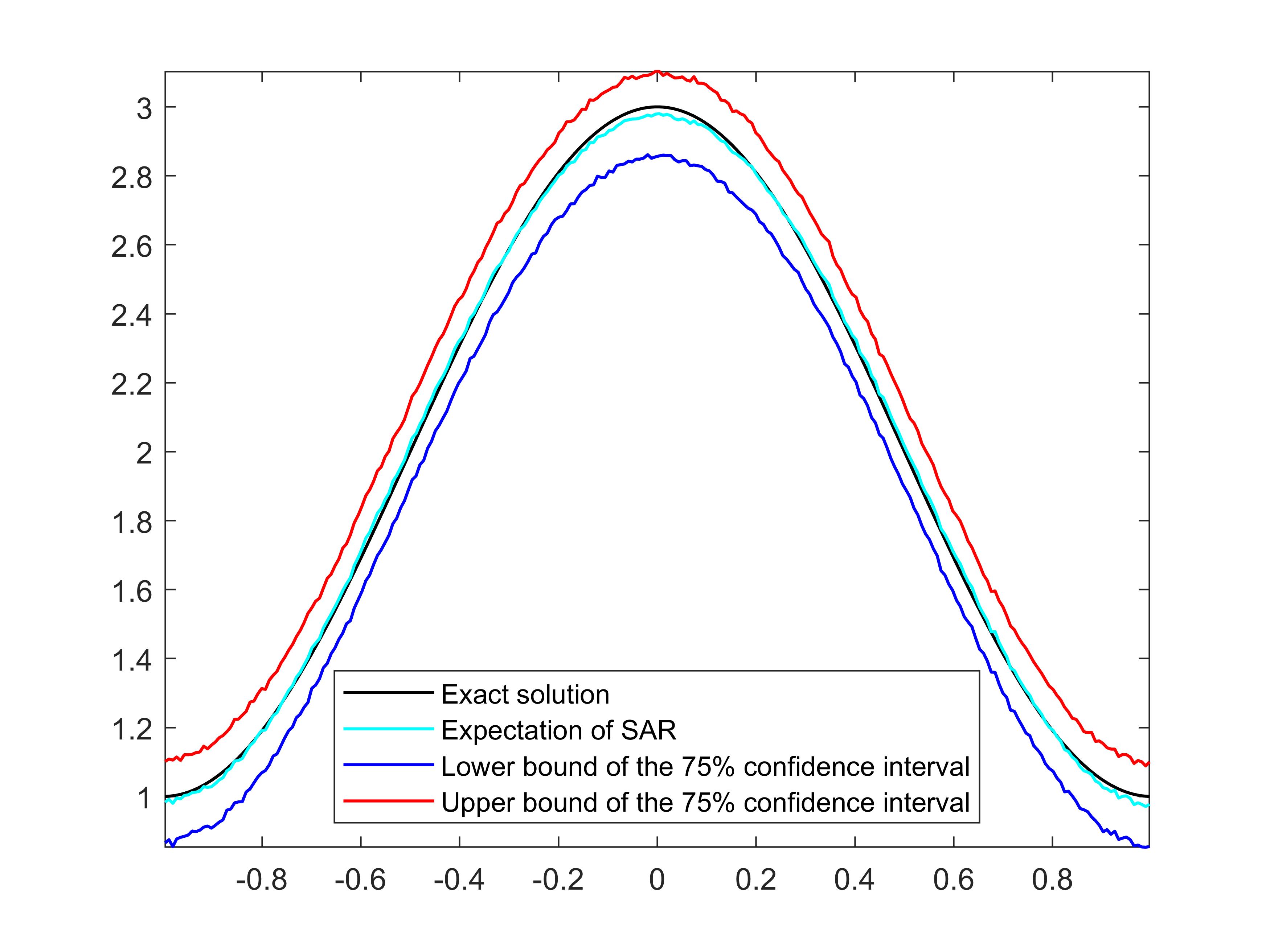} }
   	\subfigure[]{
			\includegraphics[scale=0.06]{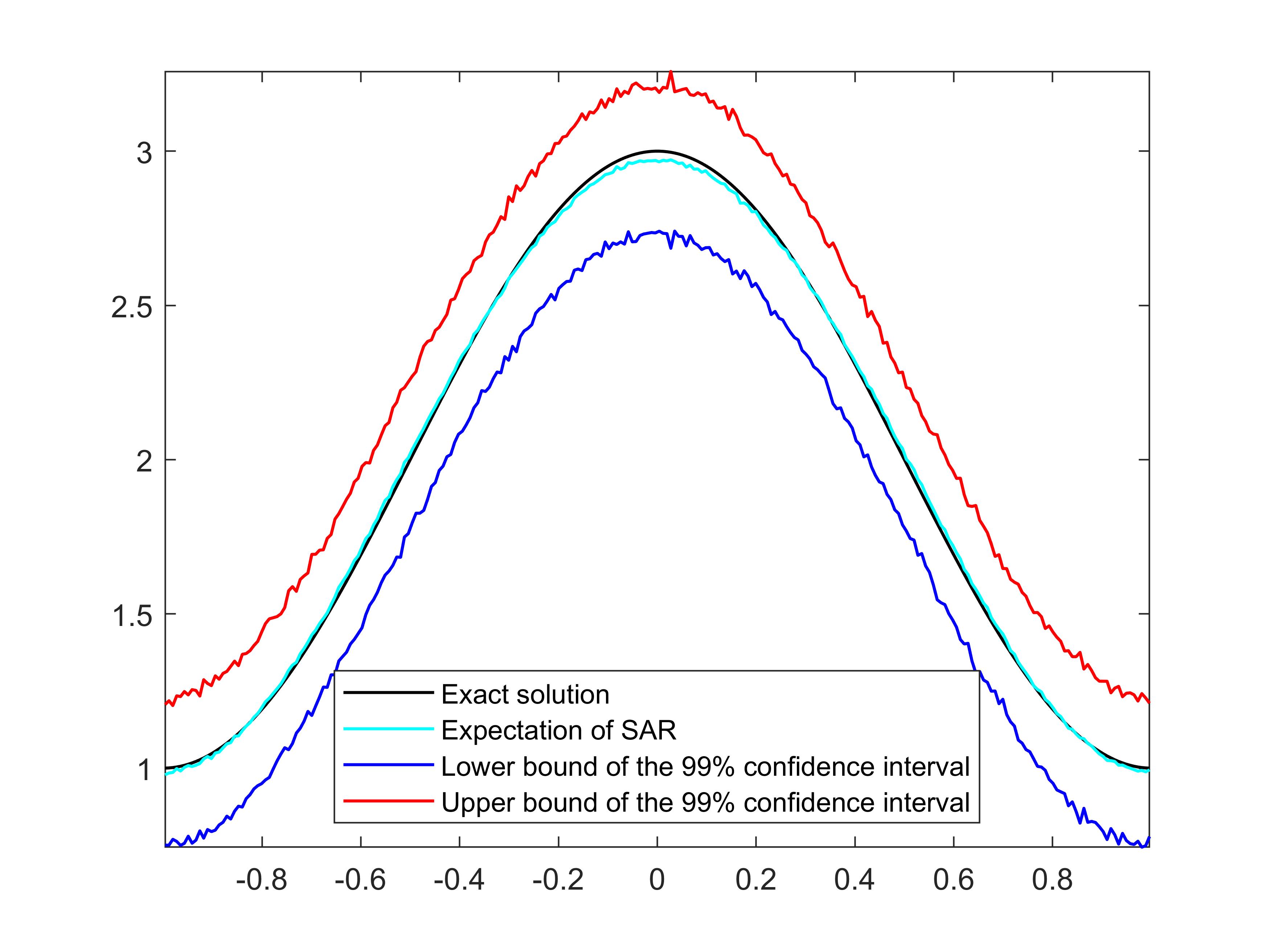} }
		\caption{(a) The expectation of SAR and the 75$\%$ confidence interval for problem \eqref{elliptic equation} with noise level $\delta = 1\%$; (b) The expectation of SAR and the 99$\%$ confidence interval for problem \eqref{elliptic equation} with noise level $\delta = 2\%$. Other parameters in both (a) and (b): sample size N = 800, $\tau$ = 1.1, $\theta=0.1$.}
		\label{1DCI75}
	}
\end{figure}

To further illustrate the impact on the confidence intervals of approximate solutions with respect to different sample sizes and levels of randomization, we present the corresponding expected reconstructions and the 90\% and 85\% confidence intervals for the 2D case in Figures \ref{2Dd1CIpathCI} and \ref{2Dd1CIpaththeta}, respectively. We also note that, in the numerical application, a preferable value for the confidence interval would be determined by both the noise level $\delta$ and the level of randomization $\theta$.
This is, effectively, also explained visually in Figure \ref{2D-CIs}, in which different confidence intervals for the 2D case are presented.


\begin{figure}[h]
	\centering
	{
		\subfigure[N=20]{
			\includegraphics[scale = 0.053]{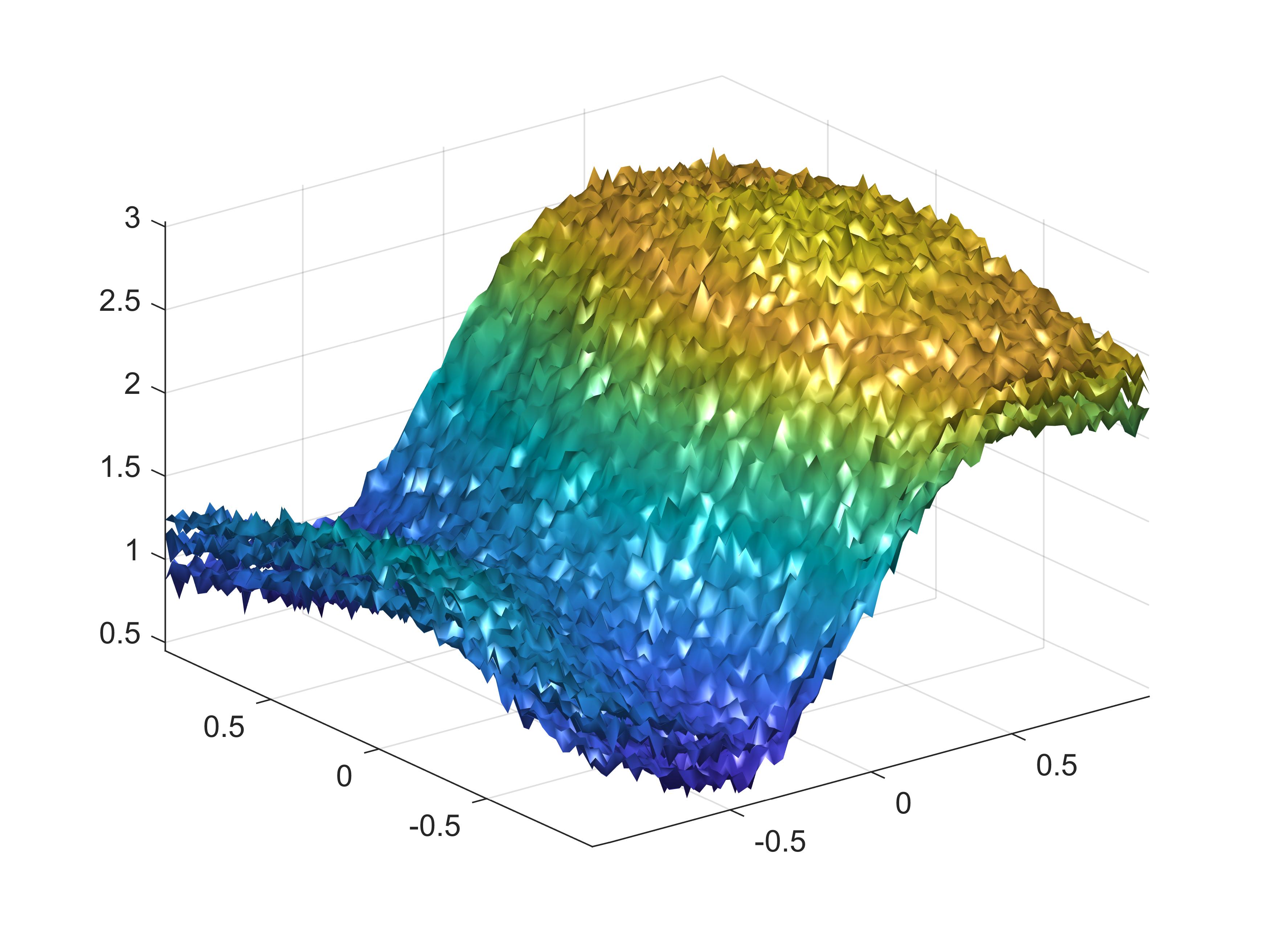} }
		\subfigure[N=50]{
			\includegraphics[scale = 0.053]{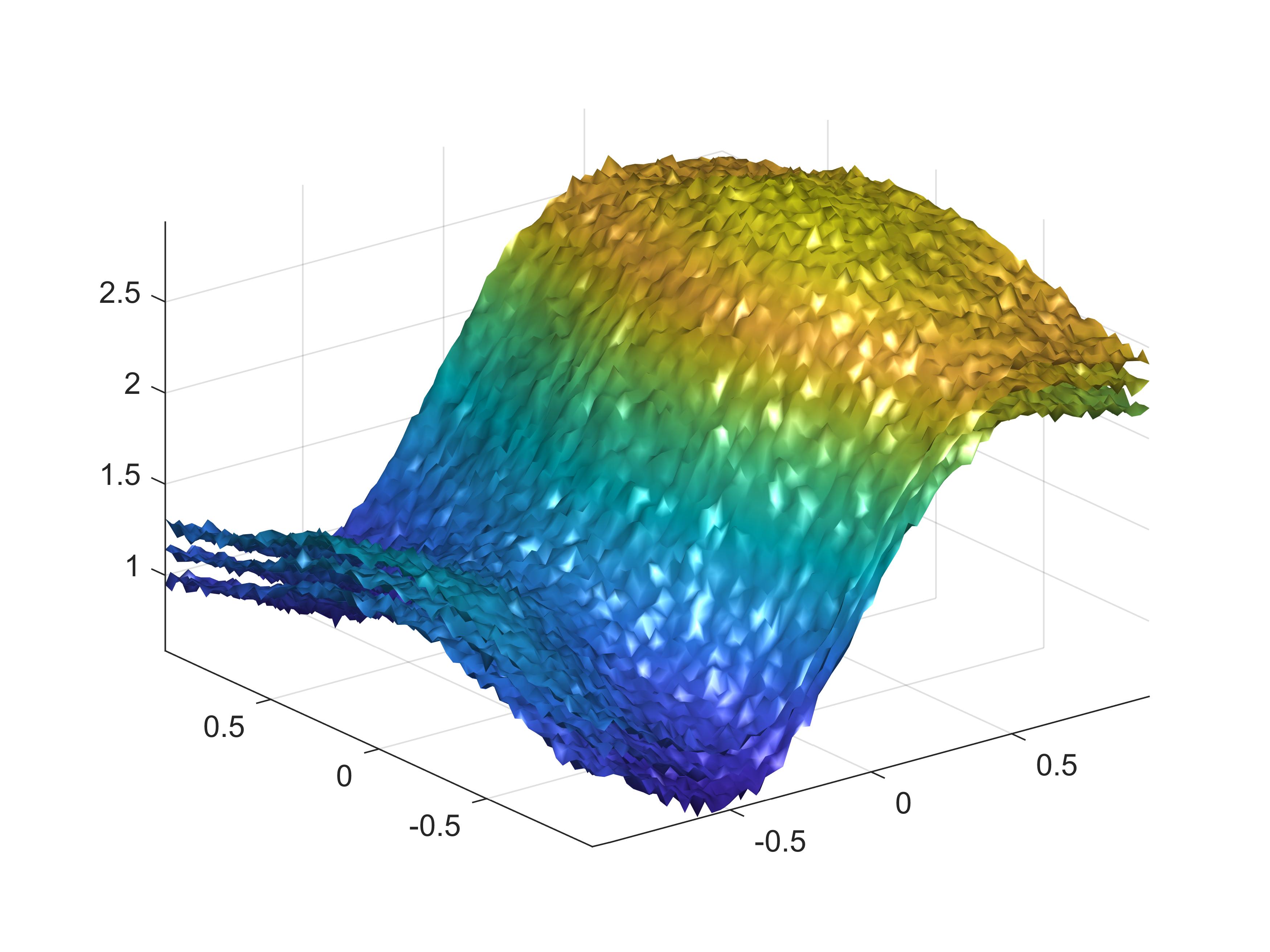} }
		\subfigure[N=100]{
			\includegraphics[scale = 0.053]{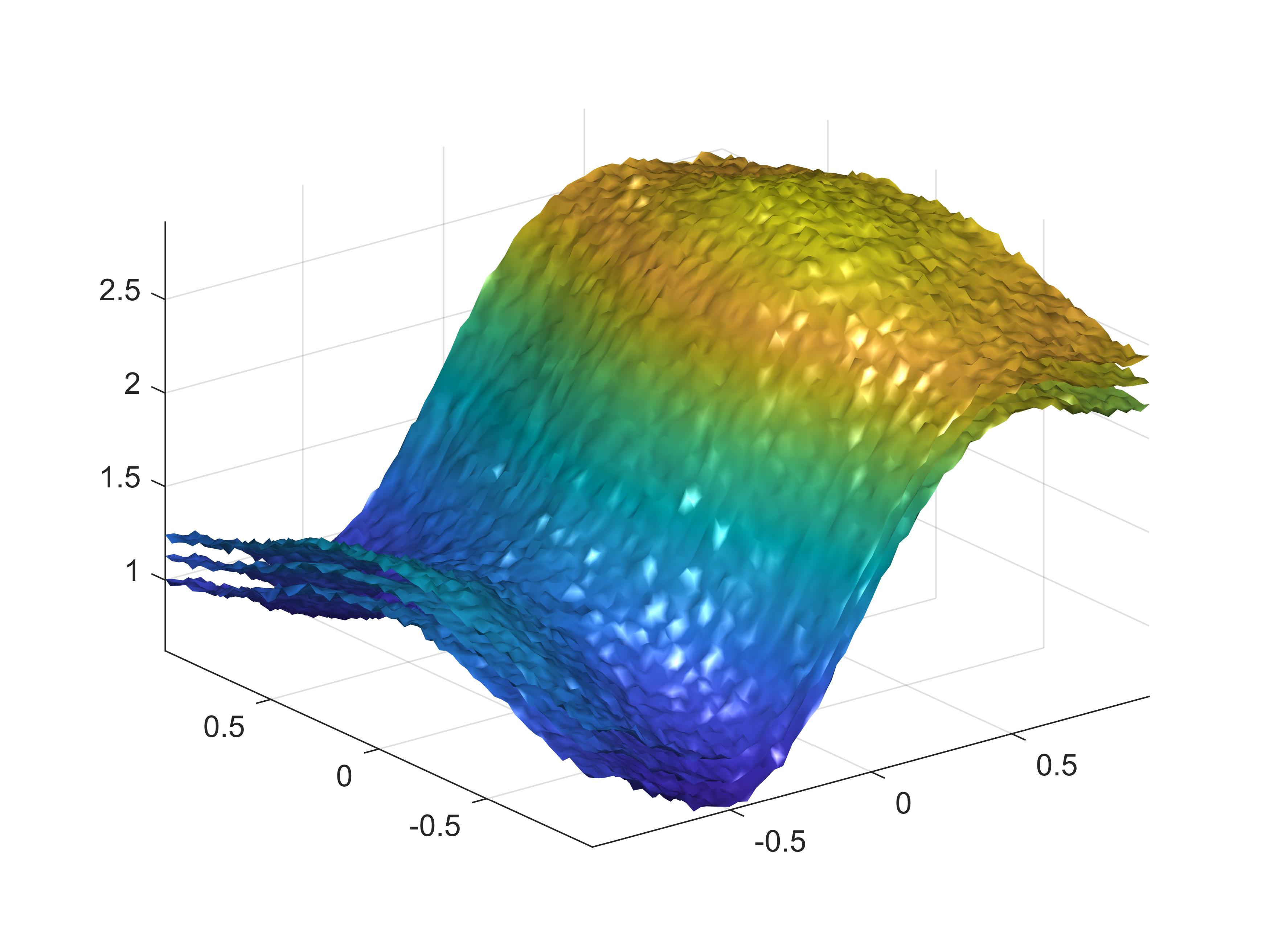} }
		\subfigure[N=400]{
			\includegraphics[scale = 0.053]{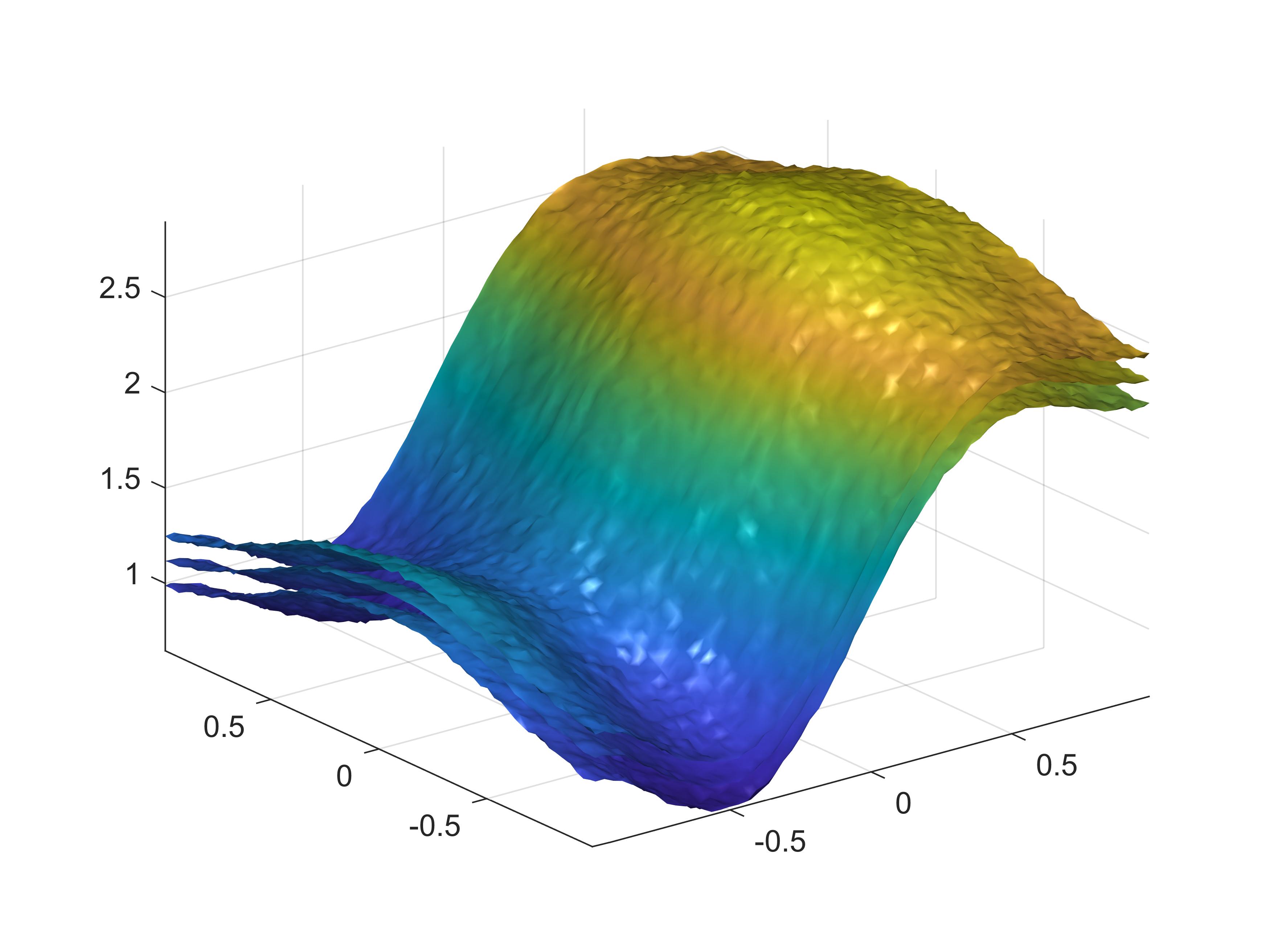} }
		\caption{Reconstructions for the 2D case, with $90\%$ confidence interval and $\delta = 1\%$, $\theta=0.02$.} 
		\label{2Dd1CIpathCI}
	}
\end{figure}

\begin{figure}[H]
	\centering
	{
		\subfigure[$\theta=0.01$]{
			\includegraphics[scale = 0.053]{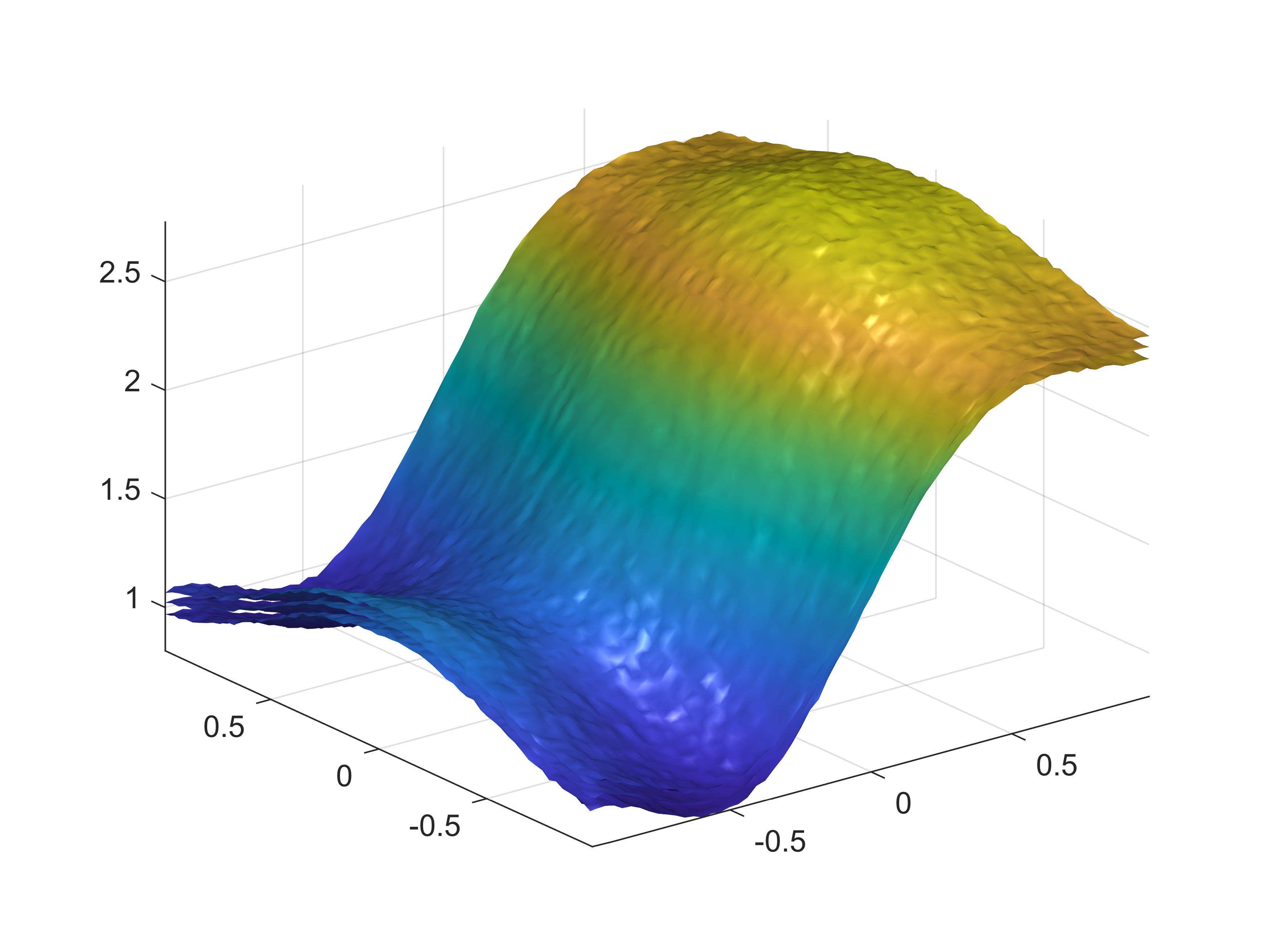} }
		\subfigure[$\theta=0.02$]{
			\includegraphics[scale = 0.053]{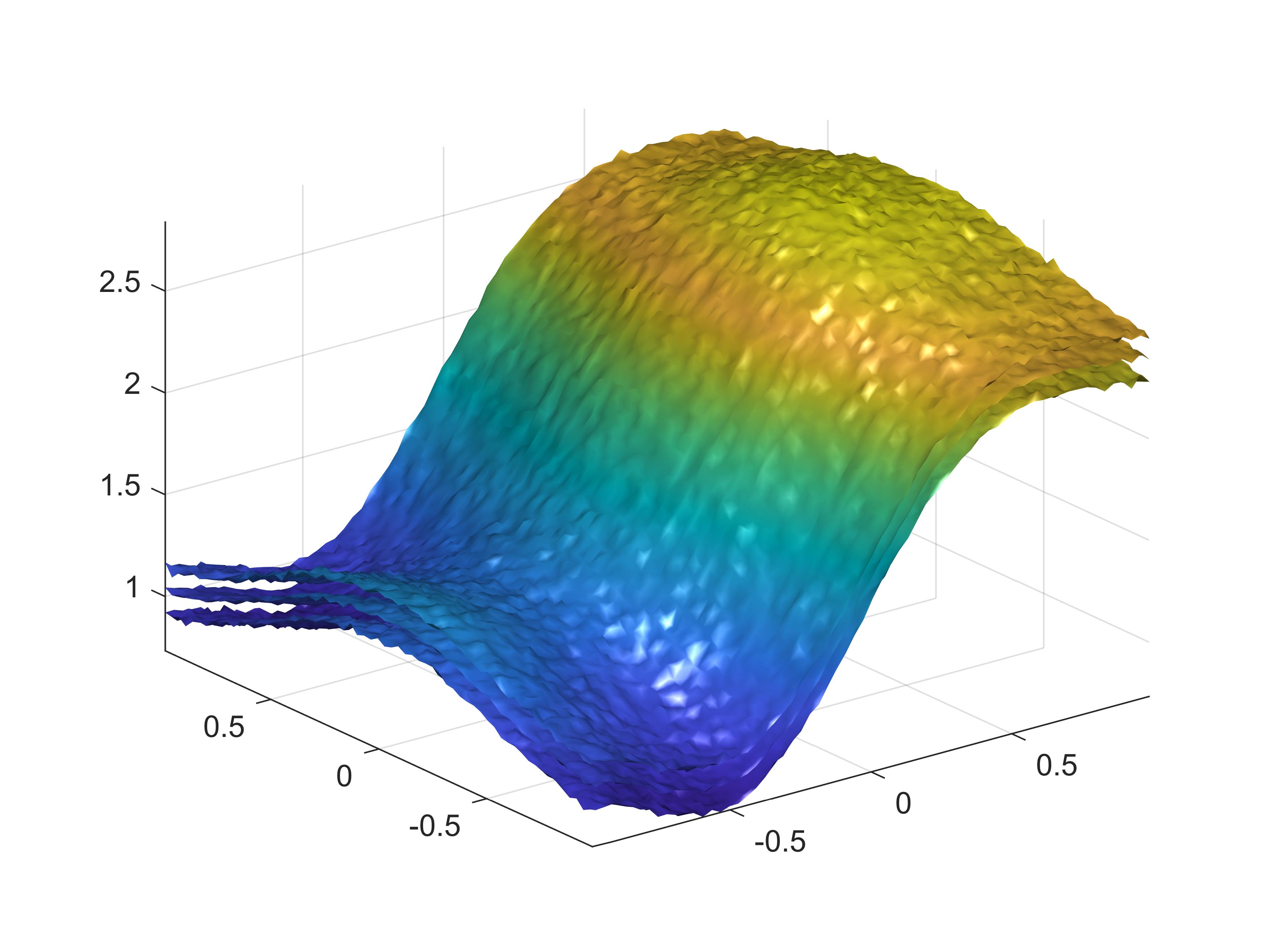} }
		\subfigure[$\theta=0.03$]{
			\includegraphics[scale = 0.053]{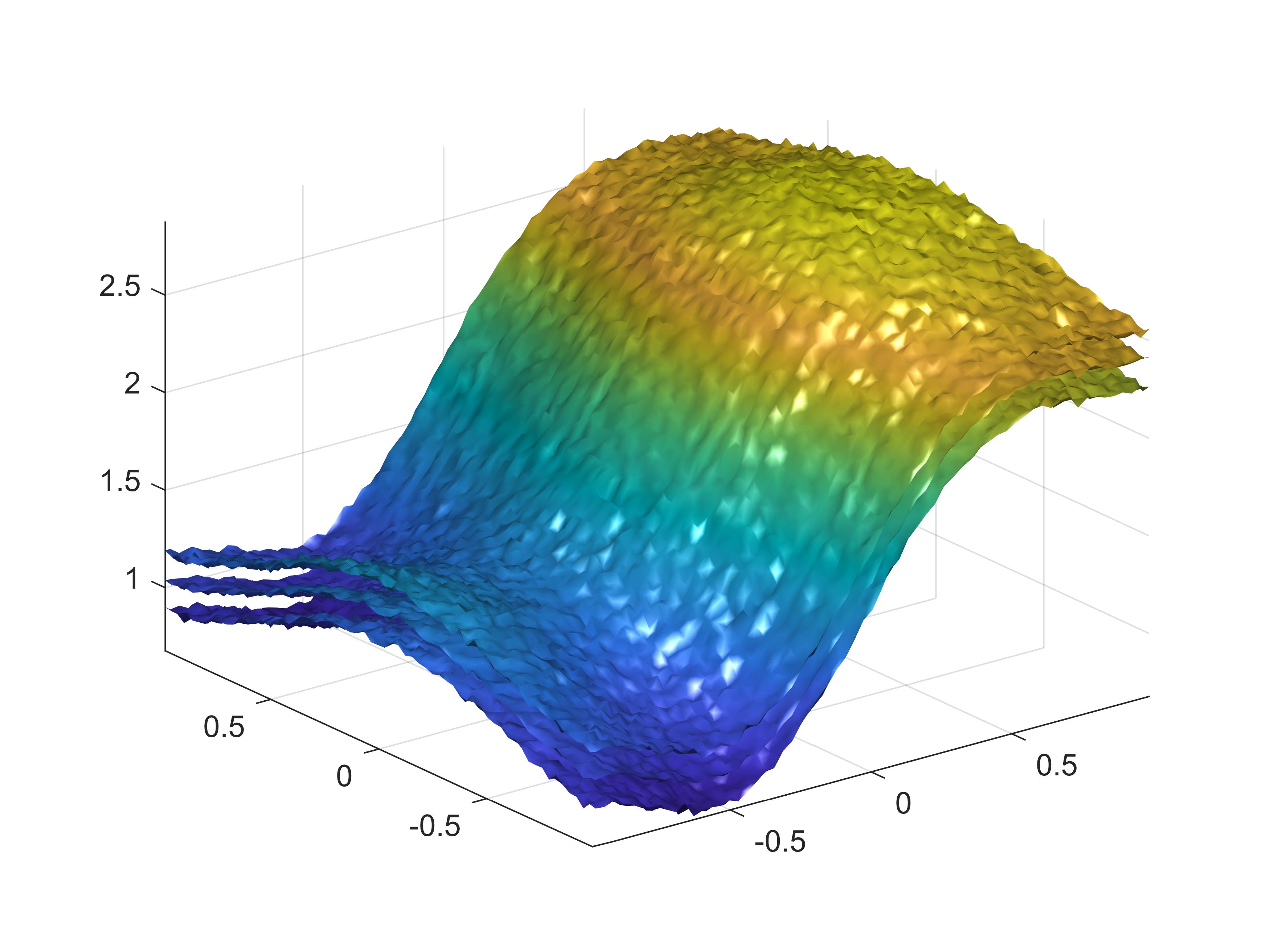} }
		\subfigure[$\theta=0.05$]{
			\includegraphics[scale = 0.053]{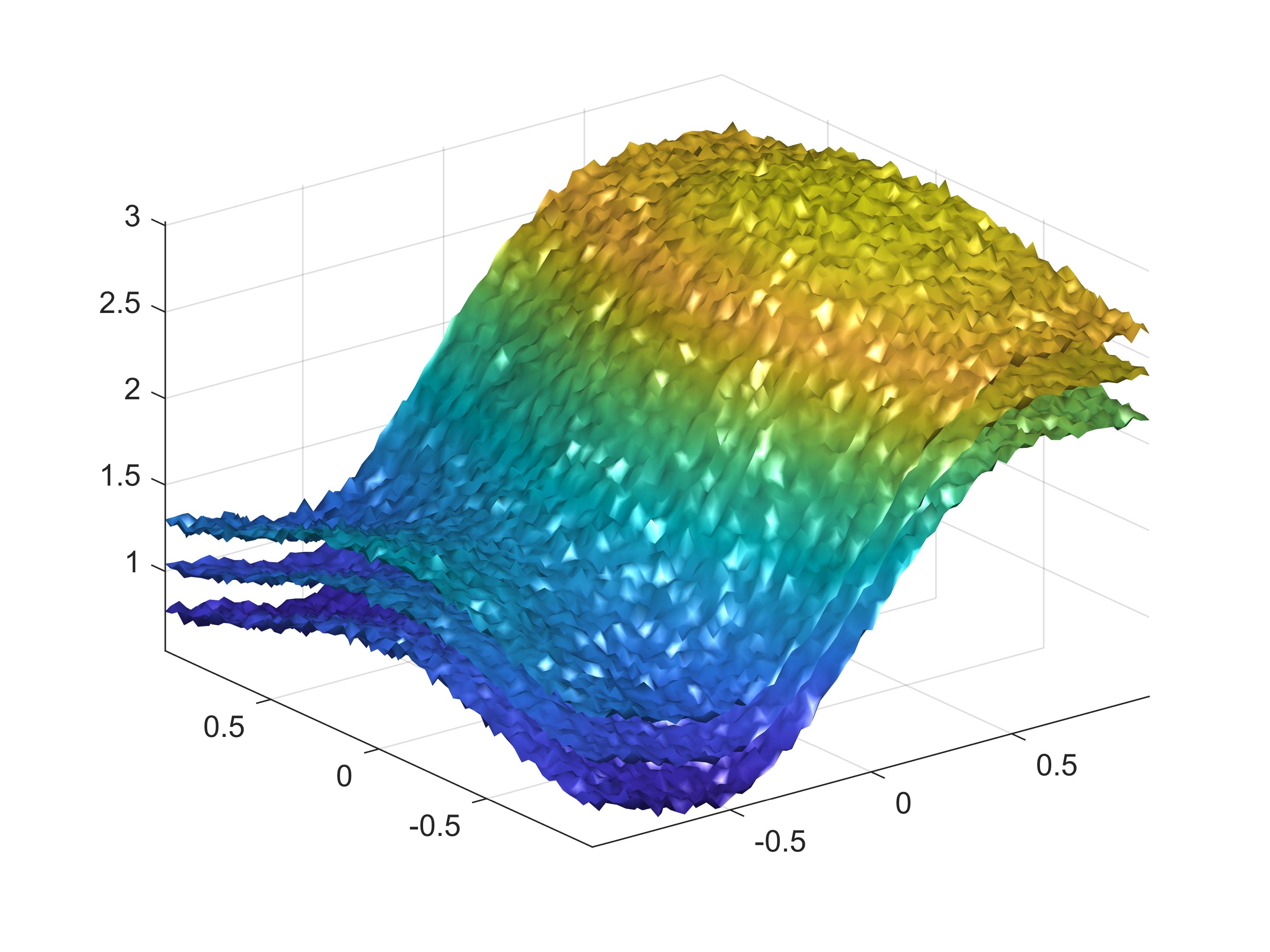} }
		\caption{Reconstructions for the 2D case, with $85\%$ confidence interval and $\delta = 1\%$, $N=200$.}
		\label{2Dd1CIpaththeta} 
	}
\end{figure}

\begin{figure}[h]
	\centering
	{
		\subfigure[$60\%$ confidence interval]{
			\includegraphics[scale = 0.053]{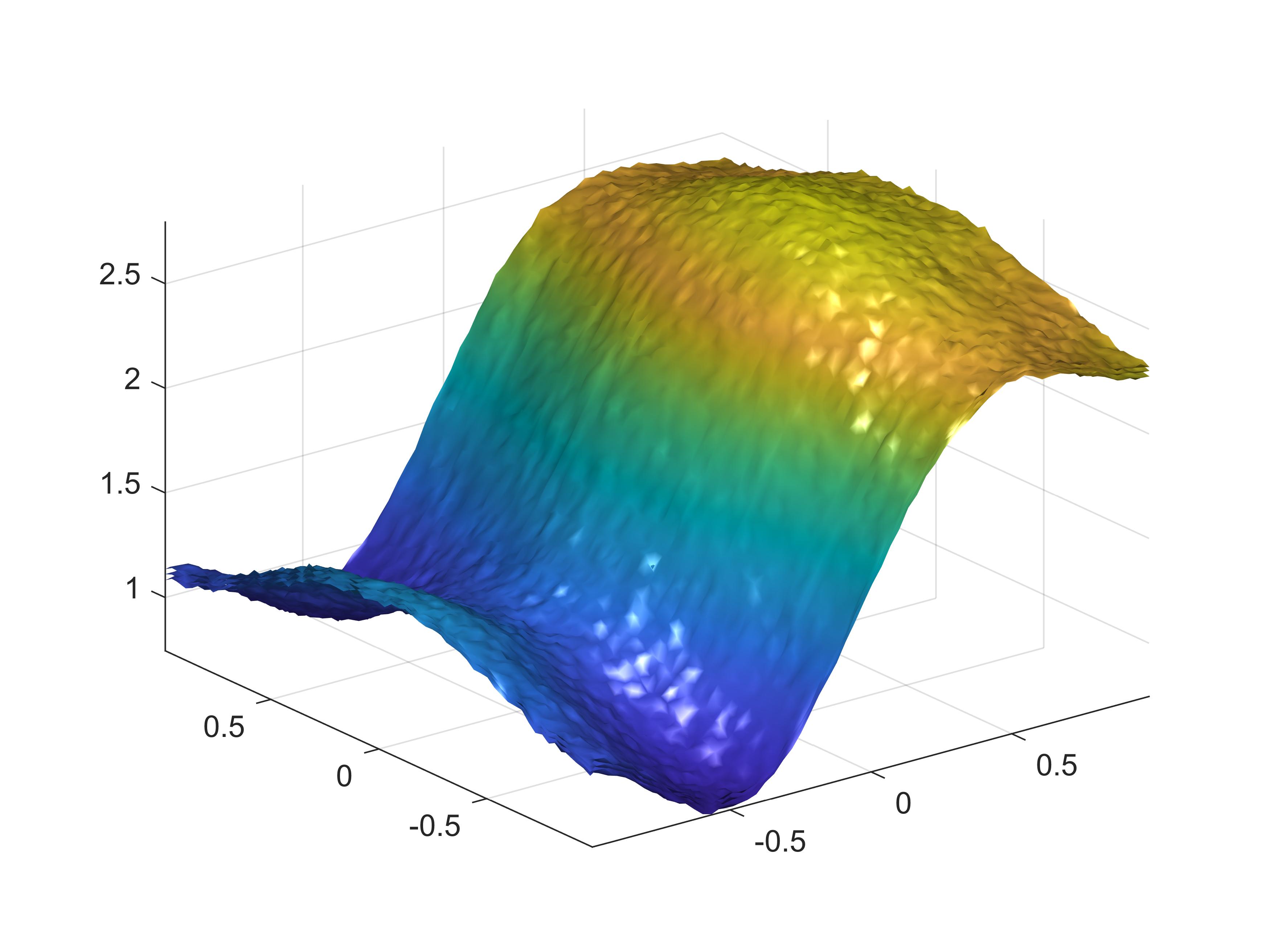} }
		\subfigure[$70\%$ confidence interval]{
			\includegraphics[scale = 0.053]{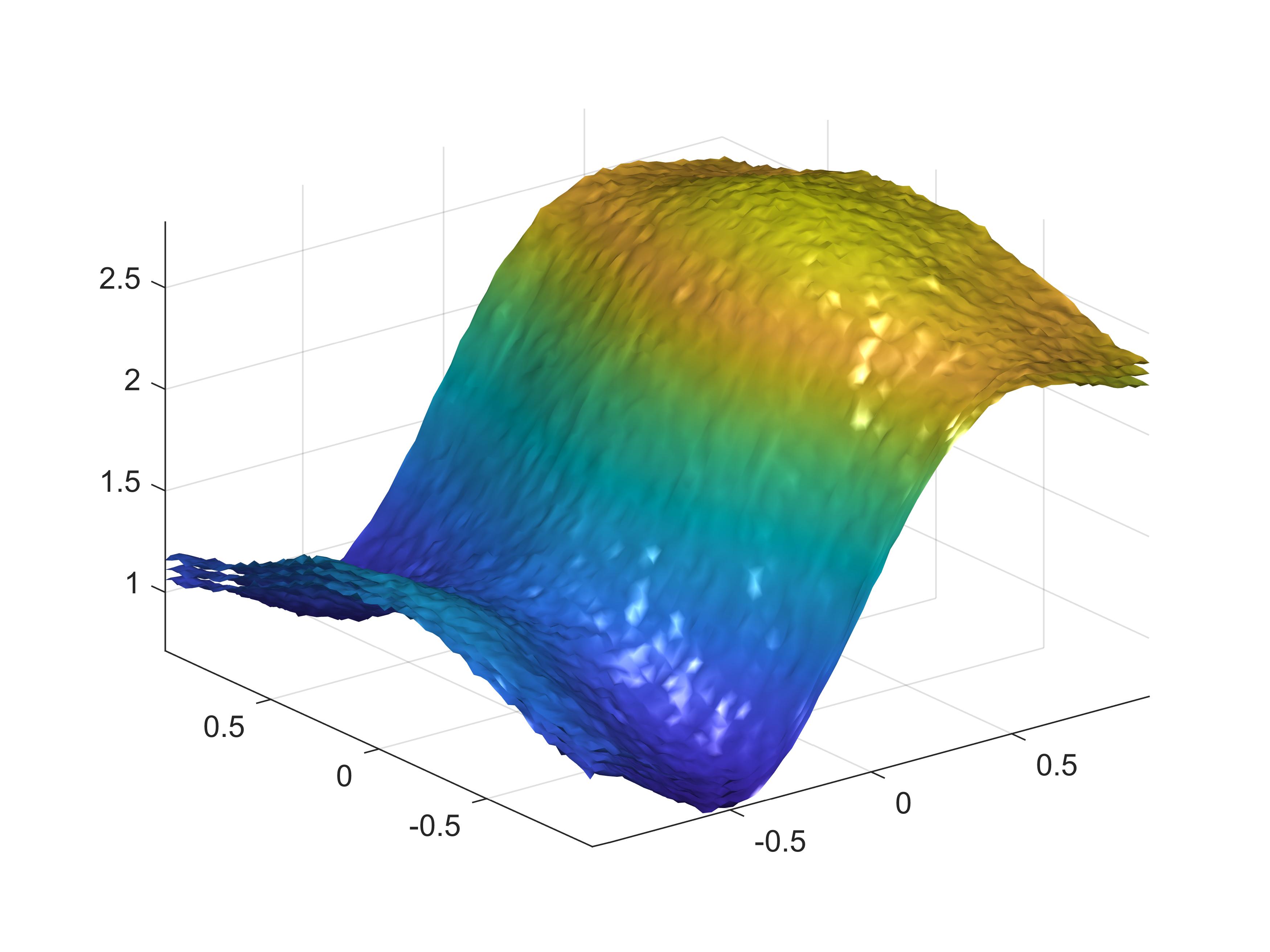} }
		\subfigure[$80\%$ confidence interval]{
			\includegraphics[scale = 0.053]{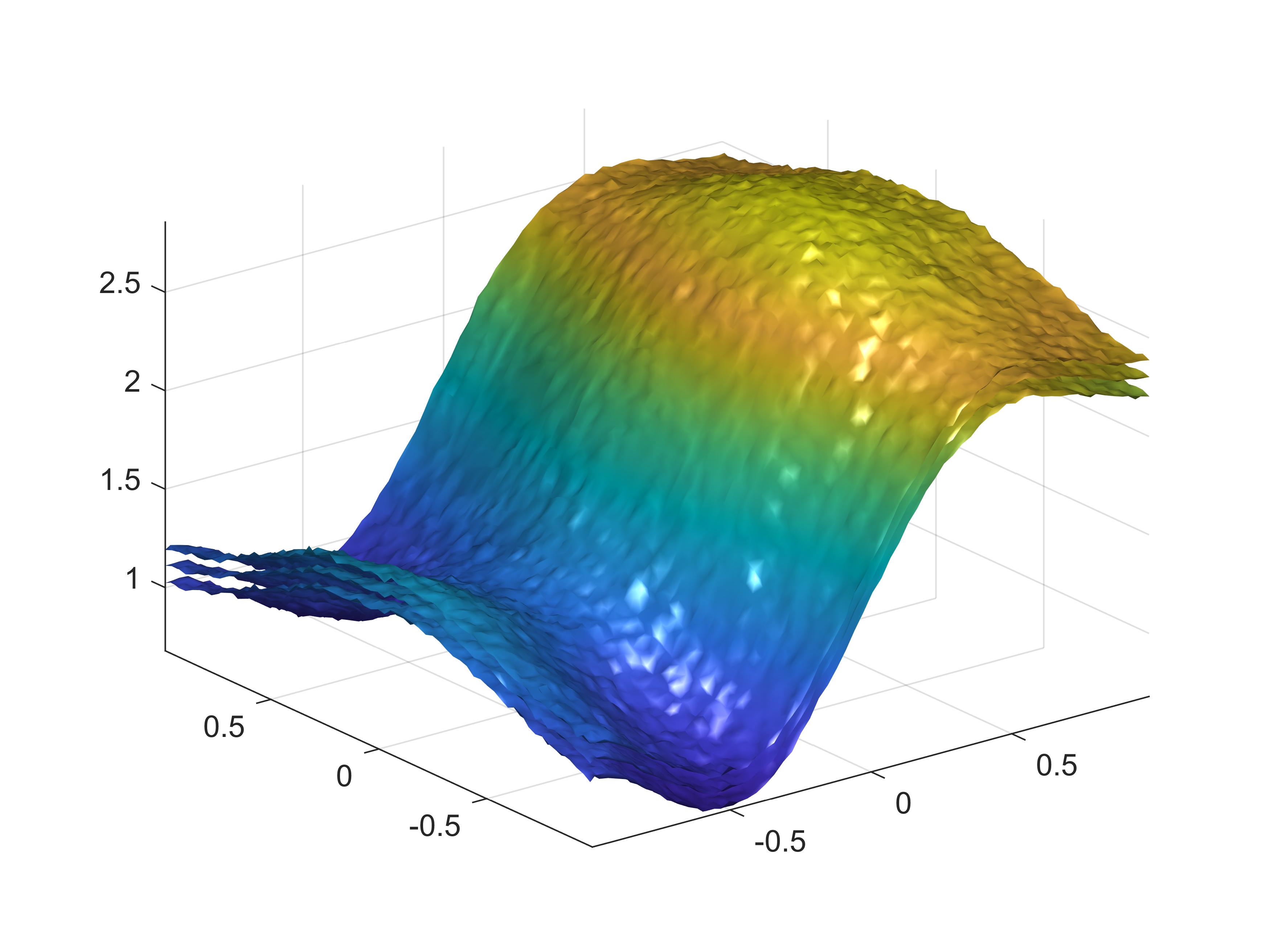} }
		\subfigure[$90\%$ confidence interval]{
			\includegraphics[scale = 0.053]{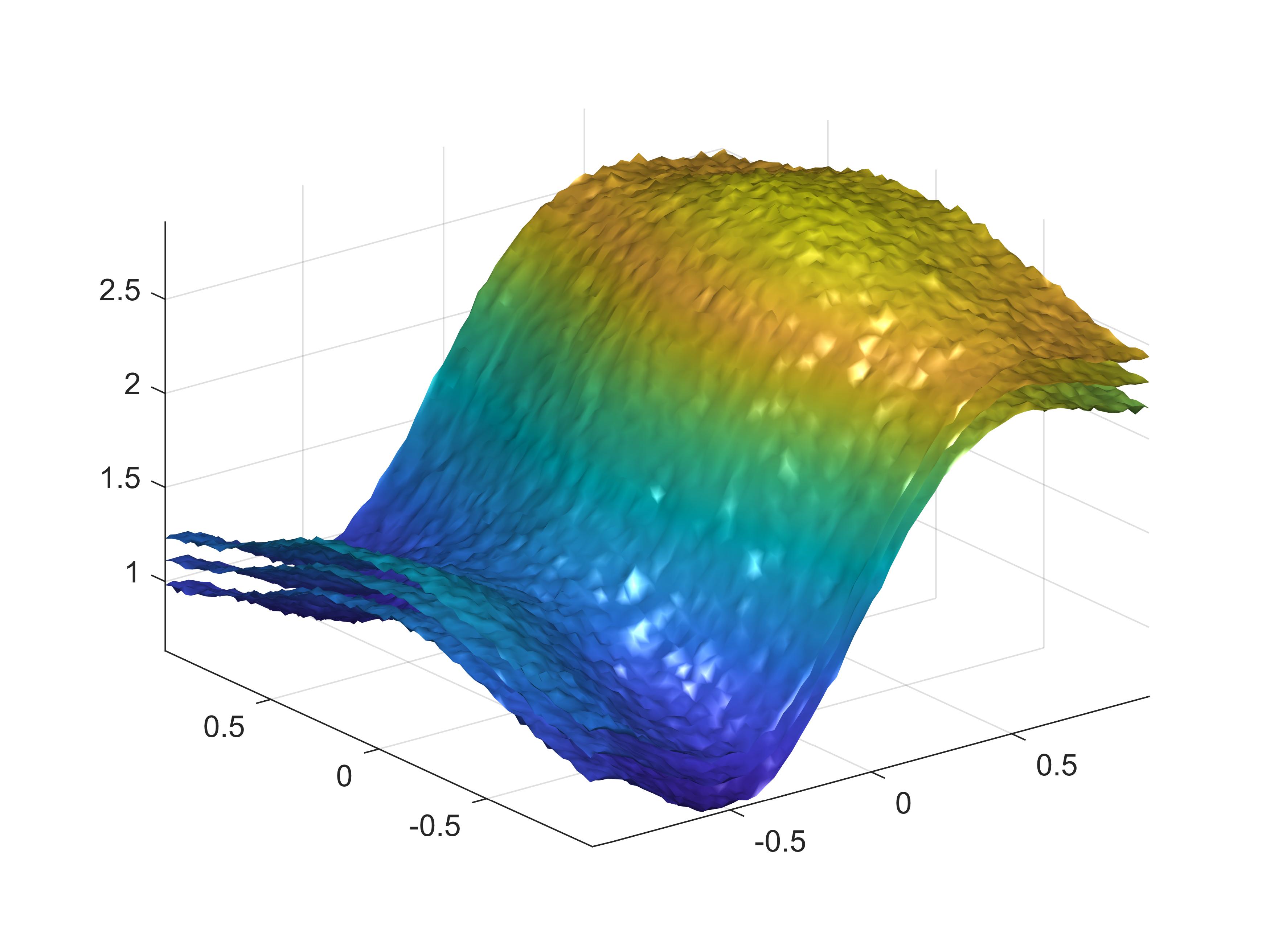} }
		\caption{Reconstructions for the 2D case, with $\delta = 1\%$, $N=200$, and $\theta=0.02$.} 
		\label{2D-CIs}
	}
\end{figure}

\subsection{Model problem 2:  the autoconvolution equation}
\label{Ex2}
In our second group of simulations, we demonstrate a new capability of SAR: the identification of multi-solutions by clustering samples of obtained approximate solutions. As a model problem, we consider the following 1D auto-convolution equation, which has many applications in spectroscopy (e.g. the modern methods of ultrashort laser pulse characterization) \cite{Baumeister1991, GerthHofmann2014}, the structure of solid surfaces \cite{Dai2013} and the nano-structures \cite{Fukuda2010}, etc.   
\begin{equation}
\label{AutoConv}
\int_{0}^{t} x(t-s) x(s) \,ds = y(t). 
\end{equation}

Before starting the numerical investigation, let us recall the weak uniqueness result for the non-linear integral equation (\ref{AutoConv}). 

\begin{prop}
	\label{ProUnique}
	(\cite{GerthHofmann2014}) For $y\in L^2[0,2]$, the integral equation has at most two different solutions $x_1(t),x_2(t) \in L^1[0,1]$. Moreover, it holds that  $x_1(t)+x_2(t)=0, t \in [0,1]$.
\end{prop}

We remark that to the best of our knowledge, the existence results for nonlinear integral equation (\ref{AutoConv}) are quite limited. Here, we omit the existence issue as for the model problem, the solution to (\ref{AutoConv}) is pre-defined, namely
\begin{equation}
\label{AutoConv-x}
x(t) = 30 t (1-t)^3+\cos (2\pi t). 
\end{equation}

The numerical experiments yielded the following results: our approach \eqref{EM} generated 4,983 satisfactory approximate solutions. It is important to note that we did not select all $N=5000$ obtained solutions, as we imposed a maximum number of iterations in our algorithm, and not all approximate source functions met the required accuracy in terms of residual error. We then classified these 4,983 candidates into two groups, $\mathfrak{C}_1$ and $\mathfrak{C}_2$, using classical $k$-means clustering. For this classification, we normalized the source functions and employed the following modified Kullback-Leibler divergence
\begin{equation*}
D^\varepsilon_{KL} (\mathbf{u} \| \mathbf{v}) := \sum^n_{i=1} |[\mathbf{u}]_i| \ln \left(  \frac{|[\mathbf{u}]_i|+\varepsilon}{|[\mathbf{v}]_i|+\varepsilon} \right) 
\end{equation*} 
for two vectors $\mathbf{u}, \mathbf{v}\in \mathbb{R}^n$. Here, $[\mathbf{u}]_i$ denotes the i-th component of vector $\mathbf{u}$.

\begin{figure}[b]
	\centering
	{
		\subfigure[]{
			\includegraphics[scale=0.06]{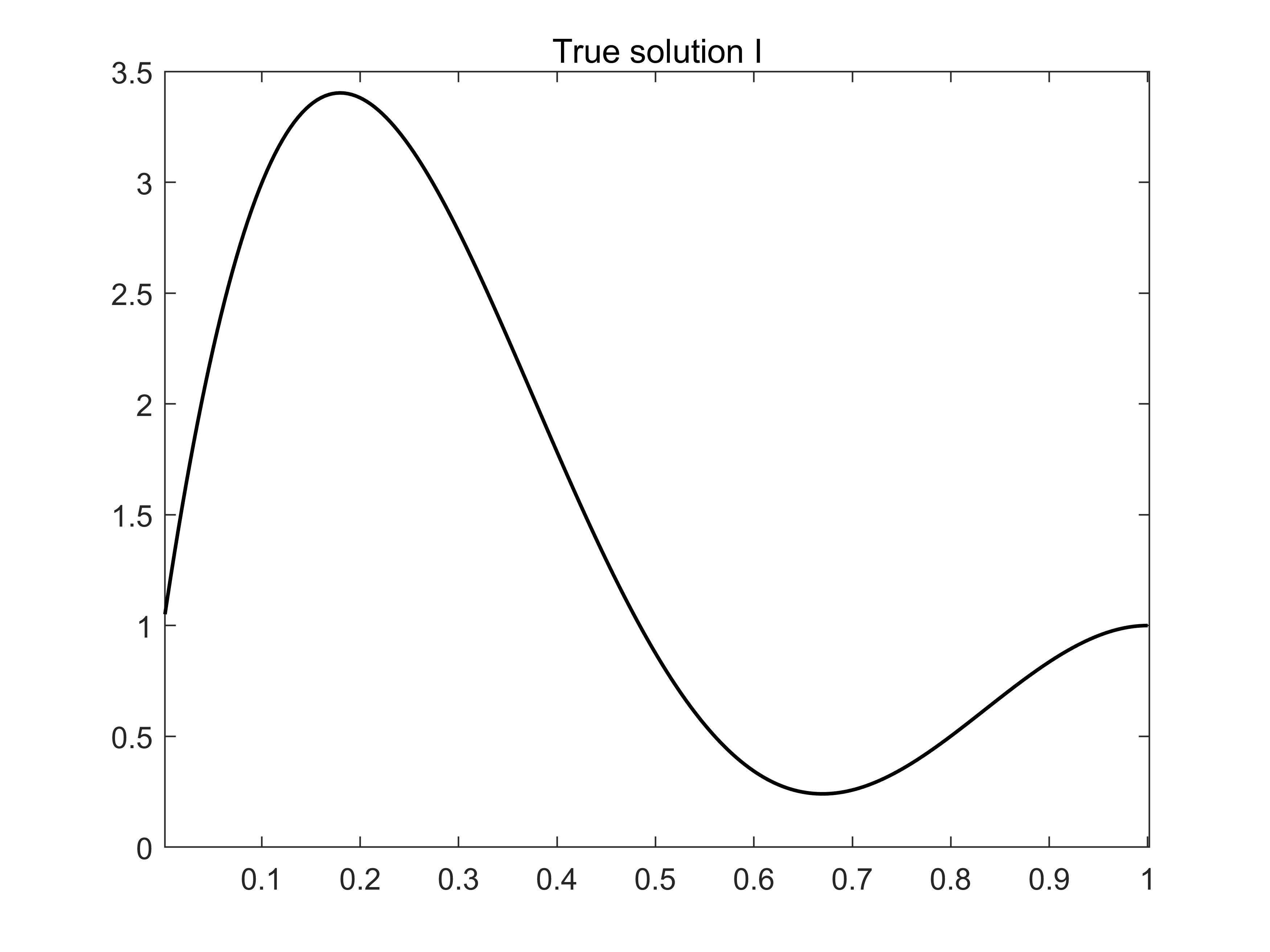} }
		\subfigure[]{
			\includegraphics[scale=0.06]{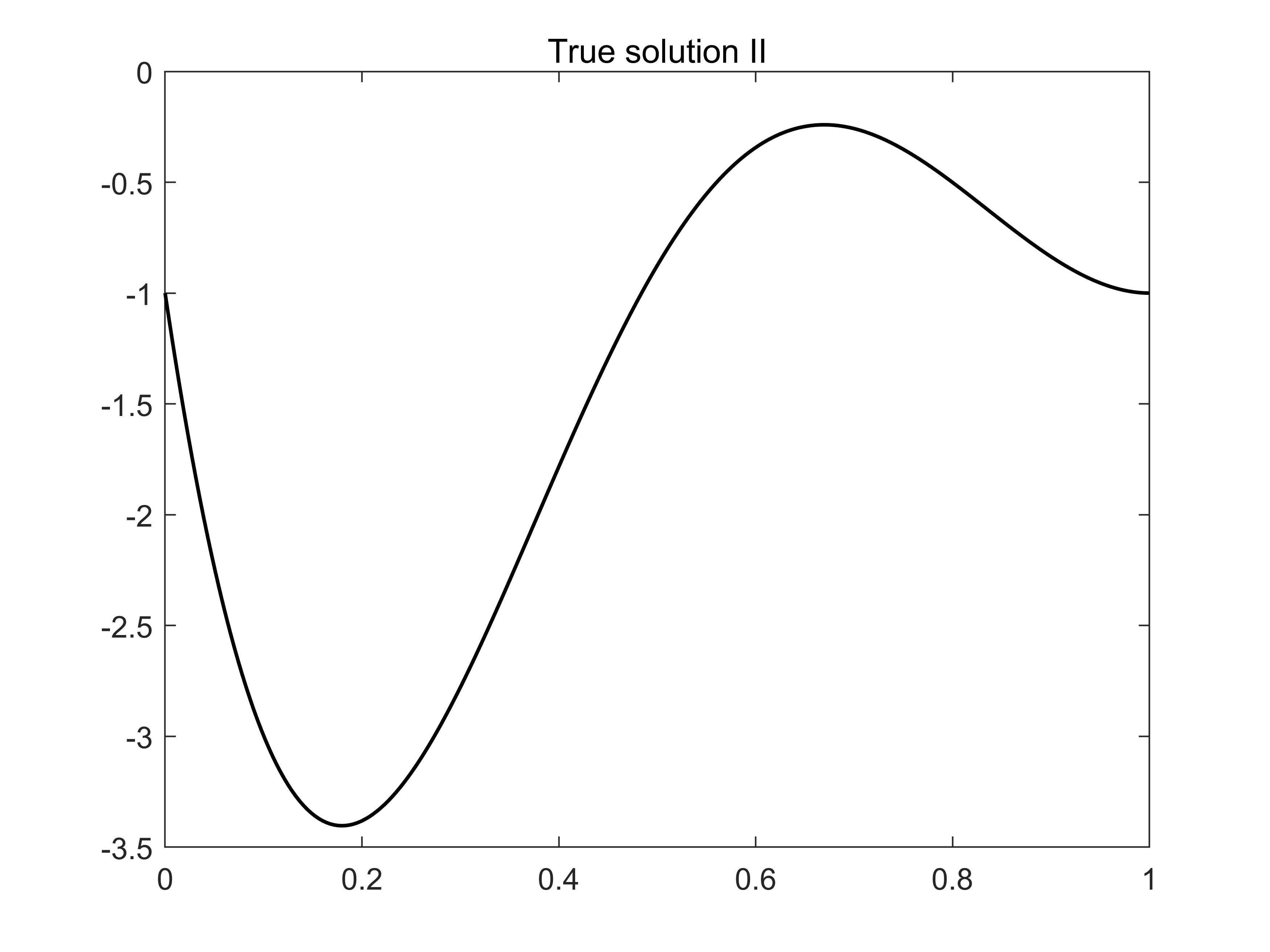} }
   		\subfigure[]{\label{class1}
			\includegraphics[scale=0.06]{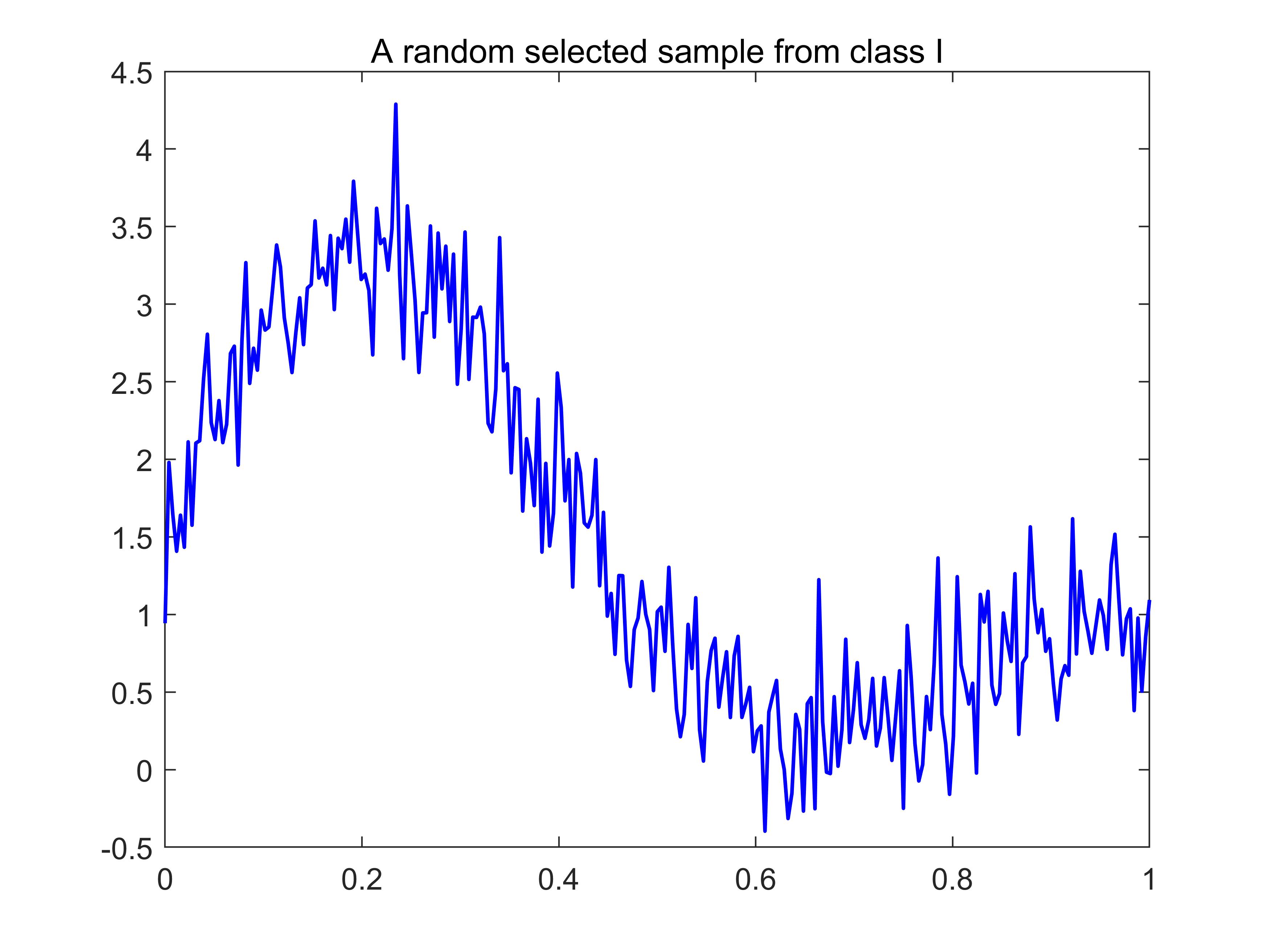} }
		\subfigure[]{\label{class2}
			\includegraphics[scale=0.06]{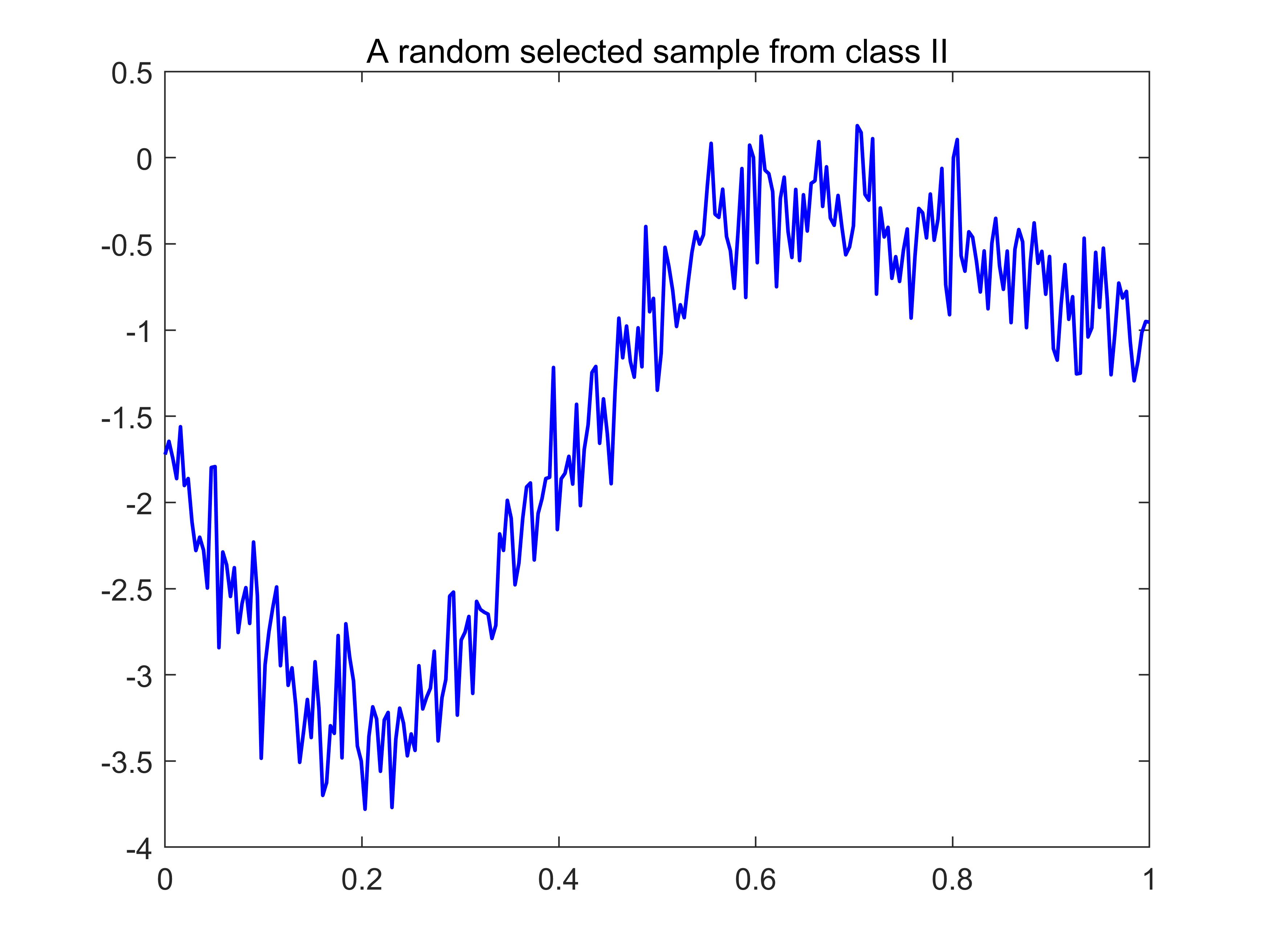} }
   		\subfigure[]{\label{estimate1}
			\includegraphics[scale=0.06]{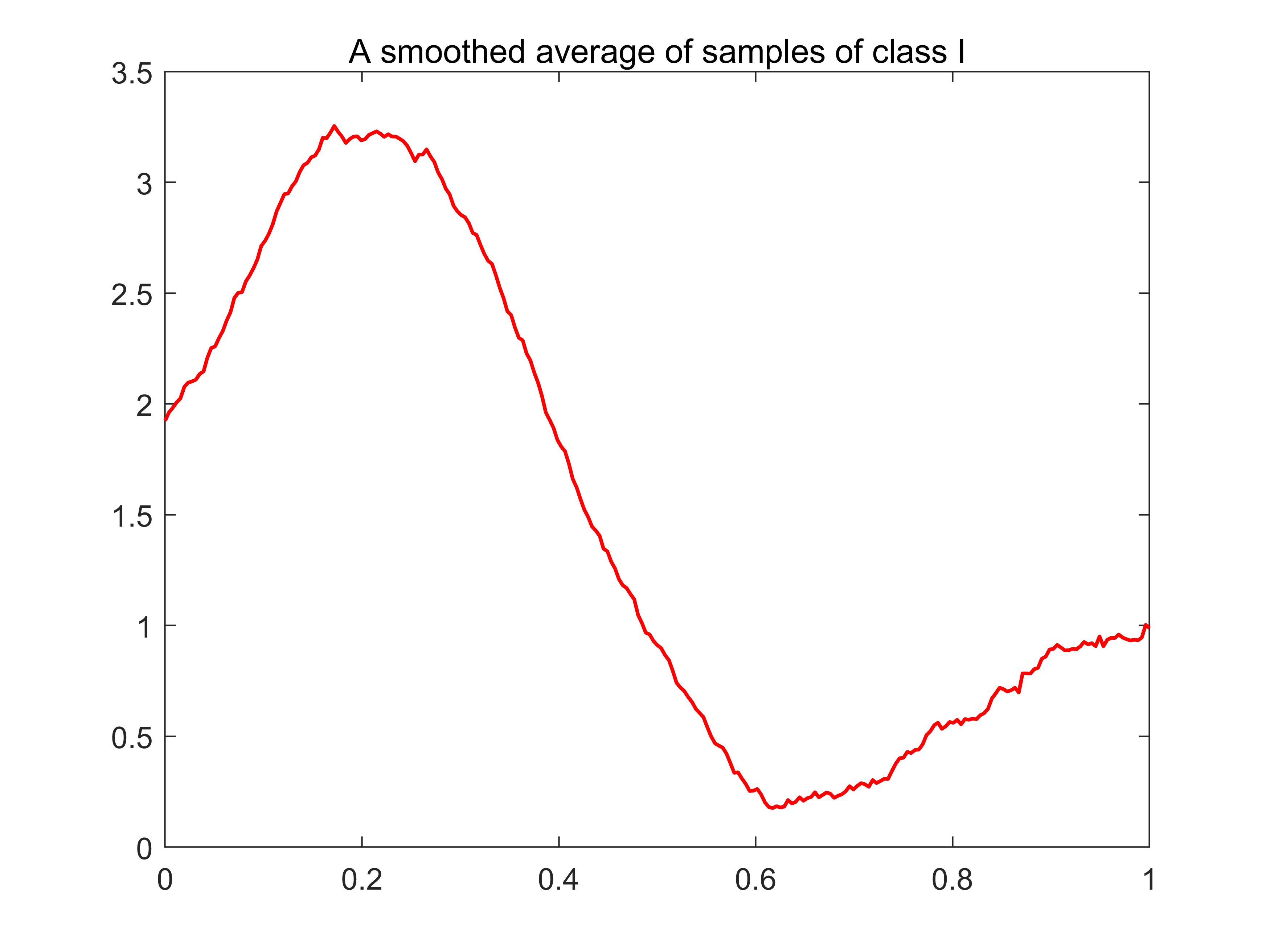} }
		\subfigure[]{\label{estimate2}
			\includegraphics[scale=0.06]{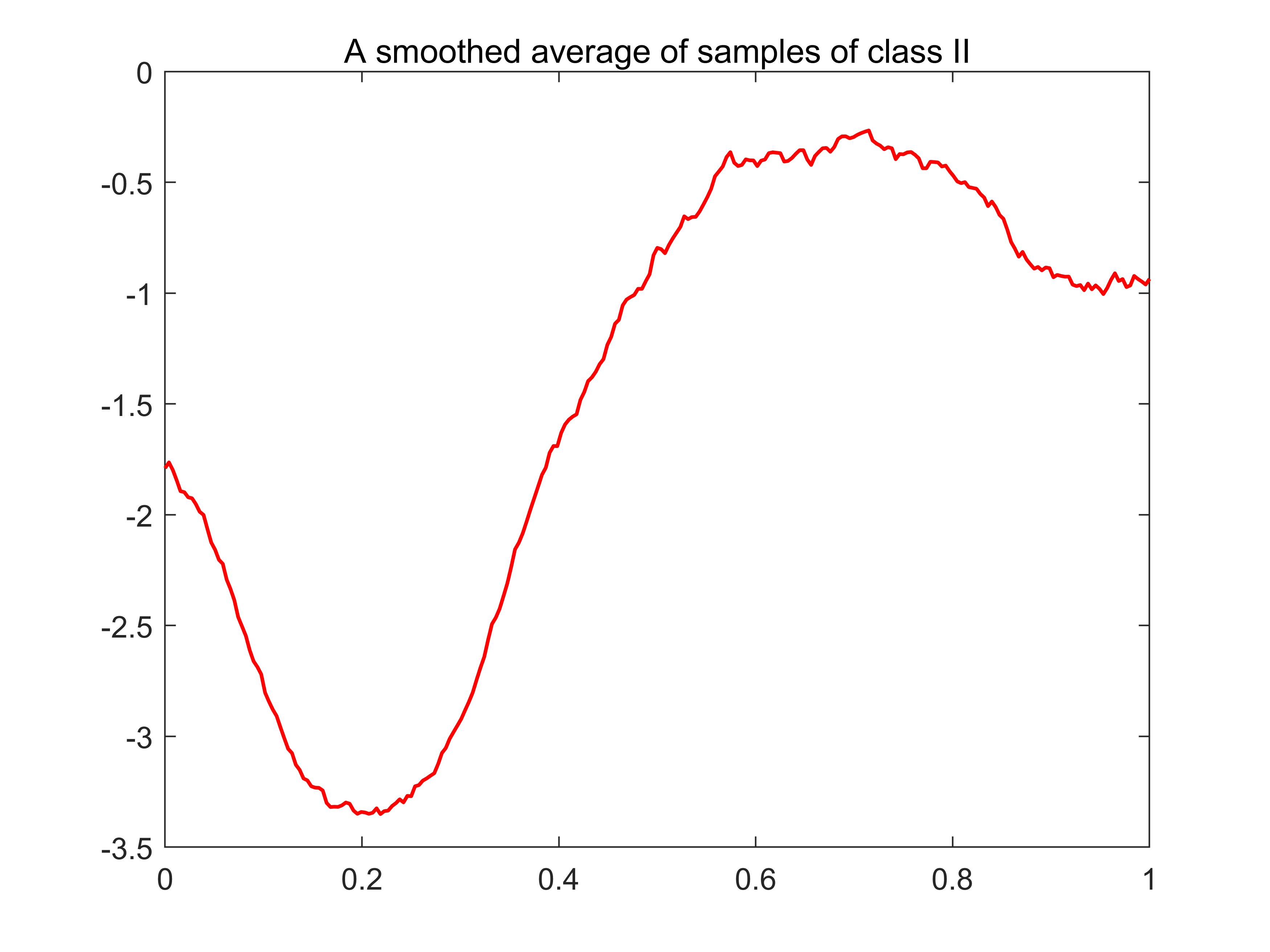} }
		\caption{Results for classification of samples by $k$-mean with with Wasserstein distance. $\delta=5\%$, $\theta=0.5$, $\tau=1.1$, $N=5000$.} \label{k-means-clustering}
	}
\end{figure}

We then construct two smoothed representatives that are closest (with respect to $D^\varepsilon_{KL}$) to the centers of the two classes as representatives ($\mathbf{x}_1$ and $\mathbf{x}_2$) of our two types of source functions, i.e. 
\begin{equation}
\label{representatives}
\mathbf{x}_j := \arg\min_{\mathbf{x}_j} \sum_{\mathbf{x}_p\in \mathfrak{C}_j} D^\varepsilon_{KL} (\mathbf{x}_j \| \mathbf{x}_p) + \alpha (\|\mathbf{x}_j\|^2 + \|D \mathbf{x}_j\|^2), \quad j=1,2, ~ \varepsilon=0.001,~  \alpha=0.01, 
\end{equation} 
where $\|D \mathbf{x}_j\|^2:= \sum^{n-1}_{i=1} ([\mathbf{x}_j]_{i+1}-[\mathbf{x}_j]_{i})^2$ for $\mathbf{x}_j\in \mathbb{R}^n$. These are plotted on Figures \ref{estimate1} (for $\mathbf{x}_1$) and \ref{estimate2} (for $\mathbf{x}_2$). Furthermore, for each group we use the bootstrapping method to estimate the 90\% confidence interval of each entry in our estimate. The results are plotted in Figure \ref{example1-2obs}, where the region bounded by the two dashed sky blue curves is the estimated $90\%$-confidence interval based on estimates in group 2. The region bounded by the two dashed red curves is for the sample means in group 1. Hence, in accordance with the definition of a confidence interval, each entry of the true solution lies in the shaded region with a probability of $90\%$. The $R^2$-statistic between $\mathbf{x}_1$ and $\bar{\mathbf{x}}^\dagger$ (the finite analog of exact pre-defined solution $x(t)$ in \eqref{AutoConv-x}) is $97.7\%$ and the $R^2$-statistic between $-\bar{\mathbf{x}}^\dagger$ and $\bar{\mathbf{x}}_2$ is $98.4\%$, both of which are high.

\begin{figure}[h]
	\centering
	{
		\subfigure[]{
			\includegraphics[scale=0.06]{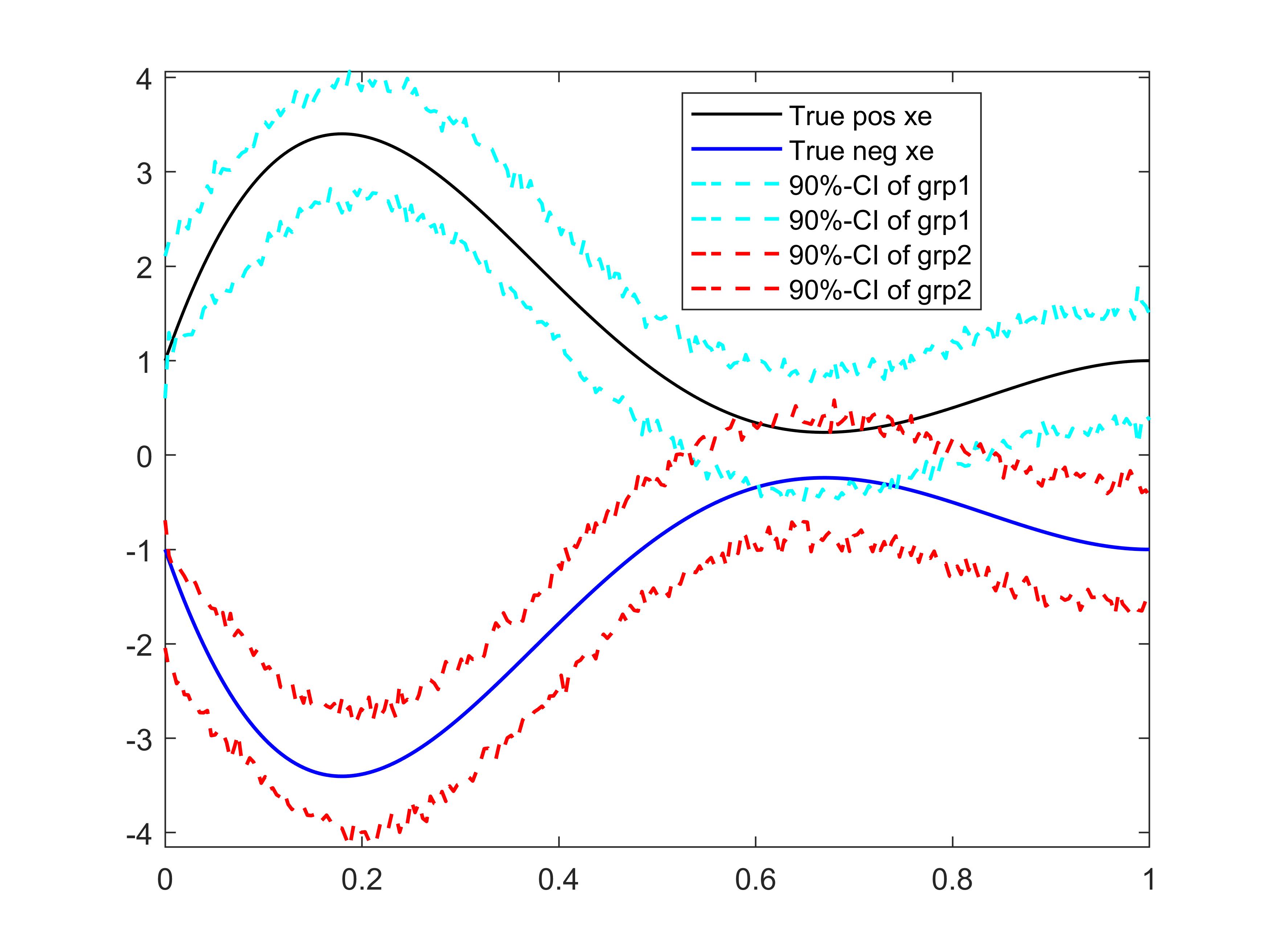} }
   	\subfigure[]{
			\includegraphics[scale=0.06]{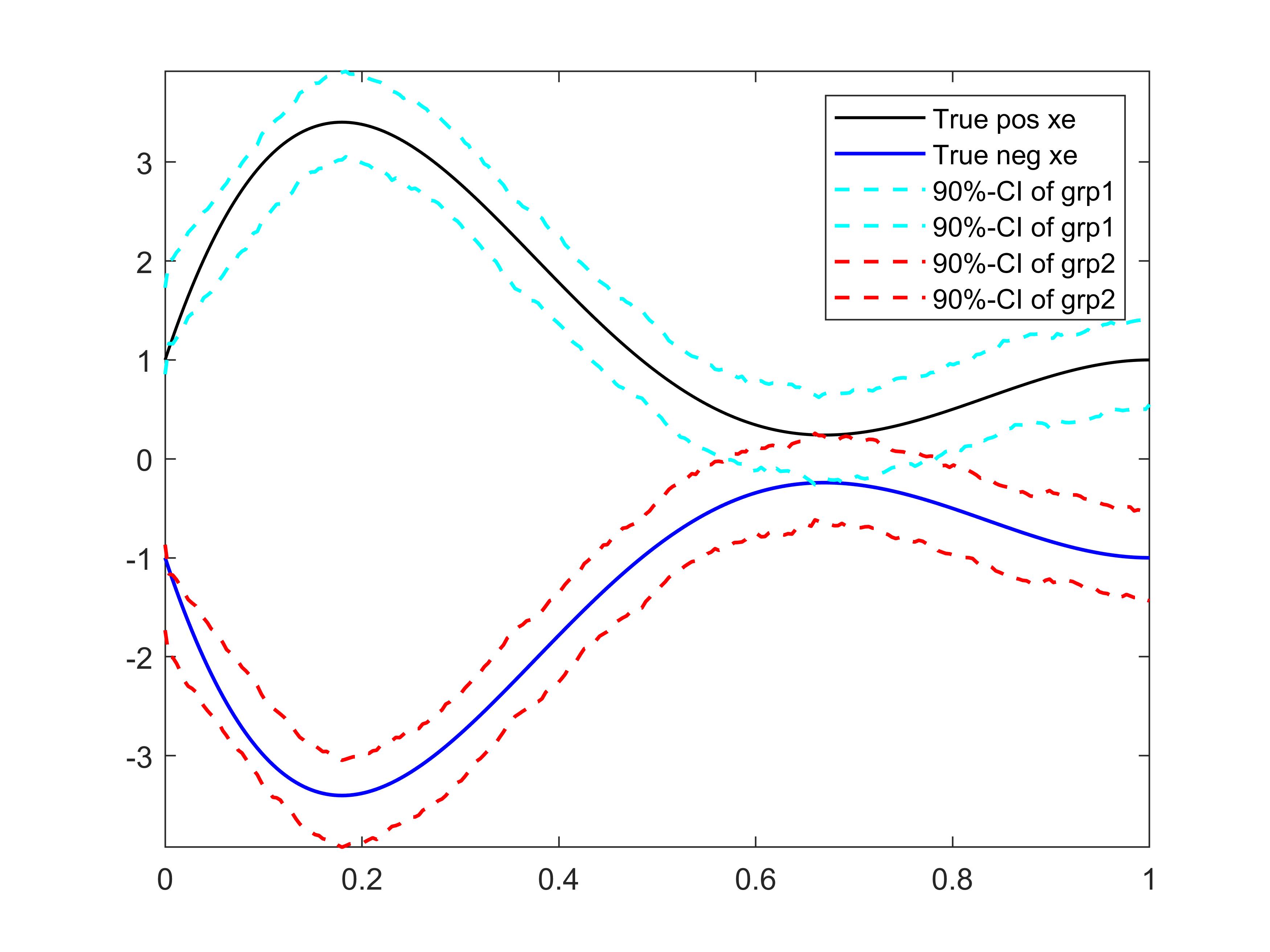} }
		\caption{The Estimations and True Values. $\delta=5\%$, $\theta=0.5$, $\tau=1.1$. (a) $N=100$; (b) $N=5000$.}
		\label{example1-2obs}
	}
\end{figure}

At the end of this section, we note that it is possible to identify another set of two exact functions corresponding to the approximate solutions shown in Figure \ref{example1-2obs}. However, any such sets of exact functions tend to exhibit similar characteristics. Furthermore, in most real-world problems, the number and form of exact solutions are unknown. In these scenarios, it suffices to identify approximate solutions with distinctly noticeable features. The primary interest lies in interpreting the physical phenomenon by examining the properties of these reconstructed solutions.

\section{Conclusion and outlook}
\label{conclusion}

We develop a statistical approach~-- SAR for deterministic nonlinear inverse problems. In this paper, we present and analyze a comprehensive regularization theory of SAR. Specifically, we demonstrate that SAR provides an optimal regularization method with respect to mean-square convergence. Compared to conventional deterministic regularization methods, SAR offers three key advantages: first, it enables uncertainty quantification of the estimated quantity, which is modeled by a deterministic forward model; second, SAR can numerically avoid local minima, a common challenge in most nonlinear inverse problems; and third, SAR can identify multiple solutions by clustering solution samples. Thus, SAR serves as a valuable tool for studying practical nonlinear inverse problems with deterministic ill-posed forward models.

Based on the convergence analysis in Sections \ref{sec-convergence}-\ref{sec-convergencerate}, proving the regularization properties of SAR for inverse problems with stochastic forward models is straightforward. However, a compelling question is how to design an efficient SAR with specific parameters that must be carefully selected according to the interplay of randomness among data noise, the forward model, and the method itself. This careful selection is essential to efficiently quantify uncertainties arising from the randomness of both the noise and the forward model.

\section*{Appendix}

\subsection*{Appendix A: Proofs of Lemmas \ref{lemma9} and \ref{lem1}}

Lemma \ref{lemma9} follows from \cite[Formula 4.8]{UT1994}. Now, we prove Lemma \ref{lem1}. We distinguish between two cases: (i) $\int_0^{\frac{t}{2}} \frac{ds}{(1+t-s)^k (1+s)^j}\leq \int_{\frac{t}{2}}^t \frac{ds}{(1+t-s)^k (1+s)^j}$ and (ii) $\int_0^{\frac{t}{2}} \frac{ds}{(1+t-s)^k (1+s)^j}> \int_{\frac{t}{2}}^t \frac{ds}{(1+t-s)^k (1+s)^j}$.

For the first case, we have $1+s\leq 1+t-s$, and hence
\begin{equation*}
\begin{split}
&\int_0^{t} \frac{ds}{(1+t-s)^k (1+s)^j}\leq 2 \int_{\frac{t}{2}}^t \frac{ds}{(1+t-s)^k (1+s)^j} \leq 2 \int_{\frac{t}{2}}^t \frac{ds}{(1+s)^{k+j}} \\
& \qquad  = \frac{2}{k+j-1} \left( \frac{1}{\left(1+\frac{t}{2}\right)^{k+j-1}} - \frac{1}{\left(1+t\right)^{k+j-1}} \right) < \frac{2^{k+j}-2}{k+j-1} \cdot \frac{1}{\left(1+t\right)^{k+j-1}}.
\end{split}
\end{equation*}

For the second case, by noting that $1+s\geq 1+t-s$, we can obtain the same inequality, i.e.
\begin{equation*}
\begin{split}
&\int_0^{t} \frac{ds}{(1+t-s)^k (1+s)^j}< 2 \int_0^{\frac{t}{2}} \frac{ds}{(1+t-s)^k (1+s)^j} \leq 2 \int_0^{\frac{t}{2}} \frac{ds}{(1+t-s)^{k+j}} \\
& \qquad  = \frac{2}{k+j-1} \left( \frac{1}{\left(1+\frac{t}{2}\right)^{k+j-1}} - \frac{1}{\left(1+t\right)^{k+j-1}} \right) < \frac{2^{k+j}-2}{k+j-1} \cdot \frac{1}{\left(1+t\right)^{k+j-1}}.
\end{split}
\end{equation*}

\subsection*{Appendix B: Proof of Proposition \ref{PropositionBiasErr}}

It follows from Lemma \ref{lemma9} and Assumption \ref{assumption-2} that
\begin{equation}\label{I1}
\left\| {{I_1}} \right\| = \left\| V^\sigma{{e^{ - Vt}}\left( {\bar x - {x^\dag }} \right)} \right\|
\le \mathop {\sup }\limits_{0 \le \lambda  \le 1} \left| {{\lambda ^{\gamma+\sigma} }{e^{ - \lambda t}}} \right| \cdot \left\| \nu  \right\|  \le \frac{c^{a}_\gamma E}{{{{\left( {1 + t} \right)}^{\gamma+\sigma} }}}.
\end{equation}

Applying Lemma \ref{lemma9} again, we obtain, for $\sigma\neq1/2$,
\begin{equation}\label{I2}
\begin{split}
\left\| {{I_2}} \right\|& = \left\| {\int_0^t {{e^{ - V\left( {t - s} \right)}}} V^{\sigma+\frac{1}{2}}\left( {{y^\delta } - y} \right)\textmd{ds}} \right\| \\
&\le  {\int_0^t \!\mathop {\sup }\limits_{\lambda  > 0} \left|{\lambda^{\sigma+\frac{1}{2}} {e^{ - {\lambda}\left( {t - s} \right)}}}\right|\textmd{ ds}}  \cdot \left\| {{y^\delta } - y} \right\| \\
&\le c^b_\gamma \delta {\int_0^t \frac{1}{(1+t-s)^{\sigma+\frac{1}{2}}} \textmd{ds}}  \le c^b_\gamma\delta (1+t)^{\frac{1}{2}-\sigma}.
\end{split}
\end{equation}
In the case where $\sigma=1/2$, \eqref{I2} remains valid, according to
\begin{equation*}
\left\| {{I_2}} \right\| = \left\| {\int_0^t {{e^{ - V\left( {t - s} \right)}}} V\left( {{y^\delta } - y} \right)\textmd{ds}} \right\| = \left\| (1-e^{-Vt})  ({y^\delta } - y) \right\| \le \delta \le c^b_\gamma \delta.
\end{equation*}

Now it remains to estimate $I_3$. For this purpose, we first have to bound the function $\psi(s)$ defined in \eqref{psi}. From \eqref{A32} and \eqref{rem1}, we derive
\begin{equation*}
\begin{split}
\left\| {\psi \left( s \right)} \right\| &= \left\| {\left( {I - R_{{x^\delta }\left( s \right)}^*} \right)\left[ {F\left( {{x^\delta }\left( s \right)} \right) - {y^\delta }} \right] - \left[ {F\left( {{x^\delta }\left( s \right)} \right) - y -V^{\frac{1}{2}}\left( {{x^\delta }\left( s \right) - {x^\dag }} \right)} \right]} \right\|\\ \nonumber
&\le \left\| {\left( {I \!-\! R_{{x^\delta }\left( s \right)}^*} \right)\left[ {F\left( {{x^\delta }\left( s \right)} \right) \!-\! {y^\delta }} \right]} \right\| \!+\! \left\| {F\left( {{x^\delta }\left( s \right)} \right) \!-\! F(x^\dag) \!-\! V^{\frac{1}{2}}\left( {{x^\delta }\left( s \right)\! -\! {x^\dag }} \right)} \right\|\\ \nonumber
&\le  {{{c_R}}\left\| {{{x^\delta }\left( s \right)}\! -\! {x^\dag }} \right\|\left\| {F\left( {{x^\delta }\left( s \right)} \right) - {y^\delta }}\right\|+ \frac{c_R}{2}} \left\| {{{x^\delta }\left( s \right)} - {x^\dag }} \right\|\left\| {V^{\frac{1}{2}}\left( {{{x^\delta }\left( s \right)} - {x^\dag }} \right)} \right\|.
\end{split}
\end{equation*}

Note that, according to the stopping rule in \eqref{stopping}, and before the terminating time, we have
\begin{equation}\label{stopingTime}
\mathbb{E} [{{{\| {F\left( {{x^\delta }\left( t \right)} \right) - {y^\delta }} \|}}}^2] > \tau^2\delta^2,
\end{equation}
which gives
\begin{equation*}
\begin{split}
& \mathbb{E}[\| {F\left( {{x^\delta }\left( s \right)} \right) - {y^\delta }}\|^2] \leq 2\mathbb{E}[\| {F\left( {{x^\delta }\left( s \right)} \right) - y}\|^2] + 2 \delta^2 \\ & \qquad < 2\mathbb{E}[\| {F\left( {{x^\delta }\left( s \right)} \right) - y}\|^2] + \frac{2}{\tau^2} \mathbb{E} [{{{\| {F\left( {{x^\delta }\left( t \right)} \right) - {y^\delta }} \|}}}^2],
\end{split}
\end{equation*}
which further implies, together with the choice of $\epsilon_0$ and Assumption \ref{assumption-2}, that
\begin{equation}\label{IneqDataSolu}
\begin{split}
& \left( 1 - \frac{2}{\tau^2} \right) \mathbb{E}[\| {F\left( {{x^\delta }\left( s \right)} \right) - {y^\delta }}\|^2] \leq 2\mathbb{E}[\| {F\left( {{x^\delta }\left( s \right)} \right) - y}\|^2] \\ & \qquad \leq \frac{2}{(1-\eta)^2} \mathbb{E}[\|{V^{\frac{1}{2}}\left( {{{x^\delta }\left( s \right)} - {x^\dag }} \right)} \|^2].
\end{split}
\end{equation}

Hence, we obtain
\begin{equation}\label{cc}
\begin{split}
& \mathbb{E}[\left\| {\psi \left( s \right)} \right\| ]
\le {c_R} \mathbb{E}[\| {{{x^\delta }\left( s \right)} - {x^\dag }} \|^2]^\frac{1}{2} \mathbb{E}[\| {F\left( {{x^\delta }\left( s \right)} \right) - {y^\delta }}\|^2]^\frac{1}{2}\\
& \qquad + \frac{c_R}{2} \mathbb{E}[\| {{{x^\delta }\left( s \right)} - {x^\dag }} \|^2]^\frac{1}{2} \mathbb{E}[\|{V^{\frac{1}{2}}\left( {{{x^\delta }\left( s \right)} - {x^\dag }} \right)} \|^2]^\frac{1}{2} \le c_0 p(s)q(s),
\end{split}
\end{equation}
with
\[{c_0}: = \left( {{\frac{\sqrt{2}\tau }{\sqrt{\left( {\tau^2  - 2} \right)} \left( {1 - \eta } \right)} }+\frac{1}{2}} \right) c_R.\]
Once again, from Lemma \ref{lemma9}, we obtain
\begin{equation}
\begin{split}
\left\| {{I_3}} \right\|& = \left\|\mathbb{E}{\int_0^t {{e^{ - V\left( {t- s} \right)}}} V^{\sigma+\frac{1}{2}}\psi \left( s \right)\textmd{ds}} \right\| \le \int_0^t \left\|{{e^{ - V\left( {t- s} \right)}}}  V^{\sigma+\frac{1}{2}}\right\|\mathbb{E}[\|{\psi \left( s \right)}\|]\textmd{ds} \\
&\le\int_0^t \!{\mathop {\sup }\limits_{0 \le \lambda  \le 1} \left| {{\lambda ^{\sigma+\frac{1}{2}}}{e^{ - \lambda \left( {t - s} \right)}}} \right| \cdot \mathbb{E}[\left\| {\psi \left( s \right)} \right\|]} \textmd{ds}\le {c_0}c^b_\gamma \int_0^t\! \! {\frac{{{p}\left( s \right){q}\left( s \right)}}{(1+t-s)^{\sigma+\frac{1}{2}}}} \textmd{ds}.
\end{split}
\end{equation}

On the other hand, combining \eqref{stopingTime} and \eqref{IneqDataSolu}, we derive
\begin{equation}\label{IneqDelta}
\delta^2< \frac{1}{\tau^2}\mathbb{E} [{{{\| {F\left( {{x^\delta }\left( t \right)} \right) - {y^\delta }} \|}}}^2] \leq \frac{2}{(1-\eta)^2(\tau^2-2)} \mathbb{E}[\|{V^{\frac{1}{2}}\left( {{{x^\delta }\left( t \right)} - {x^\dag }} \right)} \|^2].
\end{equation}

Summarizing the above results, and using formula \eqref{total error}, we conclude that
\begin{equation*}
\begin{split}
&\left\| {\mathbb{E}(V^\sigma({x^\delta }\left( t \right) - {x^\dag }))}\right\| \le \left\| {{I_1}} \right\| + \left\| {{I_2}} \right\| + \left\| {{I_3}} \right\|\\
&\le  \frac{c^a_\gamma E}{{{{\left( {1 + t} \right)}^{\gamma+\sigma} }}} + c^b_\gamma \delta (1+t)^{\frac{1}{2}-\sigma} + {c_0}c^b_\gamma \int_0^t {\frac{{p\left( s \right)q\left( s \right)}}{(1+t-s)^{\sigma+\frac{1}{2}}}} \textmd{ds}\\
&\le  \frac{c^a_\gamma E}{{{{\left( {1 + t} \right)}^{\gamma+\sigma} }}} + \frac{\sqrt{2}c^b_\gamma}{(1-\eta)\sqrt{\tau^2-2}}  q\left( t \right)(1+t)^{\frac{1}{2}-\sigma}  + {c_0}c^b_\gamma \int_0^t {\frac{{p\left( s \right)q\left( s \right)}}{(1+t-s)^{\sigma+\frac{1}{2}}}} \textmd{ds}.
\end{split}
\end{equation*}

\subsection*{Appendix C: Proof of Proposition \ref{ProfIneqpq}}

Let
\begin{equation*}
\begin{split}
G_1(t,p,q):=& \left(1+\sqrt{\frac{4^{\gamma} -1}{\gamma} } c_3 \right) c^b_\gamma \frac{E}{(1+t)^{\gamma}}+ {c_1}\sqrt{t+1} q(t) + {c_2}\int_0^t {\frac{p(s)q(s)}{{\sqrt {1 + t - s} }}} \textmd{ds},\\
G_2(t,p,q):=&\left(1+\sqrt{\frac{2^{2\gamma+2} -2}{2\gamma+1} } c_3\right) c^b_\gamma \frac{E}{(1+t)^{\gamma+\frac{1}{2}}}+c_1 q(t)+ {c_2}\int_0^t {\frac{p(s)q(s)}{ {1 + t - s} }} \textmd{ds}.
\end{split}
\end{equation*}

\begin{lemma}\label{thm2}
	The functions $p(t)$ and $q(t)$ satisfy the following system of integral inequalities of the second kind:
	\begin{equation}\label{G1}
	p(t)\le  G_1(t,p,q),
	\end{equation}
	\begin{equation}\label{G2}
	q(t)\le G_2(t,p,q).
	\end{equation}
\end{lemma}

\begin{proof}
	Taking $\sigma=0$ and $\sigma=1/2$ respectively in Proposition \ref{po2}, it follows that
	\begin{equation}
	\label{pqGIneq1}
	\mathbb{E}[\|{x^\delta }\left( t \right) - \mathbb{E}{x^\delta }\left( t \right)\|^2]^{\frac{1}{2}}\le \sqrt{\frac{4^{\gamma} -1}{\gamma} } c_3 c^b_\gamma \frac{E}{(1+t)^{\gamma}},
	\end{equation}
	and
	\begin{equation}
	\label{pqGIneq2}
	\mathbb{E}[\|V^\frac{1}{2}({x^\delta }\left( t \right) - \mathbb{E}{x^\delta }\left( t \right))\|^2]^{\frac{1}{2}}\le \sqrt{\frac{2^{2\gamma+2} -2}{2\gamma+1} } c_3 c^b_\gamma \frac{E}{(1+t)^{\gamma+\frac{1}{2}}}.
	\end{equation}
	
	Combine \eqref{pqGIneq1} and Proposition \ref{PropositionBiasErr} with $\sigma=0$, and, together with the fact that $\sqrt{a+b}\leq \sqrt{a}+\sqrt{b}$, we can obtain
	\begin{equation*}
	\begin{split}
	& \mathbb{E}[\|{x^\delta }\left( t \right) - {x^\dag }\|^2]^\frac{1}{2} \le\|\mathbb{E}{x^\delta }\left( t \right) - {x^\dag }\|+ \mathbb{E}[\|{x^\delta }\left( t \right) - \mathbb{E}{x^\delta }\left( t \right)\|^2]^\frac{1}{2}\\
	& \qquad \le \left(1+ \sqrt{\frac{4^{\gamma} -1}{\gamma} } c_3 \right)c^b_\gamma \frac{E}{(1+t)^{\gamma}}+ {c_1}\sqrt{t+1} q(t) + {c_2}\int_0^t {\frac{p(s)q(s)}{{\sqrt {1 + t - s} }}} \textmd{ds},
	\end{split}
	\end{equation*}
	which yields \eqref{G1}. In a similar way, combine \eqref{pqGIneq1} and Proposition \ref{PropositionBiasErr} with $\sigma=0$ to derive
	\begin{equation*}
	\begin{split}
	& \mathbb{E}[\|V^\frac{1}{2}({x^\delta }\left( t \right) - {x^\dag })\|^2]^\frac{1}{2}
	\le \|\mathbb{E}[V^\frac{1}{2}({x^\delta }\left( t \right) - {x^\dag })]\|+ \mathbb{E}[\|V^\frac{1}{2}({x^\delta }\left( t \right) - \mathbb{E}{x^\delta }\left( t \right))\|^2]^\frac{1}{2}\\
	& \quad \le \left(1+\sqrt{\frac{2^{2\gamma+2} -2}{2\gamma+1} } c_3 \right)c^b_\gamma  \frac{E}{(1+t)^{\gamma+\frac{1}{2}}} + c_1  q(t) + {c_2} \int_0^t {\frac{p(s)q(s)}{ {1 + t - s} }} \textmd{ds},
	\end{split}
	\end{equation*}
	which implies the desired assertion \eqref{G2}.
\end{proof}

Using the above lemma, we can prove Proposition \ref{ProfIneqpq}.
\begin{proof}
	We first show that
\begin{equation}
\label{Gest}
	G_1(t,z_1,z_2)\le z_1(t)~~\textmd{and}~~G_2(t,z_1,z_2)\le z_2(t).
	\end{equation}
	From \eqref{G1} and \eqref{G2}, we derive, with Lemma \ref{lem1} (by setting $(k,j)=(\frac{1}{2}, 2\gamma + \frac{1}{2})$ and $(k,j)=(1, 2\gamma + \frac{1}{2})$ in the inequality \eqref{A1}, respectively),
	\begin{equation*}
	\begin{split}
	& G_1(t,z_1,z_2) \\
	& \quad = \left(1+\sqrt{\frac{4^{\gamma} -1}{\gamma}} c_3 \right)c^b_\gamma \frac{E}{(1+t)^{\gamma}}+ {c_1} \sqrt{t+1} \cdot\frac{c^* E}{(1+t)^{\gamma+\frac{1}{2}}} + {c_2}\int_0^t {\frac{{\frac{c^* E}{(1+s)^\gamma} \cdot \frac{c^* E}{(1+s)^{\gamma+\frac{1}{2}}}}}{{\sqrt {1 + t - s} }}} \textmd{ds}\\
	& \quad  \le \left(1+\sqrt{\frac{4^{\gamma} -1}{\gamma}} c_3\right) c^b_\gamma \frac{E}{(1+t)^{\gamma}}+ {c_1}  c^*\frac{ E}{(1+t)^{\gamma}} +  c_2 (c^*)^2 E \frac{4^\gamma -1}{\gamma} \frac{E}{(1+t)^\gamma}\\
	& \quad = \left(c^b_\gamma+\sqrt{\frac{4^\gamma -1}{\gamma}} c_3 c^b_\gamma + {c_1}  c^*+ c_2 (c^*)^2 \frac{4^\gamma -1}{\gamma} E \right)\frac{E}{(1+t)^\gamma}
	\end{split}
	\end{equation*}
	and
	\begin{equation*}
	\begin{split}
	& G_2(t,z_1,z_2)\\
	& \quad = \left(1+\sqrt{\frac{2^{2\gamma+2} -2}{2\gamma+1} } c_3\right) c^b_\gamma \frac{E}{(1+t)^{\gamma+\frac{1}{2}}}+ {c_1}\frac{c^* E}{(1+t)^{\gamma+\frac{1}{2}}} + {c_2}\int_0^t {\frac{{\frac{c^* E}{(1+t)^\gamma} \cdot \frac{c^* E}{(1+t)^{\gamma+\frac{1}{2}}}}}{{ {1 + t - s} }}} \textmd{ds}\\
	&  \quad \le \left(1+\sqrt{\frac{2^{2\gamma+2} -2}{2\gamma+1} } c_3\right) c^b_\gamma  \frac{E}{(1+t)^{\gamma+\frac{1}{2}}}+ {c_1}  c^*\frac{ E}{(1+t)^{\gamma+\frac{1}{2}}} +  c_2 (c^*)^2 E \frac{2^{2\gamma+\frac{5}{2}} -4}{4\gamma+1} \frac{E}{(1+t)^{\gamma+\frac{1}{2}}}\\
	&  \quad = \left(c^b_\gamma+\sqrt{ \frac{2^{2\gamma+\frac{5}{2}} -4}{4\gamma+1}} c_3 c^b_\gamma + {c_1}  c^*+ c_2 (c^*)^2 E \frac{2^{2\gamma+\frac{5}{2}} -4}{4\gamma+1} \right)\frac{E}{(1+t)^{\gamma+\frac{1}{2}}}.
	\end{split}
	\end{equation*}
	From the definition of $c_1$ in Proposition \ref{PropositionBiasErr}, for $\tau>\sqrt{\frac{\sqrt{2}c^b_\gamma}{(1-\eta)} + 2}$, $c_1<1$. Thus, for a small-enough $E$ such that $E\leq\frac{1-c_1}{2c_2c_\Gamma c^*}$, we can choose $c^*\geq\frac{2c^b_\gamma(1 + \sqrt{ c_\Gamma} c_3)}{1-c_1}$ so that the following inequality holds:
	\begin{equation}
	\label{cGamma}
	c^b_\gamma + \sqrt{ c_\Gamma} c_3 c^b_\gamma + {c_1}  c^*+ c_2 (c^*)^2 E c_\Gamma \le c^*,
	\end{equation}
	where $c_\Gamma:=\max\{\frac{4^\gamma -1}{\gamma}, \frac{2^{2\gamma+\frac{5}{2}} -4}{4\gamma+1} \}$.
	This leads to the estimates \eqref{Gest}. Consequently, our assertions \eqref{pest} and \eqref{qest} follow immediately via \eqref{Gest}, Lemma \ref{thm2}, and the monotonicity of $G_1(t,p,q)$ and  $G_2(t,p,q)$ with respect to $p$ and $q$.
\end{proof}

\subsection*{Appendix D: An example of constants with given parameters}~

\begin{longtable}{|c|c|l|}
	\hline \multicolumn{1}{|c|}{\textbf{Constant}} & \multicolumn{1}{c|}{\textbf{Value }} & \multicolumn{1}{c|}{\textbf{Reference}} \\
	\hline
	\hspace{-3.5mm} \begin{tabular}{l}  $\epsilon_0$   \end{tabular} & \hspace{-3.5mm} \begin{tabular}{l}  1       \end{tabular}   & \eqref{gt} in the definition of $g(t)$ \\
	\hline
	\hspace{-3.5mm} \begin{tabular}{l}  $g(0)$   \end{tabular} & \hspace{-3.5mm} \begin{tabular}{l}  0.7071      \end{tabular}   & \eqref{gt} in the definition of $g(t)$ \\
	\hline
	\renewcommand{\arraystretch}{1.2} \hspace{-4.5mm}  \begin{tabular}{l} $\tau$ \end{tabular} & \hspace{-3.5mm} \begin{tabular}{l} 6 \end{tabular} & \eqref{hstopping} in the stopping rule \\
	\hline
	\hspace{-3.5mm} \begin{tabular}{l}    $c_\gamma$ \end{tabular} & \hspace{-3.5mm} \begin{tabular}{l}  1 \end{tabular}  & Ineq.\eqref{use1} in Lemma \ref{lemma9} \\
	\hline
	\hspace{-3.5mm} \begin{tabular}{l}  $c^a_\gamma$  \end{tabular} & \hspace{-3.5mm} \begin{tabular}{l}  1 \end{tabular}  &  Ineq.\eqref{I1} in Proposition \ref{PropositionBiasErr} \\
	\hline
	\hspace{-3.5mm} \begin{tabular}{l}  $c^b_\gamma$ \end{tabular} & \hspace{-3.5mm} \begin{tabular}{l} 1  \end{tabular} & Ineq.\eqref{I2} in Lemma \ref{lemma9} \\
	\hline
	\hspace{-3.5mm} \begin{tabular}{l}  $c_0$ \end{tabular} & \hspace{-3.5mm} \begin{tabular}{l} 3.4104 \end{tabular} & Ineq.\eqref{cc} in the proof of Proposition \ref{PropositionBiasErr} \\
	\hline
	\hspace{-3.5mm} \begin{tabular}{l}  $c_1$ \end{tabular} & \hspace{-3.5mm} \begin{tabular}{l} 0.4851 \end{tabular} & Proposition \ref{PropositionBiasErr} \\
	\hline
	\hspace{-3.5mm} \begin{tabular}{l}  $c_2$ \end{tabular} & \hspace{-3.5mm} \begin{tabular}{l} 3.4104  \end{tabular} & Proposition \ref{PropositionBiasErr} \\
	\hline
	\hspace{-3.5mm} \begin{tabular}{l}  $c_3$ \end{tabular} & \hspace{-3.5mm} \begin{tabular}{l} 2.0580  \end{tabular} & Proposition \ref{po2} \\
	\hline
	\hspace{-3.5mm} \begin{tabular}{l}  $c_\Gamma$ \end{tabular} & \hspace{-3.5mm} \begin{tabular}{l} 2.1342  \end{tabular} & Ineq.\eqref{cGamma} in the proof of Proposition \ref{ProfIneqpq} \\
	\hline
	\hspace{-3.5mm} \begin{tabular}{l}  $c^*$ \end{tabular} & \hspace{-3.5mm} \begin{tabular}{l} 15.5612  \end{tabular} & Ineq.\eqref{cGamma} in the proof of Proposition \ref{ProfIneqpq} \\
	\hline
	\hspace{-3.5mm} \begin{tabular}{l}  $E$ \end{tabular} & \hspace{-3.5mm} \begin{tabular}{l} 0.0023  \end{tabular} & Assumption \ref{assumption-2} and Proposition \ref{ProfIneqpq} \\
	\hline
	\hspace{-3.5mm} \begin{tabular}{l}  $c^0_e$ \end{tabular} & \hspace{-3.5mm} \begin{tabular}{l} 9.8960  \end{tabular} & \eqref{ce0} in the proof of Proposition \ref{PropRvar} \\
	\hline
	\hspace{-3.5mm} \begin{tabular}{l}  $c^1_e$ \end{tabular} & \hspace{-3.5mm} \begin{tabular}{l} 3.9277  \end{tabular} & \eqref{c1e} in the proof of Theorem \ref{thmrate} \\
	\hline
	\hspace{-3.5mm} \begin{tabular}{l}  $c^2_e$ \end{tabular} & \hspace{-3.5mm} \begin{tabular}{l} 11.5000  \end{tabular} & \eqref{c2e} in the proof of Theorem \ref{thmrate} \\
	\hline
	\hspace{-3.5mm} \begin{tabular}{l}  $c^3_e$ \end{tabular} & \hspace{-3.5mm} \begin{tabular}{l} 6.0362  \end{tabular} & \eqref{rate2} in the proof of Theorem \ref{thmrate} \\
	\hline
	\hspace{-3.5mm} \begin{tabular}{l}  $c^4_e$ \end{tabular} & \hspace{-3.5mm} \begin{tabular}{l} 9.3990  \end{tabular} & \eqref{ce4} in the proof of Theorem \ref{thmrate} \\
	\hline
	\hspace{-3.5mm} \begin{tabular}{l}  $c_e$ \end{tabular} & \hspace{-3.5mm} \begin{tabular}{l} 13.6482  \end{tabular} & \eqref{c_e} in the proof of Theorem \ref{thmrate} \\
	\hline
	\caption{Constants with given parameters $(\eta,\delta_0,\gamma,c_R)=(\frac{1}{2},1,\frac{1}{3},1)$ and references to their definitions.}
	\label{NotationTable}
\end{longtable}

\subsection*{Appendix E: Numerical experiments about the impacts of the level of randomization and the sample size on SAR}~

It is clear that when the sample size is limited, the approximation accuracy of SAR depends essentially on the level of randomization $\theta$. To visualize its impact, we display in Figures \ref{1Dd2N800theta} and \ref{1Dd1N200theta} the numerical results for the 1D and 2D examples with different $\theta$ values. For not large enough sample size, the larger the $\theta$ value is, the lower the accuracy and the smoothness of the approximation solution. However, a $\theta$ value that is too small may improve the accuracy only slightly, and hence a desirable $\theta$ value should be chosen carefully for convergence.

In practice, the fluctuations caused by the stochastic term may cause problems with a proper stopping rule, especially with only one single sample. We therefore try to determine the role of the sample size. Indeed, as shown in \cite[Section 2]{ZhangChen23}, for a fixed noisy data, the sample size is not necessary to be large. For linear inverse problems, an appropriate
order of magnitude of sample size is recommended by \cite[Proposition 2]{ZhangChen23} for the asymptotically optimal convergence rate. For nonlinear inverse problems \eqref{model}, we only numerically investigate the impact of the sample size in the discrete version of SAR, namely the iteration \eqref{EM}. First, we present in Figures \ref{1Dd2path} and \ref{2Dd1CIpath} the numerical results for the 1D and 2D cases respectively, using different $N$ values.
The smaller the $N$ value is, the lower the attainable accuracy of the algorithm.
Nevertheless, an $N$ that is too large could improve the accuracy very slightly, but with greater computational effort; see Tables \ref{tab1} and \ref{tab2} for more details. Hence, as in the linear scenario, a suitable sample size $N$ should be chosen carefully.

\begin{figure}[t]
	\centering
	{
		\subfigure[$\theta=0.002$]{
			\includegraphics[scale = 0.06]{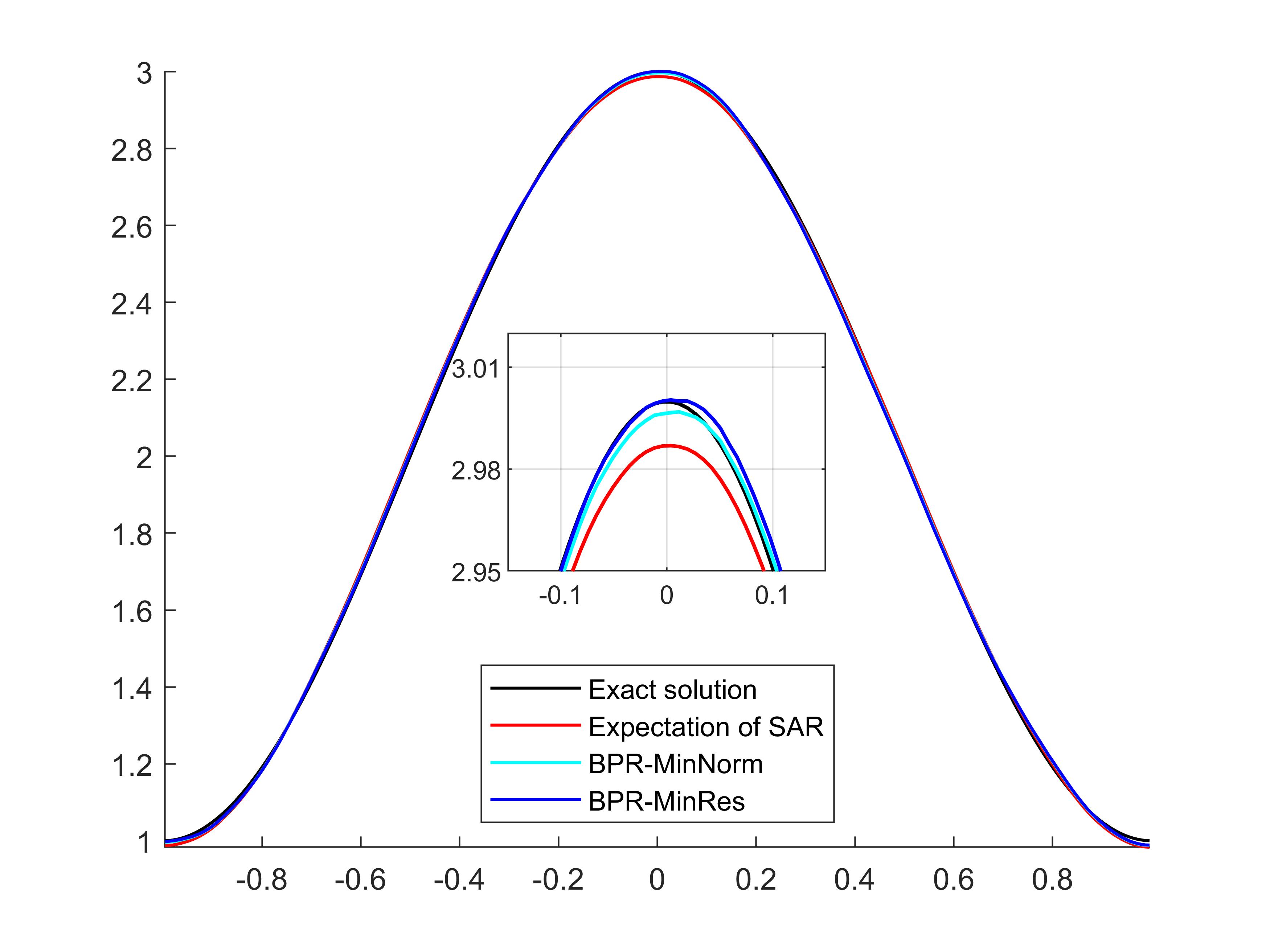} }
		\subfigure[$\theta=0.02$]{
			\includegraphics[scale = 0.06]{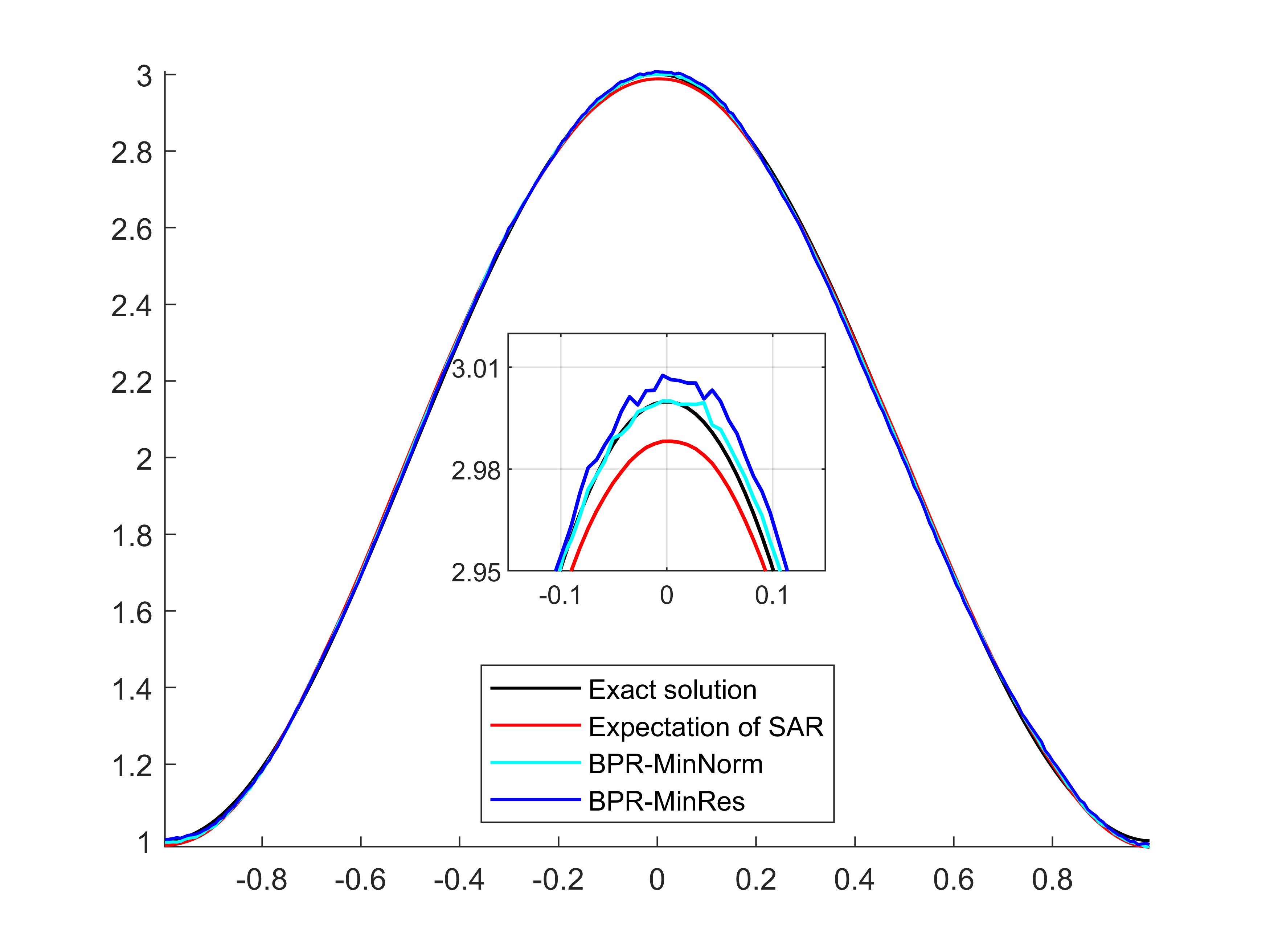} }
		\subfigure[$\theta=0.2$]{
			\includegraphics[scale = 0.06]{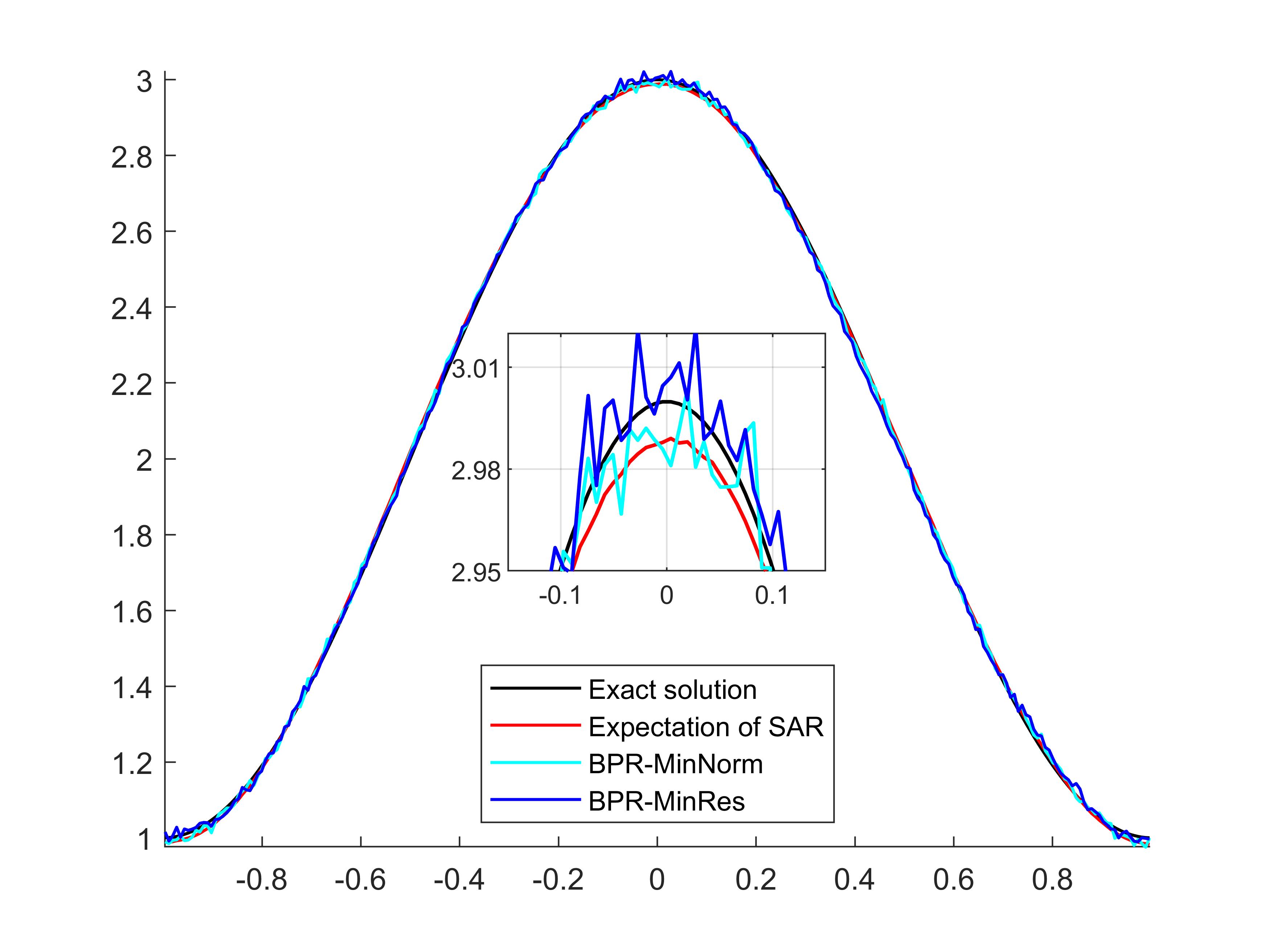} }
		\subfigure[$\theta=2.0$]{
			\includegraphics[scale = 0.06]{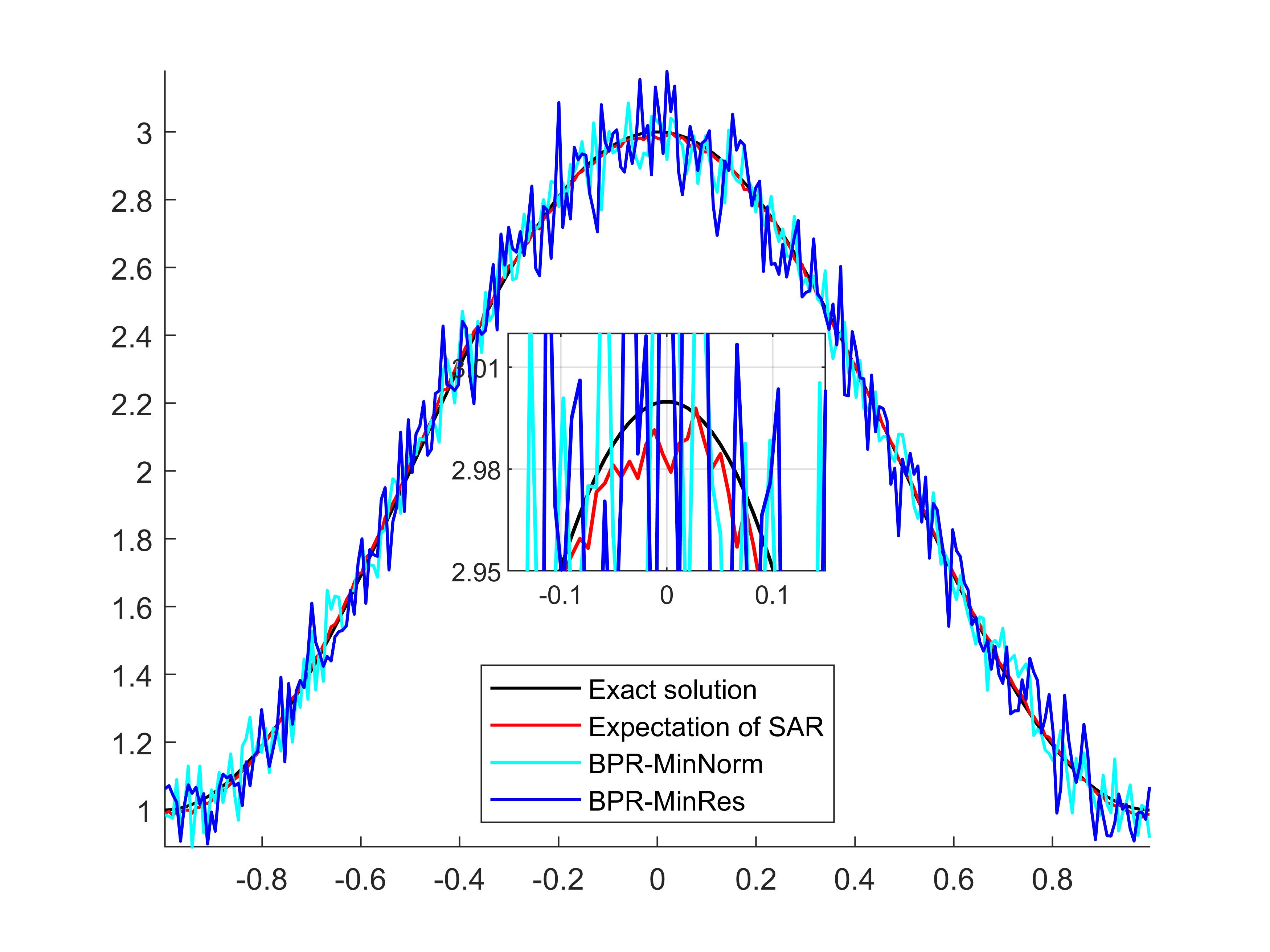} }
		\caption{Reconstruction results for the 1D case, with $\delta = 2\%$ and $N=800$.} 
		\label{1Dd2N800theta}
	}
\end{figure}

According to the numerical results, similar to the rule recommended in \cite{ZhangChen23}, for the considered model problem with given parameters, the sample size $n$ can be chosen  according to the following formula, which will be used in the thereafter numerical experiments. 
\begin{eqnarray}
	& & n = \max ( \delta^{-1}, n_0), \quad n_0:= \mathop{\arg\min}_{n\in \Theta_{k_1} \cap \mathbb{N}_{k_2}} \, n, \quad k_1=10, \quad k_2=50, \label{sampleN} \\
	& & \Theta_{k_1}: = \left\{ n_0: \frac{\| \bar{\mathbf{c}}^{n_0}_i - \bar{\mathbf{c}}^{n_0+k}_i \|}{\|\bar{\mathbf{c}}^{n_0}_i\|} \leq 0.01, \frac{\| s^{n_0}_i - s^{n_0+k}_i \|}{{\|\mathbf{s}}^{n_0}_i\|} \leq 0.05, \quad k=1, \cdots, k_1 \right\}, \nonumber
\end{eqnarray}
where $\bar{\mathbf{c}}^{n_0}_i$ and $s^{n_0}_i$ denote the sample mean and sample variance with a sample size of $n_0$, respectively. $\mathbb{N}_{k_2}$ denotes the set of integer multiples of $k_2$.

\begin{figure}[p]
	\centering
	{
		\subfigure[$\theta=0.005$]{
			\includegraphics[scale = 0.05]{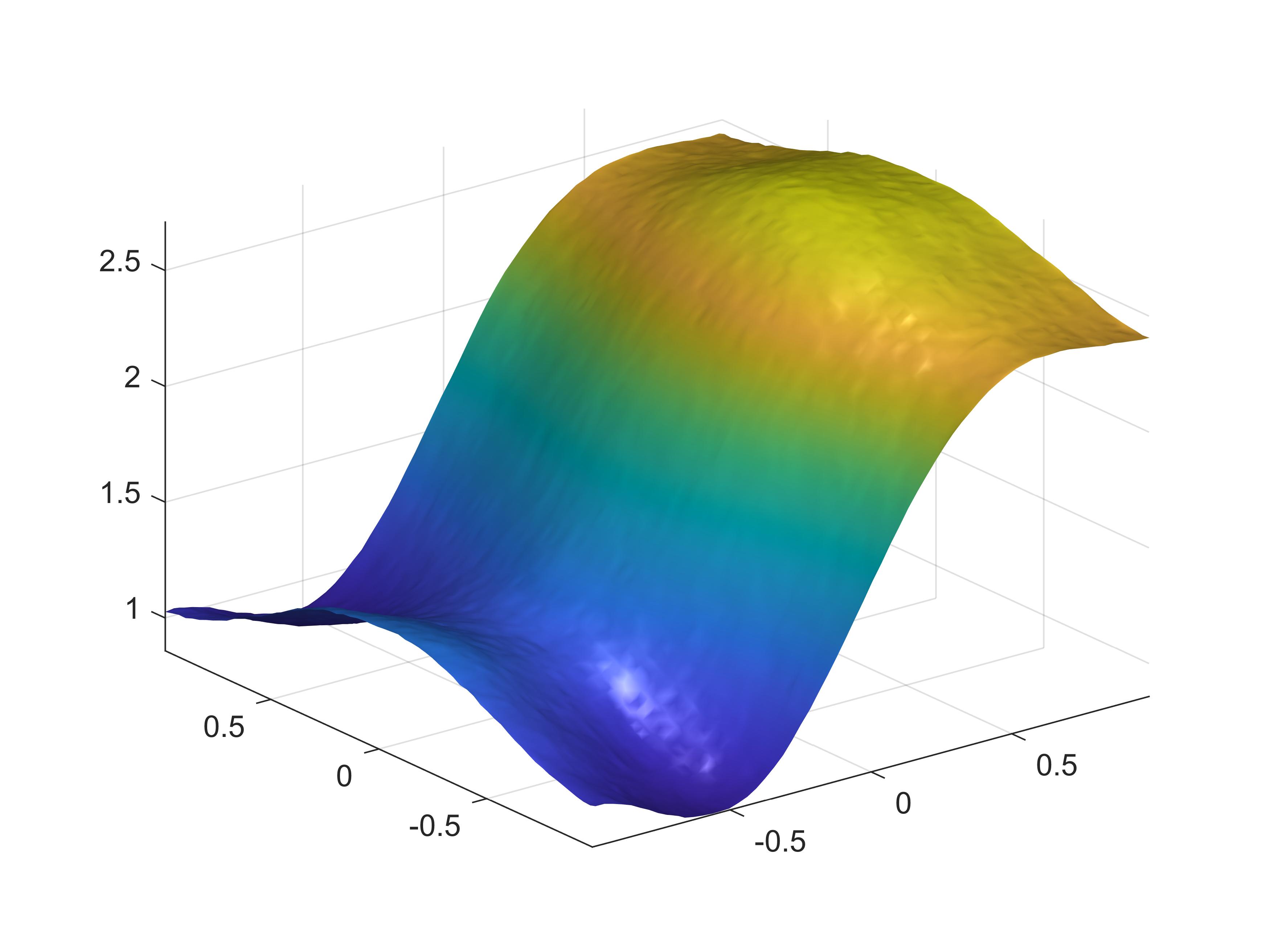} }
		\subfigure[$\theta=0.02$]{
			\includegraphics[scale = 0.05]{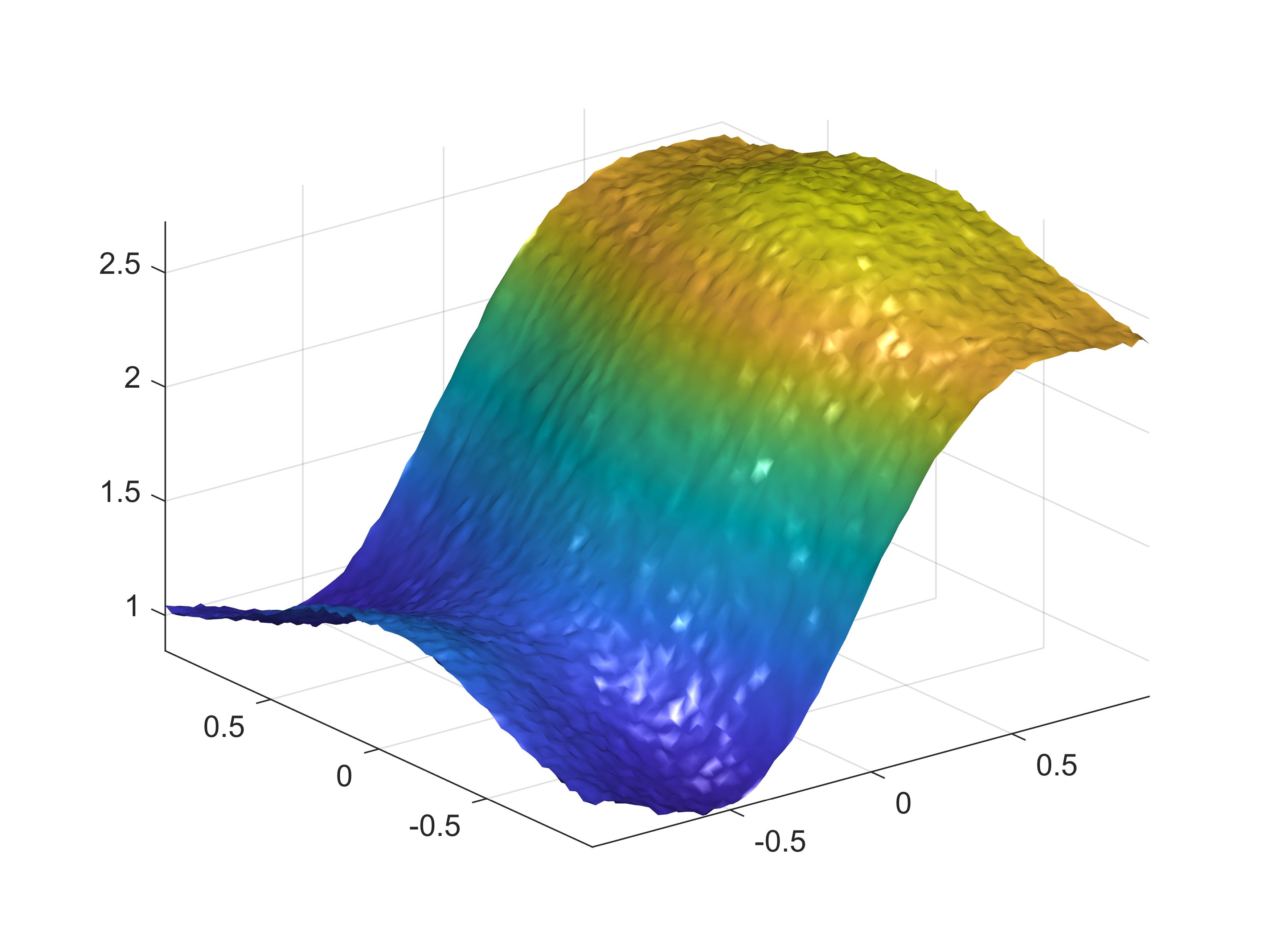} }
		\subfigure[$\theta=0.04$]{
			\includegraphics[scale = 0.05]{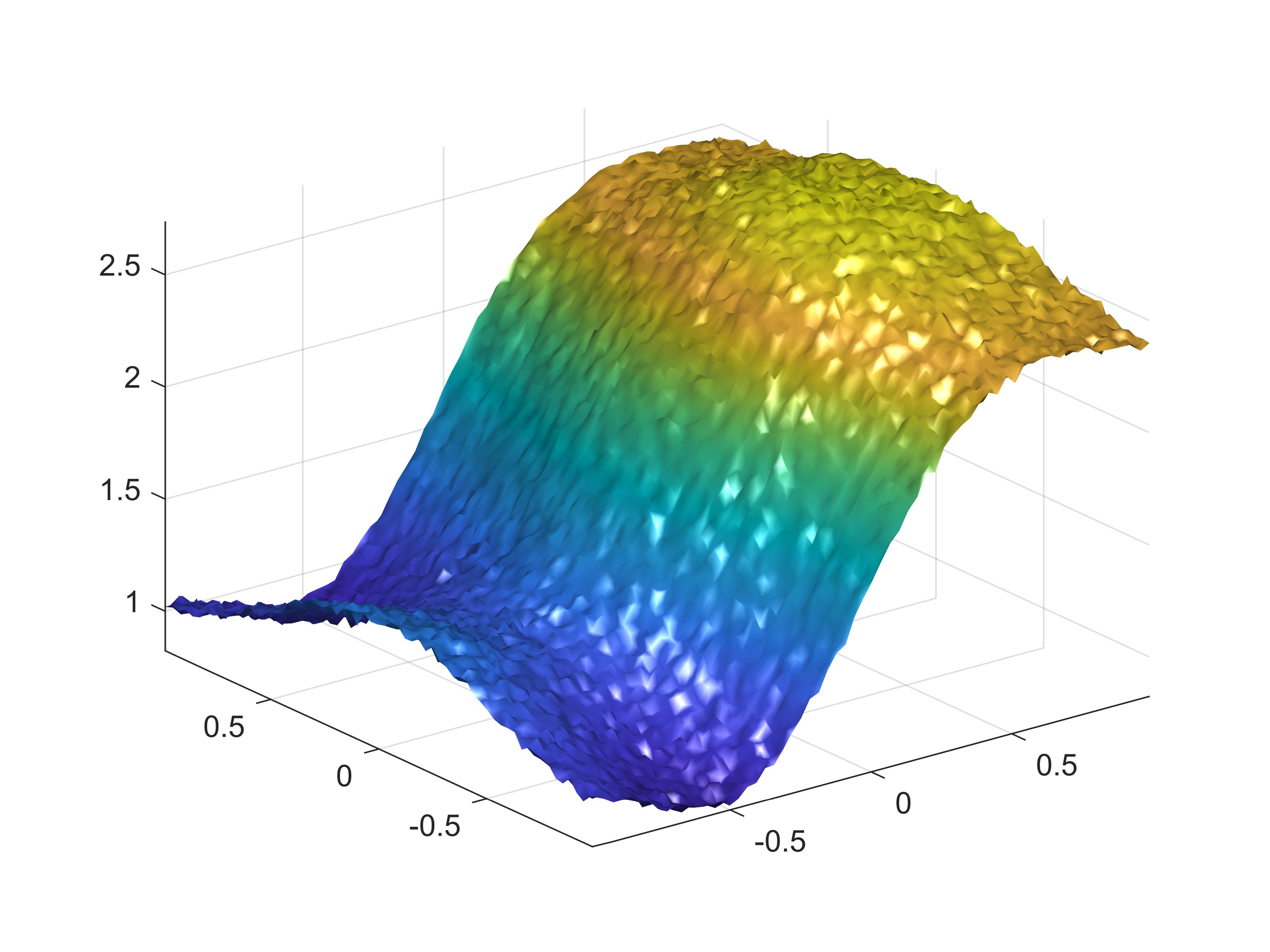} }
		\subfigure[$\theta=0.08$]{
			\includegraphics[scale = 0.05]{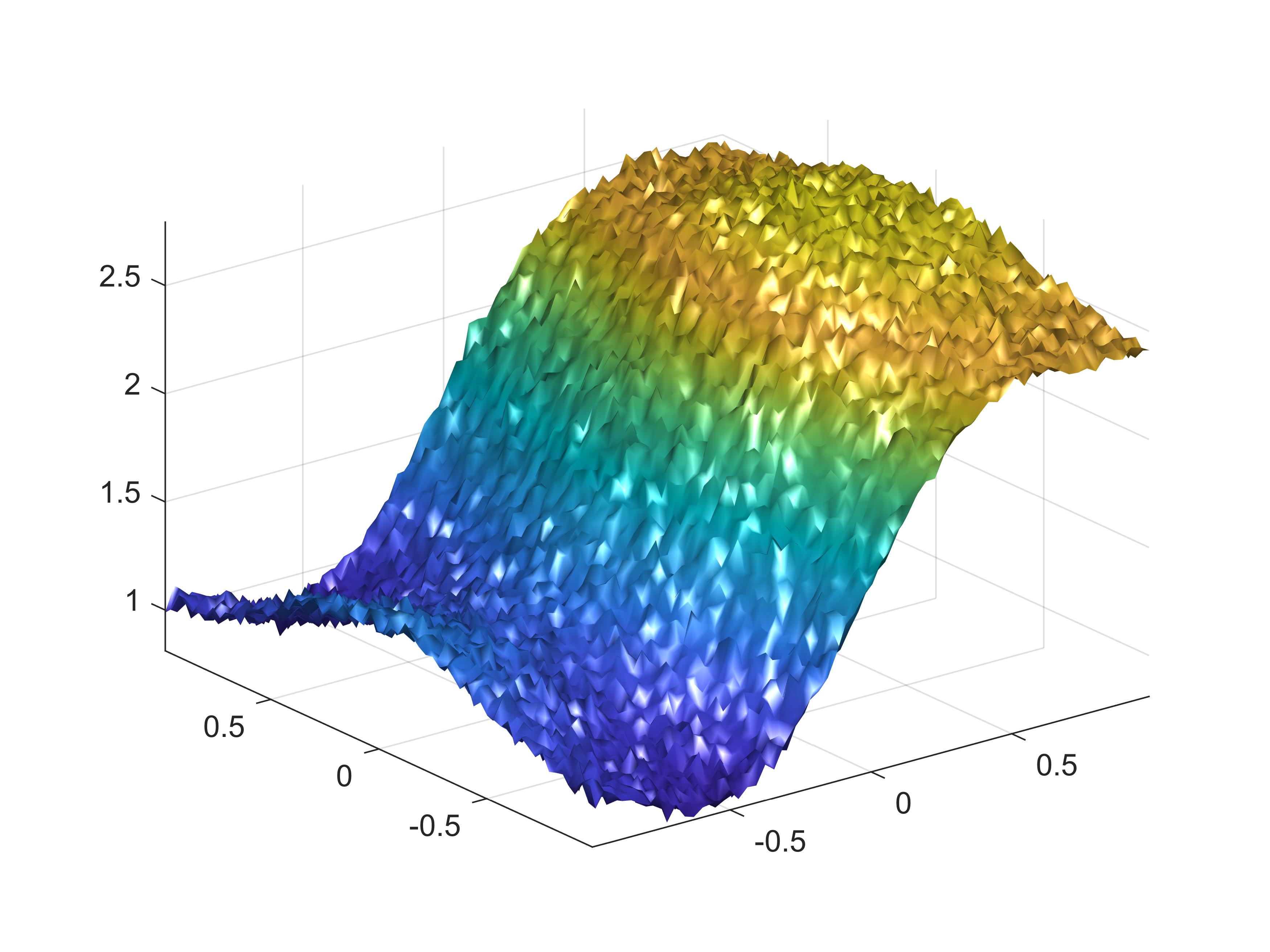} }
		\caption{Reconstruction results (Mean maps) for the 2D case, with $\delta = 1\%$ and $N=200$.} 
		\label{1Dd1N200theta}
	}
\end{figure}

\begin{figure}[htp]
	\centering
	{
		\subfigure[N=10]{
			\includegraphics[scale = 0.05]{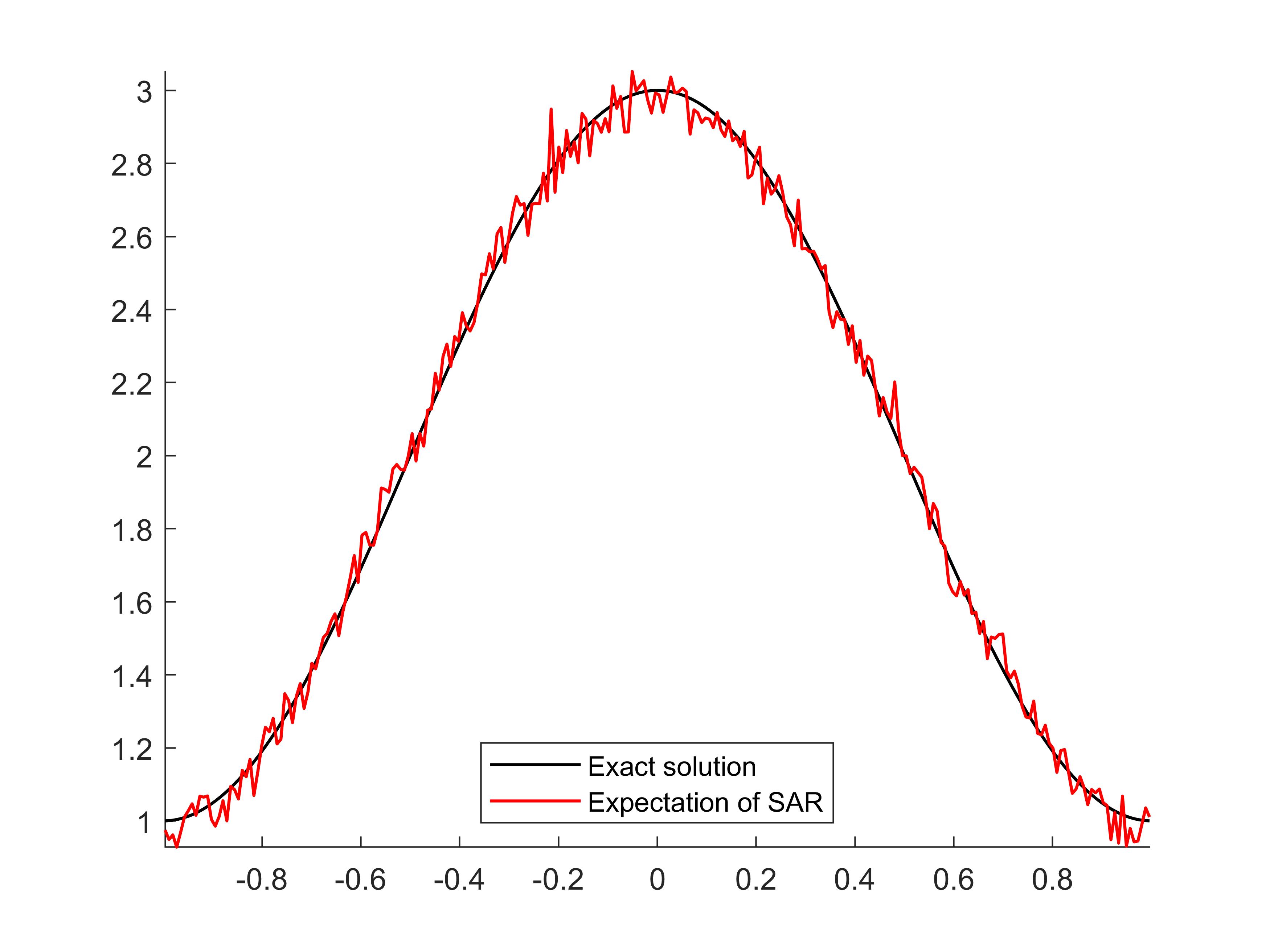} }
		\subfigure[N=50]{
			\includegraphics[scale = 0.05]{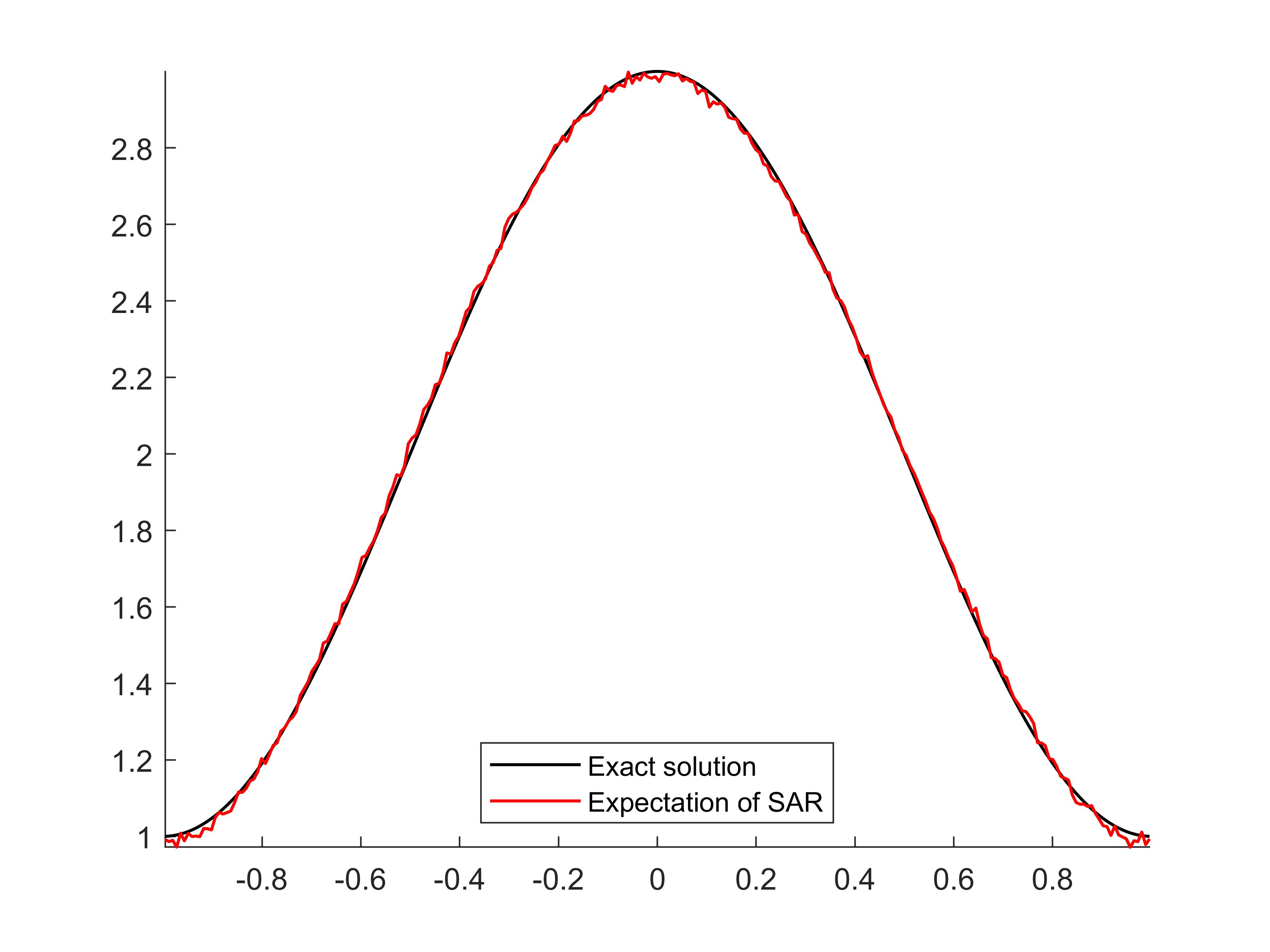} }
		\subfigure[N=100]{
			\includegraphics[scale = 0.05]{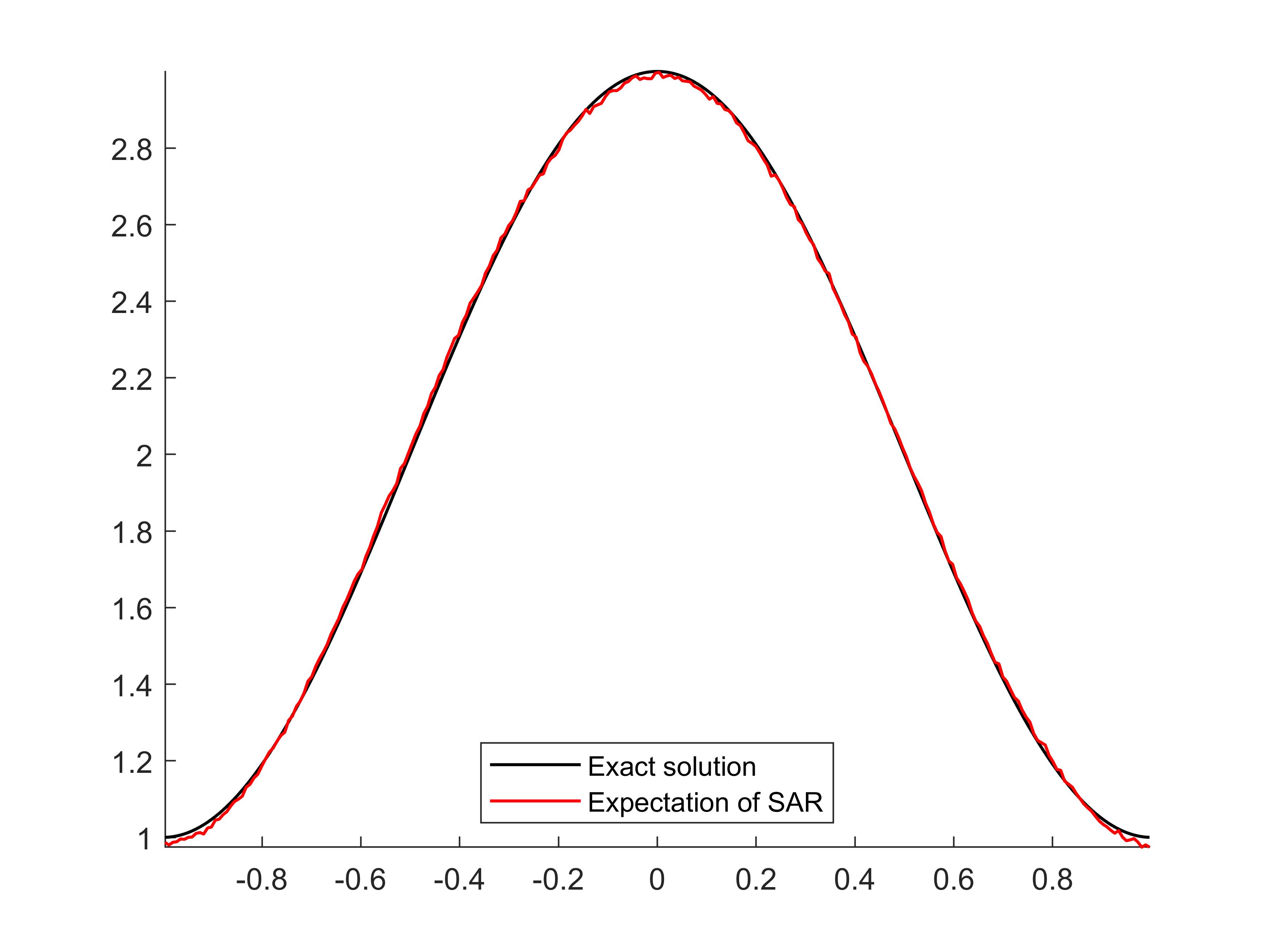} }
		\subfigure[N=500]{
			\includegraphics[scale = 0.05]{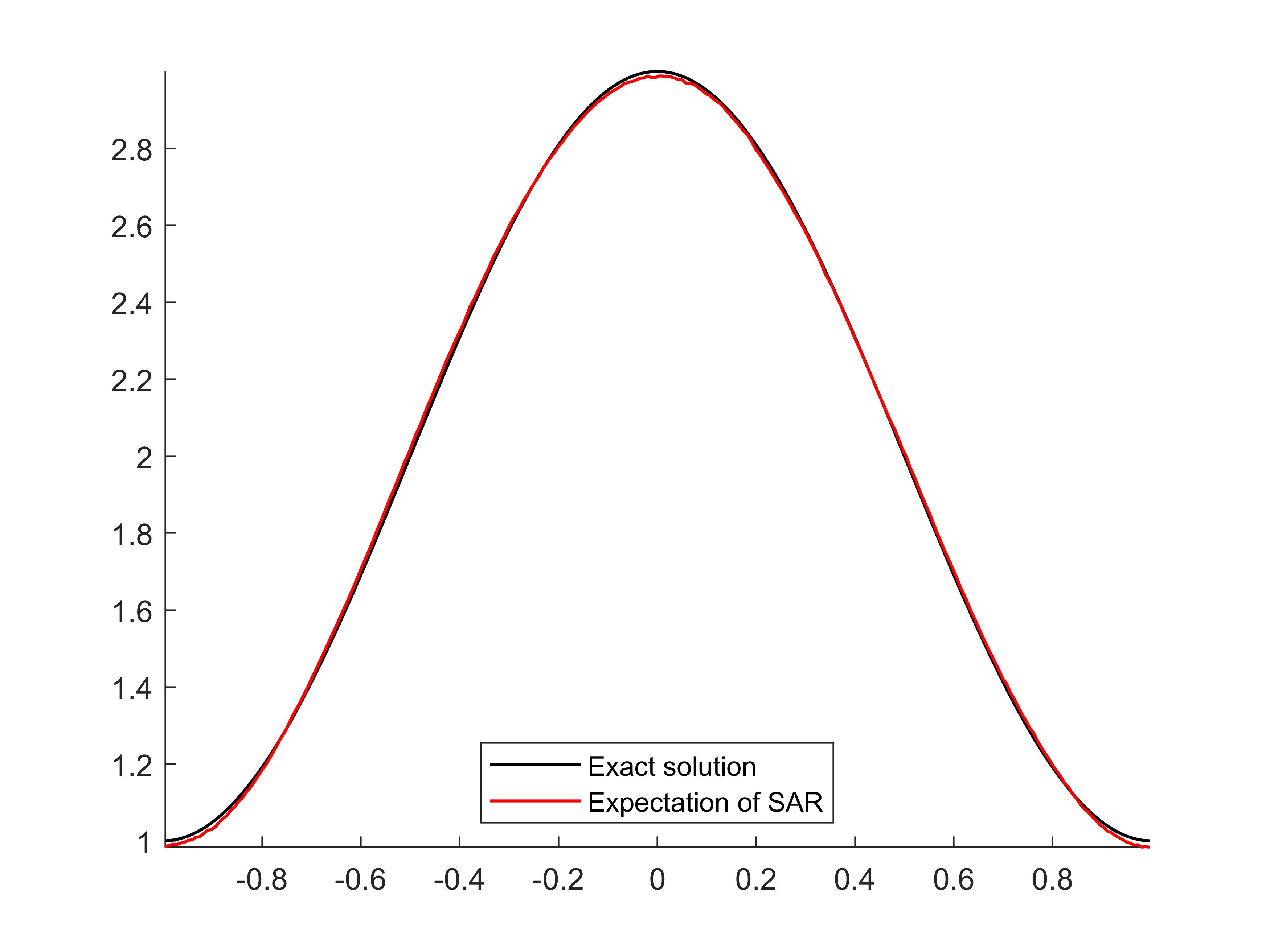} }
		\caption{The influence of expectation of approximate solutions on the sample size for the 1D case, with $\delta = 2\%$ and $\theta=0.6$.} 
		\label{1Dd2path}
	}
\end{figure}

\begin{figure}[h]
	\centering
	{
		\subfigure[N=10]{
			\includegraphics[scale = 0.05]{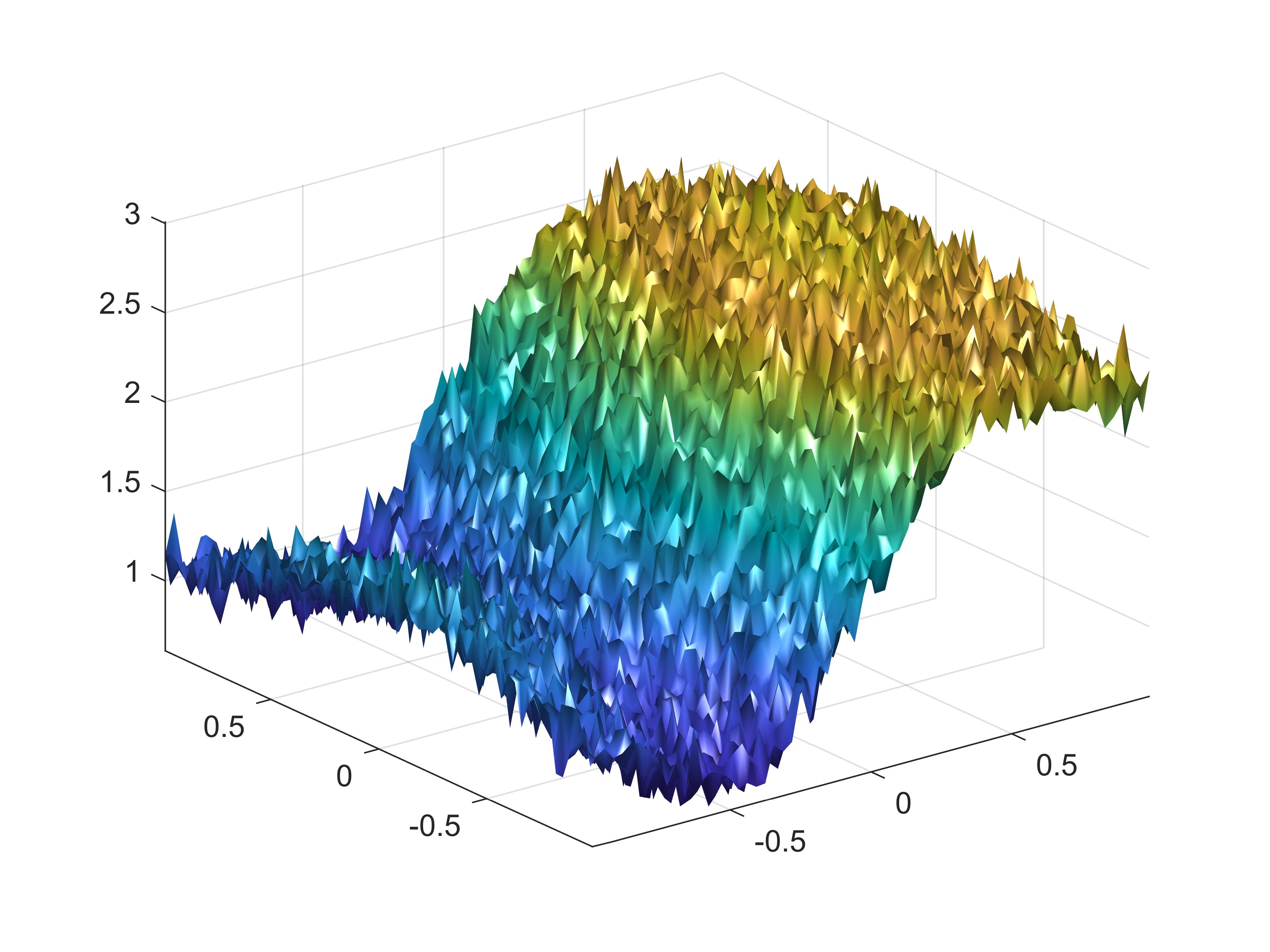} }
		\subfigure[N=50]{
			\includegraphics[scale = 0.05]{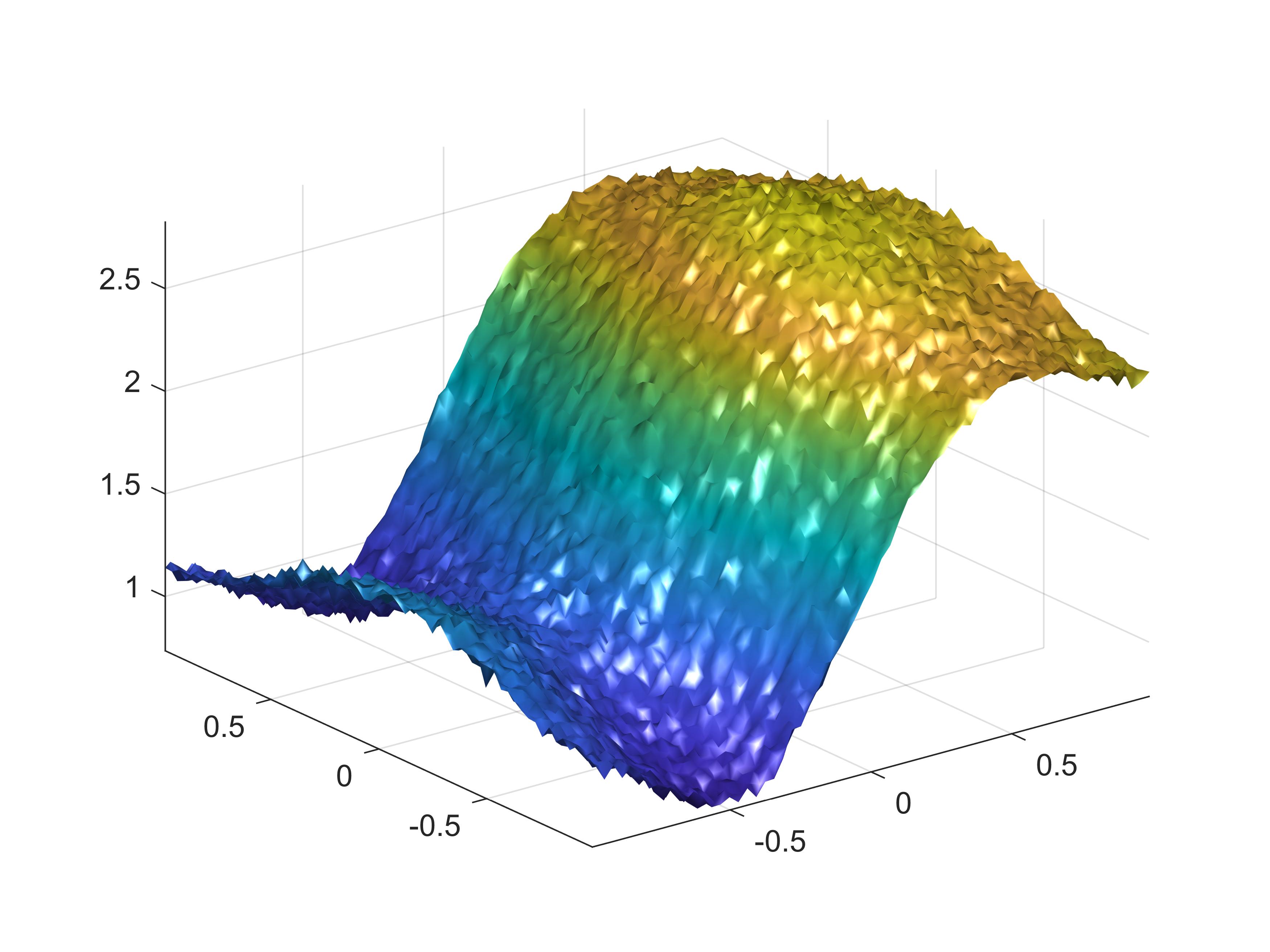} }
		\subfigure[N=150]{
			\includegraphics[scale = 0.05]{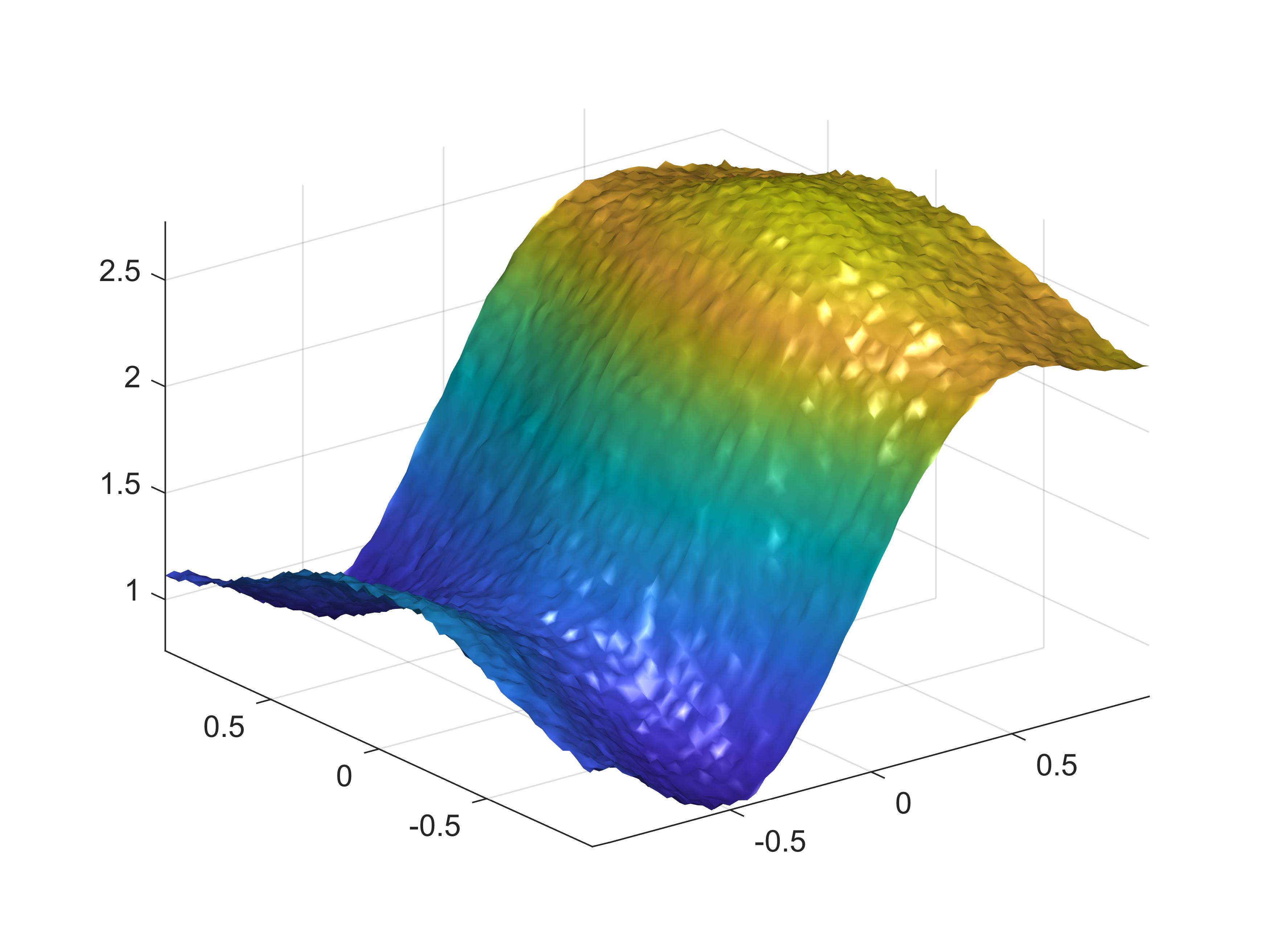} }
		\subfigure[N=400]{
			\includegraphics[scale = 0.05]{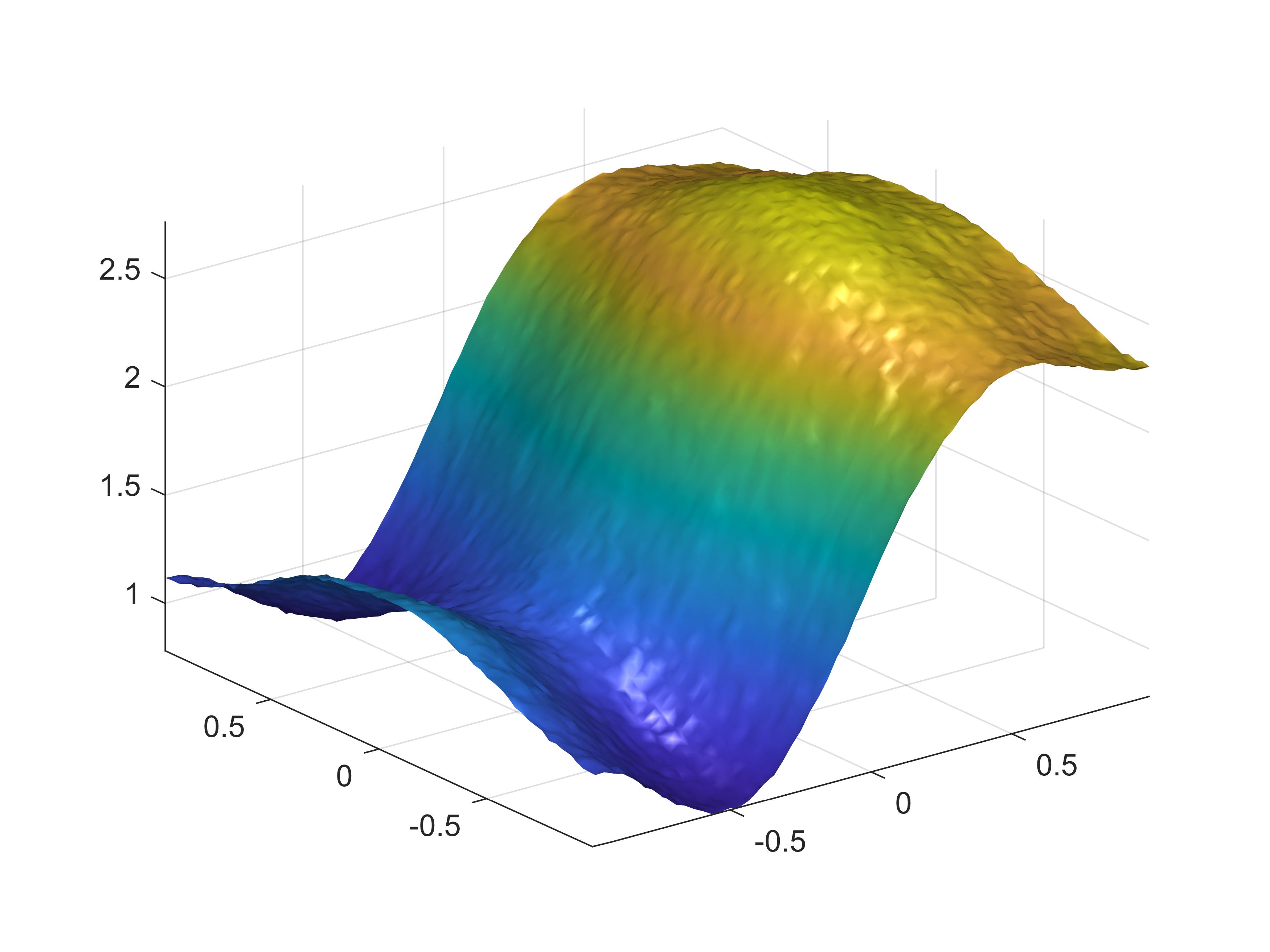} }
		\caption{The influence of expectation of approximate solutions on the sample size for the 2D case, with $\delta = 1\%$ and $\theta=0.02$.} 
		\label{2Dd1CIpath}
	}
\end{figure}


\begin{thebibliography}{10}

\bibitem{Hofmann2018}
Hofmann, B., Plato, R.: On ill-posedness concepts, stable solvability and saturation. Journal of Inverse and Ill-posed Problems, 26(2), 287-297 (2018)

\bibitem{TikhonovYagola1995}
Tikhonov, A., Goncharsky, A., Stepanov, V., Yagola, A.: Numerical Methods for the Solution of Ill-Posed Problems. Springer Science, Berlin (1995)

\bibitem{Hanke1996}
Engl, H. W., Hanke, M., Neubauer, A.: Regularization of Inverse Problems. Springer Dordrecht, (1996)


\bibitem{Cheng2011Regularization}
Cheng, J., Hofmann, B.: Regularization Methods for Ill-Posed Problems. Springer, New York (2011)

\bibitem{BK2008}
Kaltenbacher, B., Neubauer, A., Scherzer, O.: Iterative Regularization Methods for Nonlinear Ill-Posed Problems. Walter de Gruyter, Berlin (2008)

\bibitem{UT1994}
Tautenhahn, U.: On the asymptotical regularization of nonlinear ill-posed problems. Inverse Problems, 10, 1405-1418 (1994)

\bibitem{Rieder-2005}
Rieder, A.: Runge-Kutta integrators yield optimal regularization schemes. Inverse Problems, 21, 453-471 (2005)

\bibitem{ZhaoMatheLu2020}
Zhao, Y., Mathé, P., Lu, S.: Convergence analysis of asymptotical regularization and Runge-Kutta integrators for linear inverse problems under variational source conditions. CSIAM Transactions on Applied Mathematics, 1, 693-714 (2020)

\bibitem{Zhongwang2022}
Zhong, M., Wang, W., Tong, S.: An asymptotical regularization with convex constraints for inverse problems. Inverse Problems, 38(4), 045007 (2022)

\bibitem{LuNiuWerner2021}
Lu, S., Niu, P., Werner, F.: On the asymptotical regularization for linear inverse problems in presence of white noise. SIAM/ASA Journal on Uncertainty Quantification, 9, 1-28 (2021)

\bibitem{ZhangH2020}
Zhang, Y., Hofmann, B.: On the second-order asymptotical regularization of linear ill-posed inverse problems. Applicable Analysis, 99, 1000-1025 (2020)

\bibitem{ZhangHof2019}
Zhang, Y., Hofmann, B.: On fractional asymptotical regularization of linear ill-posed problems in Hilbert spaces. Fractional Calculus and Applied Analysis, 22, 699-721 (2019)

\bibitem{GonHofZhan2020}
Gong, G., Hofmann, B., Zhang, Y.: A new class of accelerated regularization methods, with application to bioluminescence tomography. Inverse Problems, 36, 055013 (2020)

\bibitem{Botetal18}
Boţ, R., Dong, G., Elbau, P., Scherzer, O.: Convergence rates of first- and higher-Order dynamics for solving linear ill-posed problems. Foundations of Computational Mathematics, 22, 1567-1629 (2022)

\bibitem{Chiang1987}
Chiang, T. S., Hwang, C. R., Sheu, S. J.: Diffusion for global optimization in $R^n$. SIAM Journal on Control and Optimization, 25, 737-753 (1987)

\bibitem{Kushner1987}
Kushner, H. J.: Asymptotic global behavior for stochastic approximation and diffusions with slowly decreasing noise effects: global minimization via Monte Carlo. SIAM Journal on Applied Mathematics, 47, 169-185 (1987)

\bibitem{Gelfand1991}
Gelfand, S. B., Mitter, S. K.: Recursive stochastic algorithms for global optimization in $R^d$. SIAM Journal on Control and Optimization, 29, 999-1018 (1991)

\bibitem{Engquist2022}
Engquist, B., Ren, K., Yang, Y.: An algebraically converging stochastic gradient descent algorithm for global optimization. arXiv:2204.05923 (2022)

\bibitem{LiTaiE2019}
Li, Q., Tai, C., E, W.: Stochastic modified equations and dynamics of stochastic gradient algorithms I: Mathematical foundations. Journal of Machine Learning Research, 20, 1-47 (2019)

\bibitem{Francis2023}
Céline, M., Adrien, T., Francis, B.: A systematic approach to Lyapunov analyses of continuous-time models in convex optimization. SIAM Journal on Optimization, 33, 1558-1586 (2023)

\bibitem{ZhangChen23}
Zhang, Y., Chen, C.: Stochastic asymptotical regularization for linear inverse problems. Inverse Problems, 39, 015007 (2023)

\bibitem{ZhangChen24}
Zhang, Y., Chen, C.: Stochastic linear regularization methods: random discrepancy principle and applications. Inverse Problems, 40, 025007 (2024)

\bibitem{Hanke1995}
Hanke, M., Neubauer, A., Scherzer, O.: A convergence analysis of the Landweber iteration for nonlinear ill-posed problems. Numerische Mathematik, 72, 21-37 (1995)

\bibitem{Rob1951}
Robbins, H., Monro, S.: A stochastic approximation method. Annals of Mathematical Statistics, 22, 400-407 (1951)

\bibitem{Kus2003}
Kushner, H. J., Yin, G. G.: Stochastic Approximation and Recursive Algorithms and Applications. 2nd ed., Springer, New York (2003)

\bibitem{Land1951}
Landweber, L.: An iteration formula for Fredholm integral equations of the first kind. American Journal of Mathematics, 73, 615-624 (1951)

\bibitem{BJin2019}
Jin, B., Lu, X.: On the regularizing property of stochastic gradient descent. Inverse Problems, 35, 015004 (2019)

\bibitem{BJin2020IP}
Jahn, T., Jin, B.: On the discrepancy principle for stochastic gradient descent. Inverse Problems, 36, 095009 (2020)

\bibitem{LuMathe2021}
Lu, S., Mathé, P.: Stochastic gradient descent for linear inverse problems in Hilbert spaces. Mathematics of Computation, 91(336), 1763--1788, (2022)

\bibitem{BJin2020SIAM}
Jin, B., Zhou, Z., Zou, J.: On the convergence of stochastic gradient descent for nonlinear ill-posed problems. SIAM Journal on Optimization, 30, 1421-1450 (2020)

\bibitem{Jin2015}
Ito, K., Jin, B.: Inverse Problems: Tikhonov Theory and Algorithms. World Scientific (2014)

\bibitem{Gaw2011}
Gawarecki, L., Mandrekar, V.: Stochastic Differential Equations in Infinite Dimensions with Applications to Stochastic Partial Differential Equations. Springer, New York (2011)

\bibitem{JentzenKloeden2009}
Jentzen, A., Kloeden, P.: Overcoming the order barrier in the numerical approximation of stochastic partial differential equations with additive space-time noise. Proceedings of the Royal Society A: Mathematical, Physical and Engineering Sciences, 465, 649-667 (2009)

\bibitem{LordPowell2014}
Lord, G., Powell, C., Shardlow, T.: An Introduction to Computational Stochastic PDEs. Cambridge Texts in Applied Mathematics. Cambridge University Press, New York (2014)

\bibitem{Baumeister1991}
Baumeister, J.: Deconvolution of appearance potential spectra. Methoden Verfahren Math. Phys., Peter Lang, Frankfurt am Main, 35, 1-13 (1991)

\bibitem{Dai2013}
Dai, Z.: Local regularization methods for inverse Volterra equations applicable to the structure of solid surfaces. Journal of Integral Equations and Applications, 25(2), 223-252 (2013)

\bibitem{Fukuda2010}
Fukuda, Y.: Appearance potential spectroscopy (APS): old method, but applicable to study of nano-structures. Analytical Sciences, 26(2), 187-197 (2010)

\bibitem{GerthHofmann2014}
Gerth, D., Hofmann, B., Birkholz, S., Koke, S., Steinmeyer, G.: Regularization of an autoconvolution problem in ultrashort laser pulse characterization. Inverse Problems in Science and Engineering, 22(2), 245-266 (2014)


\end{thebibliography}
\bibliographystyle{amsplain}

\end{document}